\documentclass[10pt]{article}
\usepackage{setspace}

\usepackage[totalwidth=14cm,totalheight=24cm]{geometry}
\usepackage[table]{xcolor}
\usepackage[utf8]{inputenc}
\usepackage{amsmath,amsthm}
\usepackage{amsfonts}
\usepackage{amssymb}
\usepackage[all,arc]{xy}
\usepackage{tikz-cd}
\usetikzlibrary{arrows}
\usepackage{tikz}
\usetikzlibrary{positioning}
\usepackage{soul}
\usepackage[title,titletoc,toc]{appendix}
\usepackage{floatflt}
\usepackage{bbm}
\usepackage{pbox}
\usepackage[labelfont=bf,justification=justified,singlelinecheck=false]{caption}
\usepackage{eqnarray}

\usepackage{epigraph}
\usepackage{thmtools}
\usepackage{thm-restate}
\usepackage{hyperref}
\usepackage{cleveref}

\swapnumbers 

\def\@oddhead{\hfill \shorttitle \hfill \thepage}
\def\@evenhead{\thepage \hfill \shortauthor \hfill}
\def\@oddfoot{}
\def\@evenfoot{}

\newtheorem{df}{Definition}[section]

\newtheorem{remark}[df]{Remark}
\newtheorem{theorem}[df]{Theorem}

\newtheorem{lemma}[df]{Lemma}

\newtheorem{question}[df]{Question}
\newtheorem{cor}[df]{Corollary}

\newcommand{\R}{\mathbb{R}}

\renewcommand{\H}{\mathbb{H}}

\newcommand{\I}{\mathrm{I}}
\newcommand{\II}{I\hspace{-0.1cm}I}
\newcommand{\III}{I\hspace{-0.1cm}I\hspace{-0.1cm}I}

\newcommand{\opdl}{\operatorname{dl}}

\newcommand{\opPSL}{\operatorname{PSL}}

\newcommand{\opLab}{\operatorname{Lab}}

\newcommand{\opCos}{{\operatorname{cos}}}
\newcommand{\opSin}{{\operatorname{sin}}}
\newcommand{\opTan}{{\operatorname{tan}}}
\newcommand{\opArcTan}{{\operatorname{arctan}}}

\def\Thol{\operatorname{T}_{\operatorname{hol}}[S]}
\def\Trep{\operatorname{T}_{\operatorname{rep}}[S]}
\def\Thyp{\operatorname{T}_{\operatorname{hyp}}[S]}

\def\opML{{\operatorname{ML}[S]}}
\def\opRT{{\operatorname{RT}[S]}}
\def\opE{\operatorname{E}}
\def\opTr{\operatorname{Tr}}
\def\opdVol{\operatorname{dVol}}
\def\opEnd{\operatorname{End}}
\def\opT{\operatorname{T}}
\def\opDet{\operatorname{Det}}
\def\opO{\operatorname{O}}
\def\opsl{{\mathfrak{sl}}}
\def\opAd{\operatorname{Ad}}
\def\opC{\operatorname{C}}
\def\opPC{\operatorname{PC}}
\def\opA{\operatorname{A}}
\def\opAdS{\operatorname{AdS}}
\def\Id{\operatorname{Id}}
\def\opI{\operatorname{I}}
\def\opFuc{\operatorname{Fuc}}
\def\opII{\operatorname{II}}
\def\opIII{\operatorname{III}}
\def\opPO{\operatorname{PO}}
\def\opRP{\operatorname{RP}}
\def\opL{\operatorname{L}}
\def\opGHMC{\operatorname{GHMC}}
\def\opdS{\operatorname{dS}}
\def\opId{\operatorname{Id}}
\def\opPO{\operatorname{PO}}
\def\opPSO{\operatorname{PSO}}
\def\opH{\operatorname{H}}
\def\opqf{{\operatorname{qf}}}
\def\opK{\operatorname{K}}

\author{Fran\c{c}ois Fillastre  and Graham Smith}
\title{Group actions and scattering problems in Teichm\"uller theory}
\begin{document}

\maketitle

\epigraph{\sl It was six men of Indostan\\
To learning much inclined,\\
Who went to see the Elephant\\
(Though all of them were blind),\\
That each by observation\\
Might satisfy his mind.}
{John Godfrey Saxe}

\tableofcontents

\section{Introduction}
In recent years, Teichm\"uller theory, which is the study of moduli spaces of marked Riemann surfaces, has come to be considered more and more from the point of view of actions of surface groups inside certain semi-simple Lie groups. In particular, we consider the case where the Lie groups in question have symmetric spaces which are lorentzian spacetimes. Indeed, this can be considered as the starting point of Mess' seminal work, which led to the development of new and strikingly simpler proofs of many results of Teichm\"uller theory by considering them in terms of geometric objects inside these symmetric spaces. Our aim is to provide a brief and straightforward introduction to this approach, whilst developing what we consider to be a useful mental framework for organising known results and open problems.
\paragraph{Group actions and Teichm\"uller theory} We first recall how Teichm\"uller theory is described in terms of group actions. Throughout this paper, $S$ will denote a compact, oriented surface of hyperbolic type, $\mathfrak{g}$ will denote its genus, and $\Pi_1$ will denote its fundamental group. Recall that the {\sl Teichm\"uller space} of $S$ can be defined as the space of marked hyperbolic metrics over $S$, modulo diffeomorphisms which preserve the marking.\footnote{Recall that a marking is defined to be an ordered set, $(a_1,...,a_\mathfrak{g},b_1,...,b_\mathfrak{g})$, of generators of $\Pi_1$ such that $[a_1,b_1][a_2,b_2]\cdots[a_\mathfrak{g},b_\mathfrak{g}]=\opId$. Two markings are identified whenever there exists an element of $\Pi_1$ which conjugates the one into the other. Distinct points of Teichm\"uller space with inequivalent markings may correspond to the same hyperbolic metric, so that Teichm\"uller space is a ramified cover of the moduli space of hyperbolic metrics. Markings were introduced into Teichm\"uller theory in order to resolve the singular points of moduli spaces, allowing Teichm\"uller space to be furnished with a smooth manifold structure (c.f.  \cite{tromba}). However, throughout much of Teichm\"uller theory, they are often used without explicit reference being made to them.} Marked hyperbolic metrics are associated to group actions of $\Pi_1$ as follows. Let $\Bbb{H}^2$ denote $2$-dimensional hyperbolic space. Recall that the identity component of its isometry group is identified with $\opPSL(2,\Bbb{R})$. Now, given a marked hyperbolic metric, $g$, over $S$, there exists a homomorphism, $\theta:\Pi_1\rightarrow\opPSL(2,\Bbb{R})$, which is injective and discrete, and which is unique up to conjugation, such that $g$ is equivalent to the quotient, $\Bbb{H}^2/\theta(\Pi_1)$. In this manner, the Teichm\"uller space of $S$ is parametrised by conjugacy classes of certain types of homomorphisms of $\Pi_1$ into $\opPSL(2,\Bbb{R})$.
\paragraph{GHMC spacetimes} We now recall the definition of ghmc spacetimes (c.f. \cite{oneill,ringstrom}), which will play a central role throughout the rest of the paper. First, a (lorentzian) spacetime is a smooth manifold, $M$, furnished with a smooth, semi-riemannian metric of signature $(n,1)$, along with spatial and temporal orientations. Lorentz\-ian spacetimes arise in a natural manner, and are rarely studied in mathematics without the assumption of some additional structure which reflects their motivations from theoretical physics. In order to understand this, consider first vectors and curves in $M$. A tangent vector, $X$, of $M$ is said to be {\sl timelike}, {\sl lightlike} or {\sl spacelike} depending on whether its Lorentz norm squared has positive, zero or negative sign respectively. A differentiable curve, $c$, in $M$ is said to be {\sl causal} whenever its tangent vector is never spacelike. The spacetime, $M$, is then said to be {\sl causal} whenever it contains no closed causal curves, reflecting the physical idea that one cannot return to the past by travelling into the future. Next, given two points, $p$ and $q$, inside $M$, $q$ is said to lie in the {\sl past} of $p$, and $p$ is said to lie in the {\sl future} of $q$, whenever there exists a future oriented, causal curve starting at $q$ and ending at $p$. A causal spacetime is then said to be {\sl globally hyperbolic} whenever the intersection of the past of any point, $p$, with the future of any other point, $q$, is compact (c.f. \cite{BS03,BS07}).
\par
Global hyperbolicity is equivalent to the existence in $M$ of a {\sl Cauchy hypersurface}, which is a smooth, embedded, spacelike hypersurface, $N$, which intersects every inextensible causal curve in $M$ once and only once. Although the Cauchy hypersurfaces in $M$ are trivially not unique, they are all diffeomorphic, and, given a Cauchy hypersurface, $N$, the spacetime itself is diffeomorphic to $N\times\Bbb{R}$.
\par
A globally hyperbolic spacetime is said to be {\sl maximal} whenever it cannot be embedded isometrically into a larger spacetime of the same dimension in such a manner that the image of a Cauchy hypersurface is still a Cauchy hypersurface. To illustrate this, consider the {\sl Minkowski space}, $\Bbb{R}^{d,1}$, which is defined to be the space of all real $(d+1)$-tuplets, $x:=(x_1,...,x_{d+1})$, furnished with the metric
\begin{equation*}
\langle x,y\rangle_{d,1} := -x_1y_1- \cdots -x_dy_d+x_{d+1}y_{d+1}~.
\end{equation*}
This space is trivially a maximal, globally hyperbolic spacetime, with Cauchy hypersurface given by the horizontal hyperplane, $\Bbb{R}^d\times\left\{0\right\}$. Consider now its future oriented light cone,
\begin{equation*}
C^+ := \left\{x\ |\ \langle x,x\rangle_{d,1} =0,\ x_{d+1}\geq 0\right\}~.
\end{equation*}
The interior, $\operatorname{Int}(C^+)$, of this cone is also a maximal, globally hyperbolic spacetime, with Cauchy hypersurface given by
\begin{equation*}
H^+ := \left\{x\ |\ \langle x,x\rangle_{d,1}=1,\ x_{d+1}>0\right\}~.
\end{equation*}
However, the fact that $\operatorname{Int}(C^+)$ embeds isometrically into $\Bbb{R}^{d,1}$, does not contradict maximality, since $H^+$ is not a Cauchy hypersurface of the latter spacetime.
\par
Finally, a globally hyperbolic spacetime is said to be {\sl Cauchy compact} whenever its Cauchy hypersurface is compact. Spacetimes possessing all these properties are described as {\sl ghmc} (Globally Hyperbolic, Maximal, Cauchy Compact), and it is these that will be studied in this paper.
\paragraph{GHMC Minkowski spacetimes} We henceforth focus on $(2+1)$-dimensional ghmc spacetimes of constant sectional curvature. In order to form a good intuition for what follows, it is worth reviewing first in some detail the case of flat spacetimes, which we henceforth refer to as {\sl Minkowski} spacetimes. The simplest maximal, $(2+1)$-dimensional, globally hyperbolic Minkowski spacetime is, of course, the Minkowski space, $\Bbb{R}^{2,1}$, itself. More generally, consider a convex, open subset, $\Omega$, of $\Bbb{R}^{2,1}$, formed by taking the intersection of the future sides of all elements of some family of null planes, no two of which are parallel. Every such $\Omega$ is a maximal, globally hyperbolic, Minkowski spacetime with a Cauchy surface which is complete with respect to the induced metric. Up to reflection in the horizontal plane, $\Bbb{R}^2\times\left\{0\right\}$, all maximal, globally hyperbolic, Minkowski spacetimes with complete Cauchy surfaces arise in this manner (c.f. \cite{Bar05}).
\par
Consider now the compact case, and let $\opGHMC_0:=\opGHMC_0(S)$ denote the space of marked ghmc Minkowski spacetimes with Cauchy surface diffeomorphic to $S$. Elements of $\opGHMC_0$ are constructed by taking suitable quotients of convex, open sets of the type constructed in the preceding paragraph. To see this, observe first that the identity component of the isometry group of $\Bbb{R}^{2,1}$ is given by $\opPSL(2,\Bbb{R})\ltimes\Bbb{R}^{2,1}$, where the first component acts linearly, and the second by translation. Consider now a homomorphism, $\theta:\Pi_1\rightarrow\opPSL(2,\Bbb{R})\ltimes\Bbb{R}^{2,1}$. When the linear component, $\theta_0:\Pi_1\rightarrow\opPSL(2,\Bbb{R})$, of $\theta$ is injective and discrete, there exists a unique open subset, $\Omega$, of $\Bbb{R}^{2,1}$, of the type described above, over which $\theta(\Pi_1)$ acts properly discontinuously (c.f. Theorem~\ref{mess mink 23}). The quotient of $\Omega$ by this action is then an element of $\opGHMC_0$. Furthermore, up to reversal of the time orientation, this accounts for all ghmc Minkowski spacetimes, so that $\opGHMC_0$ is actually parametrised by the set of all homomorphisms, $\theta:\Pi_1\rightarrow\opPSL(2,\Bbb{R})\ltimes\Bbb{R}^{2,1}$, whose linear component is injective and discrete (c.f. Theorem~\ref{thm:mess bij mink}). In particular, this parametrisation furnishes the space, $\opGHMC_0$, with the structure of a $(12\mathfrak{g}-12)$-dimensional real algebraic variety.
\par
In particular, this parametrisation of $\opGHMC_0$ shows that ghmc Minkowski spacetimes naturally identify with tangent vectors over Teichm\"uller space. Indeed, by considering the derivatives of smooth families of homomorphisms of $\Pi_1$ into $\opPSL(2,\Bbb{R})$, we see that every homomorphism, $\theta:\Pi_1\rightarrow\opPSL(2,\Bbb{R})\ltimes\Bbb{R}^{2,1}$, whose linear component, $\theta_0$, is injective and discrete, identifies with a unique tangent vector of Teichm\"uller space\footnote{In fact, the derivative of a smooth family of homomorphisms of $\Pi_1$ into $\opPSL(2,\Bbb{R})$ is given by a homomorphism of $\Pi_1$ into $\opPSL(2,\Bbb{R})\ltimes\opsl(2,\Bbb{R})$, where, $\opsl(2,\Bbb{R})$ is the Lie algebra of $\opPSL(2,\Bbb{R})$, over which $\opPSL(2,\Bbb{R})$ acts by the adjoint action. It is therefore homomorphisms into this group which naturally identify with tangent vectors over Teichm\"uller space. However, since $\opsl(2,\Bbb{R})$, together with its Killing form, naturally identifies with $\Bbb{R}^{2,1}$, this twisted product is really the same as the twisted product, $\opPSL(2,\Bbb{R})\ltimes\Bbb{R}^{2,1}$, justifying the identification given above.} over the point determined by $\theta_0$. The significance of this correspondence will become apparent presently.
\paragraph{GHMC anti de Sitter and de Sitter spacetimes}We now consider spacetimes of constant sectional curvature equal to $-1$, which will henceforth be referred to as {\sl anti de Sitter (AdS)} spacetimes. Let $\opGHMC_{-1}:=\opGHMC_{-1}(S)$ denote the space of marked ghmc anti de Sitter spacetimes with Cauchy surface diffeomorphic to $S$. In order to understand the structure of this space, consider first $(2+1)$-dimensional anti de Sitter space, which we denote by $\opAdS^{3}$ (c.f. Section~\ref{sec:ads}). The identity component of the isometry group of this space is given by $\opPSL(2,\Bbb{R})\times\opPSL(2,\Bbb{R})$. Consider now a homomorphism, $\theta:=(\theta_l,\theta_r):\Pi_1\rightarrow\opPSL(2,\Bbb{R})\times\opPSL(2,\Bbb{R})$. When each of the components, $\theta_l$ and $\theta_r$, are injective and discrete, there exists a unique open subset, $\Omega$, of $\opAdS^3$, having a structure analogous to that described above for the Minkowski case, over which $\theta(\Pi_1)$ acts properly discontinuously (c.f. Theorem~\ref{messAdS2}). The quotient of $\Omega$ by
this action is then an element of $\opGHMC_{-1}$. Furthermore, this accounts for all ghmc AdS spacetimes, so that $\opGHMC_{-1}$ is parametrised by the set of homomorphisms, $\theta:\Pi_1\rightarrow\opPSL(2,\Bbb{R})\times\opPSL(2,\Bbb{R})$, both of whose components are injective and discrete. In particular, this parametrisation also furnishes $\opGHMC_{-1}$ with the structure of a $(12\mathfrak{g}-12)$-dimensional real algebraic variety.
\par
In a similar manner to the Minkowski case, ghmc anti de Sitter spacetimes naturally identify with pairs of points in Teichm\"uller space. Indeed, every homomorphism, $\theta:\Pi_1\rightarrow\opPSL(2,\Bbb{R})\times\opPSL(2,\Bbb{R})$, each of whose two components are injective and discrete, naturally identifies with a unique pair of points in Teichm\"uller space. This correspondence, together with its Minkowski analogue mentioned above, will play a key role in what follows.
\par
For completeness, we consider also spacetimes of constant sectional curvature equal to $1$, which will henceforth be referred to {\sl de Sitter (dS)} spacetimes. Let $\opGHMC_1:=\opGHMC_1(S)$ denote the space of marked ghmc de Sitter spacetimes with Cauchy surface diffeomorphic to $S$. The global structure of $\opGHMC_1$ is more complicated than in the Minkowski and anti de Sitter cases, since there exist ghmc de Sitter spacetimes which are not quotients of invariant open subsets of de Sitter space. Nonetheless, by restricting attention to what are known as ``quasi-Fuchsian'' de Sitter spacetimes, whose moduli space we denote by $\opGHMC_1^\opqf$, we recover a theory similar to that developed in the Minkowski and anti de Sitter cases, with corresponding applications to Teichm\"uller theory.
\paragraph{Teichm\"uller theory as Lorentzian geometry} The correspondences between Teich\-m\"uller space and $\opGHMC_0$ and $\opGHMC_{-1}$ indicated above lead to the following phenomena. Firstly, many constructions of classical Teichm\"uller theory involving tangent vectors of Teichm\"uller space find their counterparts in terms of geometric objects inside ghmc Minkowski spacetimes. Likewise, many constructions involving pairs of points in Teichm\"uller space find their counterparts in terms of geometric objects inside ghmc AdS spacetimes. For example, one can associate to every tangent vector over a given point of Teichm\"uller space, a unique measured geodesic lamination \cite{Ker85}. In Section~\ref{MinkowskiSpace}, it will be shown that this measured geodesic lamination is naturally constructed in terms of the past boundary of the corresponding ghmc Minkowski spacetime. Likewise, consider the earthquake theorem (Theorem~\ref{EarthquakeTheorem}), which says that given any two given points, $P_1$ and $P_2$, of Teichm\"uller space, there exists a unique left earthquake, $\mathcal{E}^l$, and a unique right earthquake, $\mathcal{E}^r$, such that $\mathcal{E}^l(P_1)=P_2$ and $\mathcal{E}^r(P_1)=P_2$. In Section~\ref{sec:ads}, it will be shown that these earthquakes correspond to the measured geodesic laminations of the two boundary components of the convex core of the corresponding ghmc AdS spacetime. Alternatively, consider the theorem of Schoen and Labourie (Theorem~\ref{MinimalLagrangianDiffeomorphisms}), which says that given any two points, $P_1$ and $P_2$, of Teichm\"uller space, there exists a unique minimal lagrangian diffeomorphism  sending $P_1$ into $P_2$. In Section~\ref{sec:ads}, it will also be shown that this map corresponds to the unique maximal surface in the fundamental class of the corresponding ghmc AdS spacetime, and so on.
\par
Indeed, we will see that ghmc Minkowski, anti de Sitter and de Sitter spacetimes possess a wealth of geometric structures which reflects the diversity of structures constructed over Teichm\"uller space. It is one of our main objectives to provide --- in Sections \ref{MinkowskiSpace},  \ref{sec:ads}
and \ref{DeSitterSpaceRiemannSphereHyperbolicSpace} --- an elementary overview of the key aspects of this geometry, and then to show how various, disparate elements of modern Teichm\"uller theory --- reviewed in Section~\ref{PathsAndMidPointTheorems} --- are elegantly unified within this framework.
\paragraph{Scattering problems and ghmc spacetimes}
By considering constant curvature ghmc spacetimes as flows inside Teichm\"uller space, we obtain a good mental framework for organising open problems and known results. Indeed, consider first a marked ghmc Minkowski spacetime, $M$. By Theorem~\ref{BBZmink}, there exists a unique foliation, $(\Sigma_{\kappa})_{\kappa\in]0,\infty[}$, of $M$ such that, for each $\kappa$, the leaf, $\Sigma_\kappa$, is a smoothly embedded, locally strictly convex, spacelike surface of constant extrinsic curvature equal to $\kappa$. Now, if, for each $\kappa$, we denote by $g_\kappa$ the intrinsic metric of the $\kappa$'th leaf, then the rescaled family, $(\kappa g_\kappa)_{\kappa\in]0,\infty[}$, consists only of hyperbolic metrics, and therefore defines a smooth curve inside Teichm\"uller space. It is then natural to ask, for given pairs, $0<\kappa_1<\kappa_2$, of positive real numbers, and, $g_1$ and $g_2$, of marked hyperbolic metrics, whether there exists a ghmc Minkowski spacetime interpolating between $(g_1,\kappa_1)$ and $(g_2,\kappa_2)$, in the sense that, for each $i$, the intrinsic metric of its $\kappa_i$'th leaf corresponds to the point, $g_i$, of Teichm\"uller space.
\par
Similar problems can be studied by taking limits as $\kappa_1$ and $\kappa_2$ converge to $0$, to infinity, or even to each other. Furthermore, other curves in Teichm\"uller space are obtained by considering anti de Sitter and de Sitter spacetimes, as well as other geometric invariants of the leaves, such as their second or third fundamental forms, and so on. Overall, we obtain a large family of maps, all sending $\opGHMC_0$, $\opGHMC_{-1}$ and $\opGHMC_1^\opqf$ into spaces of Teichm\"uller data (such as Teichm\"uller space itself, its cotangent bundle, the space of measured geodesic laminations, and so on). Suitable combinations of these maps then yield other maps which take values inside spaces of real dimension $(12\mathfrak{g}-12)$. It is then natural to ask whether these maps parametrise $\opGHMC_0$, $\opGHMC_{-1}$ and $\opGHMC_1^\opqf$. More precisely, we ask whether they are injective or surjective, and when they are bijective, what their regularity might be with respect to other parametrisations.
\par
As we will see in Sections \ref{MinkowskiSpace},  \ref{sec:ads}
and \ref{DeSitterSpaceRiemannSphereHyperbolicSpace}, these problems are merely reformulations of more classical problems of Teichm\"uller theory concerning the relationship between different parametrisations. However, this scattering perspective presents a useful tool for visualising and organising these results. Indeed, a good portion of this part of Teichm\"uller theory simply amounts to adding new spokes to ``wheels'' such as that illustrated for the Minkowski case in Figure~\ref{fig:wheel}, where, here, the centre represents the space, $\opGHMC_0$, and each spoke represents a different parametrisation.
\begin{figure}[t]
\begin{center}
\includegraphics[scale=0.7]{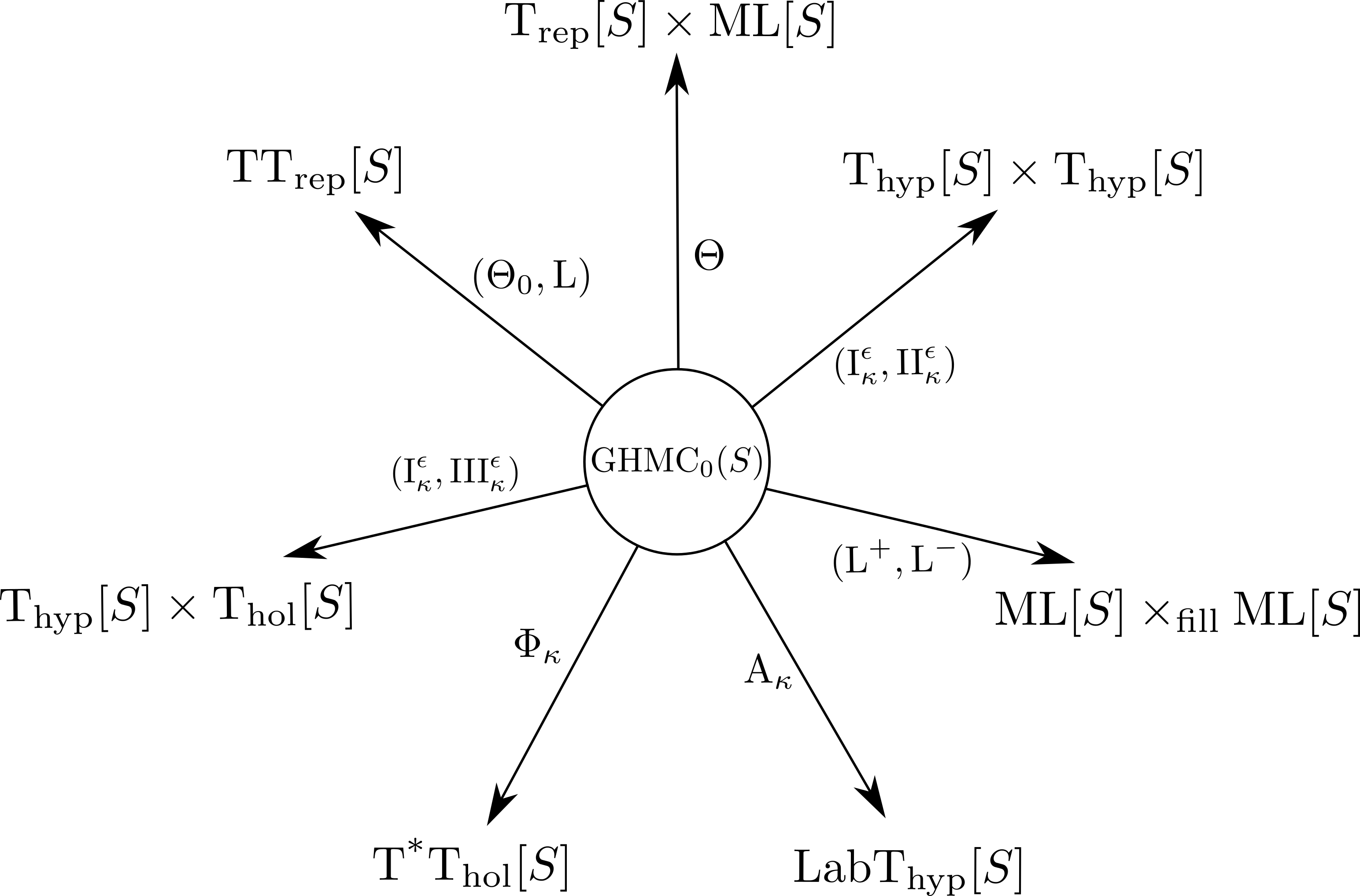}\caption{\textbf{A selection of parametrisations of $\opGHMC_0$}. Each of the spokes of the wheel represents one of the parametrisations introduced in Section~\ref{MinkowskiSpace}.}\label{fig:wheel}
\end{center}
\end{figure}
Similarly, as in Figure~\ref{fig:commutation}, we can then consider curves joining different spokes, which correspond to compositions of different parametrisations, where it is then natural to ask what the properties of these compositions might be.
\begin{figure}
\begin{center}
\includegraphics[scale=0.7]{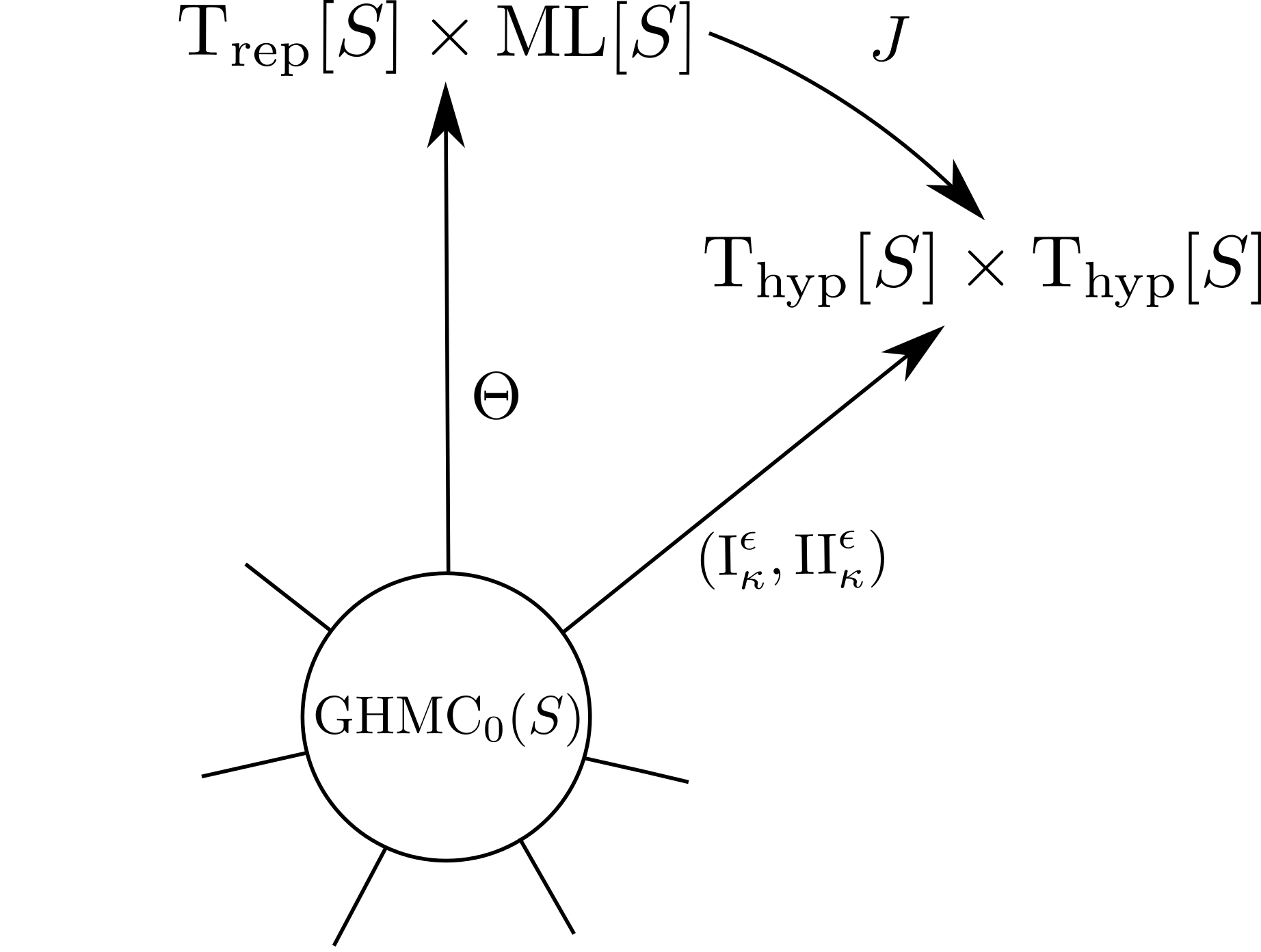}\caption{\textbf{Composing different spokes}. The parametrisations used here are discussed in detail in Section~\ref{MinkowskiSpace}. Here the map $J$ is given by the almost complex structure of $\operatorname{TT}_{\operatorname{rep}}[S]$, and the diagram commutes (c.f. \cite{BS12}).}\label{fig:commutation}
\end{center}
\end{figure}
By approaching the theory in this manner --- that is, by thinking in terms of scattering problems and parametrisations of ghmc lorentz spacetimes --- we obtain, not only new results, but also simpler proofs of known results. Various cases will be discussed in Sections~\ref{MinkowskiSpace}, \ref{sec:ads} and \ref{DeSitterSpaceRiemannSphereHyperbolicSpace}. Nevertheless, at the time of writing, we still find ourselves presented with a great number of questions, many of which remain unresolved, or, at the very least, unpublished.
\paragraph{Other perspectives}
Finally, we refer the reader to the reviews \cite{mes+,ads+,guer,bar} for other approaches to this theory. Furthermore, in \cite{KS07,MS09,LS14,BS09,BBS1,BBS2,BS12,tou1,BS,Bru}  an analogous theory is developed for spacetimes with certain types of singularities which are considered as world lines of particles; in \cite{BB2,BB1,Dan13,CDW} a theory is developed to study the degeneration of de Sitter and anti de Sitter spacetimes to Minkowski space; in \cite{SS}, compositions of different parametrisations are used to develop various dualities between $\opGHMC_0$, $\opGHMC_{-1}$ and $\opGHMC_1^{\opqf}$; and in  \cite{CR} results of a similar nature are developed using riemannian ambient spaces. Indeed, in recent years, a rich and active field has blossomed out of Mess' work, but we shall not discuss these other ideas further in the current paper.
\par
We are grateful to Athanase Papadopoulos for having invited us to contribute to the present volume of the ``Handbook of group actions''. We are also grateful to Alexis Gilles for helpful comments on later drafts of this paper, and to Cyril Lecuire and to Giongling Li for fruitful conversations about the subject. The first author would also like to thank the Federal University of Rio de Janeiro (UFRJ) for their hospitality during the preparation of this work.
%
%
\section{Paths and mid point theorems}\label{PathsAndMidPointTheorems}
\subsection{General}\label{General}
Let $S$ be a smooth, compact, oriented surface of hyperbolic type, let $\mathfrak{g}$ denote its genus, and let $\Pi_1$ denote its fundamental group. Let $\Thol$, $\Thyp$ and $\Trep$ denote respectively its Teichm\"uller spaces of marked holomorphic structures, marked hyperbolic metrics and properly discontinuous injective homomorphisms of $\Pi_1$ into $\opPSL(2,\R)$. These spaces are naturally endowed with the respective structures of a complex $(3\mathfrak{g}-3)$-dimensional manifold, a real analytic $(6\mathfrak{g}-6)$-dimensional manifold, and a real $(6\mathfrak{g}-6)$-dimensional algebraic variety. Significantly, although these spaces are naturally diffeomorphic, it is nonetheless often preferable to maintain the distinction, as properties that are natural for one are not necessarily natural for the others (c.f. the introduction to \cite{Ker85} for a deeper discussion of this subtle point).
\par
We first consider paths between different points of these spaces which arise from natural geometric constructions over Teichm\"uller space. It turns out that the existence of such paths can often be neatly expressed in terms of what we will call ``mid point'' theorems that will be introduced presently. We will find that these ``mid point'' theorems also arise naturally in the study of moduli spaces of ghmc Minkowski, anti de Sitter and de Sitter spacetimes, and it is the striking interplay between these two frameworks that forms the basis of this paper.
\subsection{Measured geodesic laminations}\label{MeasuredGeodesicLaminations}
The geometric constructs that will be studied in the sequel fall broadly into two classes, namely those that are real analytic, and those that are not. The non-analytic constructs will be defined using measured geodesic laminations, and for this reason, we consider it worthwhile to review this concept in a fair amount of detail (c.f. \cite{PH} for a thorough treatment).
\par
Given a hyperbolic metric, $g$, over $S$, a {\sl geodesic lamination} for $g$ is a closed subset, $L$, of $S$ that consists of a union of complete, simple, pairwise disjoint geodesics. Elementary properties of the geometry of hyperbolic metrics show that $L$ has measure zero in $S$, and that its complement consists of at most $(4\mathfrak{g}-4)$ connected components.
\par
Observe that, for a given geodesic lamination, $L$, the unit speed geodesic flow of $S$ defines everywhere locally a flow over $L$ that is unique up to a choice of direction. A {\sl measured geodesic lamination} is then defined to be a Radon measure, $\lambda$, over $S$, which is supported on some geodesic lamination, $L$, and which is everywhere locally invariant with respect to the geodesic flow of $L$. If $\lambda$ is a measured geodesic lamination, then so too is $a\lambda$ for all non-negative real $a$. Likewise, if $\lambda'$ is another measured geodesic lamination, then the sum, $\lambda+\lambda'$, is also a measured geodesic lamination if and only if the union, $L\cup L'$, of their respective underlying geodesic laminations is itself also a geodesic lamination. We therefore see that the space of measured geodesic laminations is a closed, piecewise linear subspace of the space of Radon measures over $S$, where each linear component consists of those measured geodesic laminations that are supported over a given fixed geodesic lamination (c.f. \cite{Bon01} or \cite{DW08} for a discussion of piecewise linearity). This space is denoted by $\opML_g$, and is furnished with the topology of weak convergence for Radon measures.
\par
In order to understand constructions involving measured geodesic laminations, it is standard practice to work with {\sl rational laminations}. These are measured geodesic laminations whose underlying geodesic lamination is a disjoint union of finitely many simple closed curves.\footnote{Non-rational laminations contain complete, non-compact leaves in their support sets. In order to visualise this case, it is fundamental to understand that the intersection of any non-compact leaf with any transverse curve will contain no isolated points (c.f. \cite{Bon01} for a good discussion with figures).} The simplest rational measured geodesic laminations are {\sl weighted geodesics}. These are pairs, $\lambda:=(c,a)$, where $c$ is a simple, closed geodesic, and $a$ is a non-negative real number, and the corresponding Radon measure is $a\opdl$, where $\opdl$ here denotes the $1$-dimensional Hausdorff measure of $c$. Every rational measured geodesic lamination is trivially a sum of finitely many measured geodesics, whilst the rational measured geodesic laminations themselves form a dense subset of $\opML_g$. For this reason, many constructions are first described explicitly for measured geodesics or for rational measured geodesic laminations, and are then extended by continuity to the whole of $\opML_g$.
\par
The transverse measure is a typical example of a geometric construct defined in this manner. Consider first a measured geodesic lamination, $\lambda$, with underlying geodesic lamination, $L$. A piecewise $C^1$, immersed curve $c:[0,1]\rightarrow S$ is said to be {\sl compatible} with $\lambda$ (or, equivalently, with $L$) whenever it is transverse to $L$ and, in addition, each of the end points of all of its $C^1$ components lie in the complement of $L$. When a compatible curve, $c$, is also embedded, its {\sl transverse measure} is defined by
\begin{equation}\label{TransverseMeasure}
\tau(c,\lambda) := \lim_{r\rightarrow 0}\frac{\lambda(N_r(c))}{2r}~,
\end{equation}
where $N_r(c)$ here denotes the neighbourhood of radius $r$ about $c$. This function is additive under the action of concatenation of this curves, and via this property extends to a unique function defined over the space of all immersed, compatible curves. Now, when $\lambda$ is rational, this limit is well defined, and counts with appropriate weights the number of times that $c$ crosses $L$. More generally,
\begin{theorem}\label{ExistenceOfTransverseMeasure}
\noindent The limit, $\tau(c,\lambda)$, exists for every measured geodesic lamination, $\lambda$, and for every curve, $c$, which is compatible with $\lambda$. Furthermore, if $(\lambda_m)_{m\in \mathbb{N}}$ is a sequence of measured geodesic laminations converging to $\lambda$, and if $c$ is compatible with $\lambda_m$ for all $m$, then
\begin{equation*}
\lim_{m\rightarrow+\infty}\tau(c,\lambda_m) = \tau(c,\lambda)~.
\end{equation*}
\end{theorem}
Conversely, consider a fixed geodesic lamination, $L$, and let $c(L)$ denote the family of all immersed curves that are compatible with it. A {\sl transverse measure} over $L$ is defined to be a function, $\tau:c(L)\rightarrow[0,\infty[$, which is constant over each $C^1$ isotopy class in $c(L)$ and which is additive with respect to concatenation of curves. We now have the following converse to Theorem \ref{ExistenceOfTransverseMeasure}.
\begin{theorem}
\noindent Given a geodesic lamination, $L$, and a transverse measure, $\tau$, there exists a unique measured geodesic lamination, $\lambda$, supported on $L$ such that $\tau=\tau(\cdot,\lambda)$.
\end{theorem}
It follows that measured geodesic laminations can equally well be defined via their transverse measures. In fact, as we will see presently, it is the transverse measures that generally arise in a natural manner from constructions of Teichm\"uller theory, and it is for this reason that the definition in terms of transverse measures is actually more standard (c.f. \cite{Bon01}).
\par
The transverse measure also serves to parametrise the space of measured geodesic laminations in a manner that does not depend on the hyperbolic metric chosen. To see this, first let $\langle\Pi_1\rangle$ denote the set of free homotopy classes in $S$. For a given measured geodesic lamination, $\lambda$, its {\sl mass function}, $M_g(\lambda):\langle\Pi_1\rangle\rightarrow[0,\infty[$, is defined by
\begin{equation*}
M_g(\lambda)(\langle\gamma\rangle) := \inf_{\eta\in\langle\gamma\rangle}\tau(\eta,\lambda)~,
\end{equation*}
where $\langle\gamma\rangle$ here denotes a free homotopy class in $\langle\Pi_1\rangle$, and $\eta$ varies over all immersed curves in $\langle\gamma\rangle$ which are compatible with $\lambda$. It turns out that the mass function uniquely defines the measured geodesic lamination. Indeed,
\begin{theorem}
\noindent $M_g$ defines a piecewise linear homeomorphism from $\opML_g$ onto a closed subset of $[0,\infty[^{\langle\Pi_1\rangle}$, where the latter is furnished with the topology of pointwise convergence.
\end{theorem}
\noindent When $\lambda$ is rational, with underlying geodesic lamination, $L$, $M_g(\lambda)(\langle\gamma\rangle)$ counts, with appropriate weights, the infimal number of times that an element of the free homotopy class, $\langle\gamma\rangle$, crosses $L$. More precisely, if $L=c_1\cup \ldots \cup c_m$, where $c_1,\ldots,c_m$ are disjoint, simple, closed geodesics, then
\begin{equation*}
M_g(\lambda)(\langle\gamma\rangle) := \inf_{\eta\in\langle\gamma\rangle}\sum_{k=1}^m\frac{\lambda(c_k)}{l(c_k)}\#\eta^{-1}(c_k)~,
\end{equation*}
where $\lambda(c_k)$ here denotes the mass of $c_k$ with respect to the measure, $\lambda$, $l(c_k)$ denotes its length with respect to the metric, $g$, $\#\eta^{-1}(c_k)$ denotes the cardinality of its preimage under $\eta$, and $\eta$ varies over all immersed curves in the free homotopy class, $\langle\gamma\rangle$, which are compatible with $\lambda$. Significantly, the concept of mass function is purely topological in the sense that the image under $M_g$ of the set of rational measured geodesic laminations is independent of the metric chosen, and since the rational measured geodesic laminations form a dense subset of $\opML_g$, the entire image of $M_g$ is also independent of the metric chosen. Furthermore,
\begin{theorem}
\noindent Given two hyperbolic metrics, $g$ and $g'$, the composition, $M_g^{-1}\circ M_{g'}$, defines a piecewise linear homeomorphism from $\opML_{g'}$ into $\opML_g$.
\end{theorem}
\noindent We denote the image of $M_g$ by $\opML$. This space, which parametrises measured geodesic lamination independently of the hyperbolic metric chosen, naturally carries the structure of a piecewise linear manifold of real dimension $(6\mathfrak{g}-6)$. We consider $\opML$ as the space of abstract measured geodesic laminations.
\subsection{Earthquakes and graftings}\label{EarthquakesAndGraftings}
The two non-analytic constructs defined using measured geodesic laminations are earthquakes and graftings. We first consider earthquakes. These are parametrised by $\opML$ and act naturally on $\Thyp$. Consider first a marked hyperbolic metric, $g$, and a measured geodesic, $\lambda:=(c,a)$. Supposing that $\lambda$ is oriented, the {\sl left earthquake} of $g$ along $\lambda$, which we denote by $\mathcal{E}^l_\lambda(g):=\mathcal{E}^l(g,\lambda)$, is obtained by cutting $S$ along $c$ and regluing after rotating the left hand side along $c$ in the positive direction by a distance $a$.\footnote{This is the classical Fenchel-Nielsen deformation by a twist of length $a$ along the curve $c$.} In fact, this definition is independent of the chosen orientation of $c$, since reversing the orientation also exchanges the left and right hand sides. The operation $\mathcal{E}^l$ naturally extends to rational measured geodesic laminations, and then to a unique continuous map from the whole of $\Thyp\times\opML$ into $\Thyp$.
\begin{restatable}[Earthquake theorem,  Kerckhoff \cite{Ker85}]{theorem}{earthquake}
\label{EarthquakeTheorem}
\noindent For all fixed $\lambda$, the map $\mathcal{E}^l_\lambda$ is a real analytic diffeomorphism.
\end{restatable}
The {\sl right earthquake} operator, $\mathcal{E}^r$, is defined in a similar manner by rotating the right hand sides instead of the left hand sides. In particular, for all $\lambda$, the map, $\mathcal{E}^r_\lambda$, is the inverse of $\mathcal{E}^l_\lambda$, and it is therefore natural to define the {\sl earthquake flow},  $\mathcal{E}:\mathbb{R}\times\Thyp\times\opML\rightarrow\Thyp$, by
\begin{equation*}
\mathcal{E}_{t,\lambda}(g):=\mathcal{E}(t,g,\lambda):=
\begin{cases}
\mathcal{E}^l_{t\lambda}(g), & \text{if}\ t\geq 0,\ \text{and}\\
\mathcal{E}^r_{\left|t\right|\lambda}(g) & \text{if}\ t<0~.
\end{cases}
\end{equation*}
For any fixed $\lambda$, the map $\mathcal{E}(\cdot,\cdot,\lambda)$ is real analytic in $\R\times\Thyp$, and, for all $s$ and for all $t$,
\begin{equation*}
\mathcal{E}_{s,\lambda}\circ\mathcal{E}_{t,\lambda} = \mathcal{E}_{s+t,\lambda}~,
\end{equation*}
so that $(\mathcal{E}_{t,\lambda})_{t\in\R}$ constitutes a real analytic group of real analytic diffeomorphisms of $\Thyp$. In particular, for any fixed $(g,\lambda)$, the earthquake flow through $g$ in the direction of $\lambda$ defines an embedded real analytic curve in $\Thyp$, and by Theorem~\ref{EarthquakeTheorem}, there is a unique such curve joining any two points of Teichm\"uller space. However, we find it more suggestive to re-express this result in terms of the mid point of this curve, and we thereby obtain our first ``mid point'' theorem, which is illustrated schematically in Figure~\ref{fig:earthquake}.
\begin{restatable}[Earthquake theorem, Thurston,  Kerckhoff \cite{Ker83}]{theorem}{earthquakebis}
\label{EarthquakeTheorem2}
\noindent Given two marked hyperbolic metrics, $g_1$ and $g_2$, there exists a unique marked hyperbolic metric, $h$, and measured geodesic lamination, $\lambda$, such that
\begin{equationarray*}{llllll}
\ g_1&=&\mathcal{E}(-1,h,\lambda)&=&\mathcal{E}^r_\lambda(h)& \text{and}\\
\ g_2&=&\mathcal{E}(1,h,\lambda)&=&\mathcal{E}^l_\lambda(h)~.
\end{equationarray*}
\end{restatable}
\begin{figure}
\begin{center}
\includegraphics[scale=0.5]{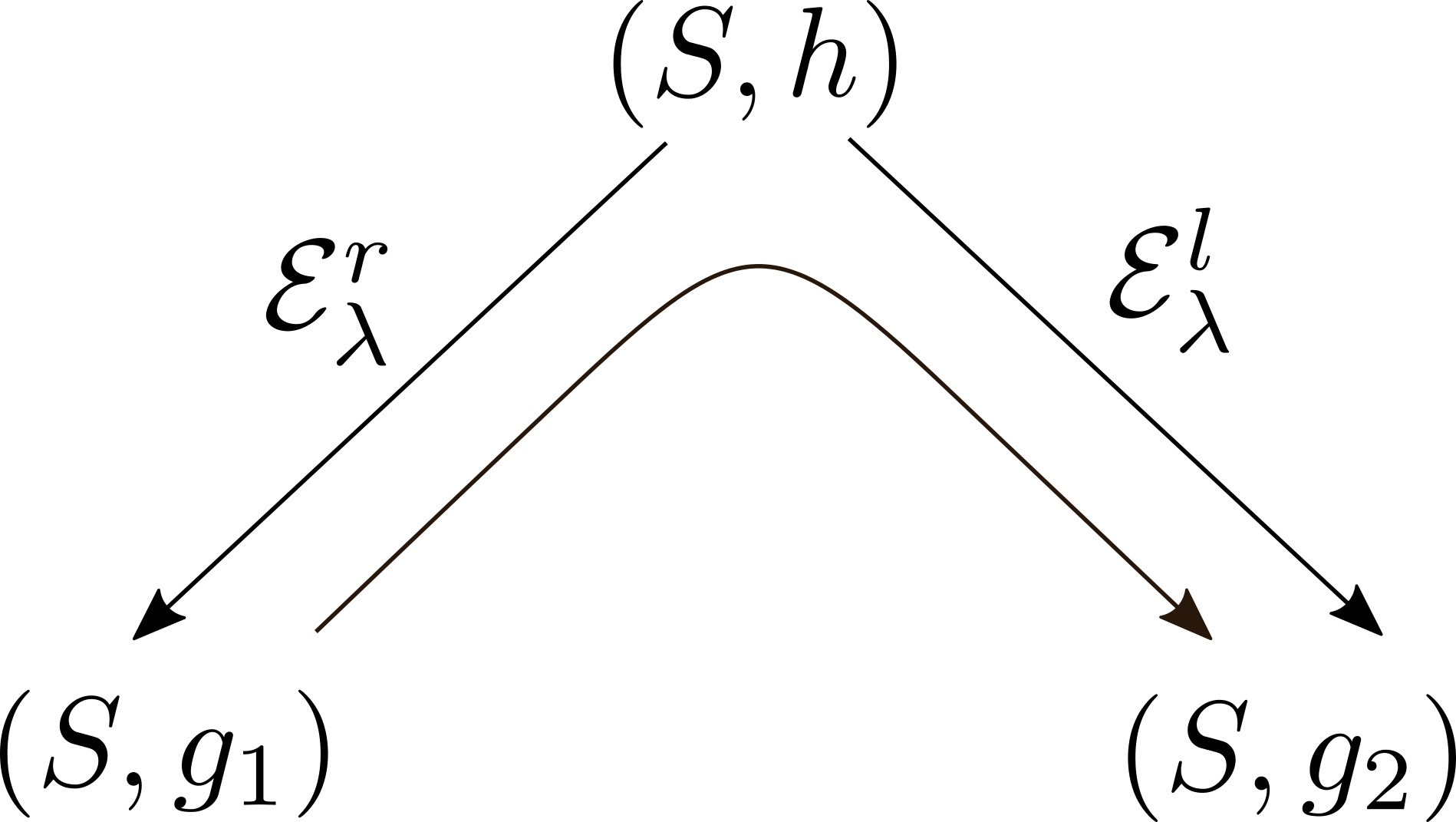}\caption{\label{fig:earthquake} \textbf{The earthquake theorem.}\ Here the upper vertex is given by a marked hyperbolic structure over $S$. The earthquake maps, $\mathcal{E}_\lambda^r$ and $\mathcal{E}_\lambda^l$, are piecewise continuous, with discontinuities along the measured geodesic lamination, $\lambda$. The curved arrow joining the lower two vertices corresponds to the earthquake $\mathcal{E}_{2\lambda}^l$.}
\end{center}
\end{figure}
\begin{remark} In fact, this is the manner in which the earthquake theorem is usually expressed.\end{remark}
Graftings are likewise parametrised by $\opML$ and map $\Thyp$ naturally into $\Thol$. They are defined as follows. Consider first a marked, hyperbolic metric, $g$, and a weighted geodesic, $\lambda:=(c,a)$. The {\sl grafting} of $g$ along $\lambda$, which we denote by $\mathcal{G}_\lambda(g):=\mathcal{G}(g,\lambda)$, is the holomorphic structure obtained by slicing $S$ along $c$ and inserting the cylinder $c\times[0,a]$. As before, this operation naturally extends to rational measured geodesic laminations, and then to a unique continuous map from the whole of $\Thyp\times\opML$ into $\Thol$. Furthermore, this map is at every point real analytic in the first variable (c.f. \cite{SW02}) and tangentiable in the second (c.f. \cite{DW08}). In fact, we have
\begin{theorem}
\noindent For all fixed $\lambda$, $\mathcal{G}_\lambda$ is a real analytic diffeomorphism. For all fixed $g$, the map $\lambda\mapsto\mathcal{G}(g,\lambda)$ is a bitangentiable homeomorphism.
\end{theorem}
Earthquakes and graftings are unified as follows. Let $\mathbb{H}^+$ be the upper half space in $\mathbb{C}$, let $\overline{\Bbb{H}}^+$ denote its closure, and define the {\sl complex earthquake map}, $\mathcal{E}:\overline{\mathbb{H}}^+\times\Thyp\times\opML\rightarrow\Thol$ by
\begin{equation*}
\mathcal{E}(s+it,g,\lambda) := \mathcal{G}_{t,\lambda}\circ\mathcal{E}_{s,\lambda}(g)~.
\end{equation*}
\begin{theorem}[McMullen \cite{mcmullen}]\label{complexearthquake}
\noindent For all $(g,\lambda)$, the map $z\mapsto\mathcal{E}(z,g,\lambda)$ defines a holomorphic map from $\overline{\mathbb{H}}^+$ into $\Thol$.
\end{theorem}
This map can also be visualised as follows. Consider a marked, hyperbolic metric, $g$, a compact geodesic, $c$, in $S$, and a marked torus, $T$, furnished with a flat metric. Let $c'$ be a geodesic in the homology class of the first basis element of $\pi_1(T)$ (which is well defined by the marking). Upon rescaling the metric of $T$, we may suppose that $c'$ has the same length as $c$. The complex earthquake of $g$ by $T$ is then obtained by slicing $S$ along $c$, $T$ along $c'$, and by joining the resulting two surfaces together by identifying these two geodesics. In particular, by varying $T$ over its Teichm\"uller space, we see how complex earthquakes continuously interpolate between graftings and earthquakes, and also how the space of complex earthquakes along a given geodesic identifies with the upper half space, since this is the natural parametrisation of the Teichm\"uller space of the marked torus.
\par
Theorem~\ref{EarthquakeTheorem2} is in some sense compatible with the symmetry of the domain of $\mathcal{E}$ with respect to the imaginary axis, and by the same token, we should not expect results analogous to this theorem for graftings. However, since $\mathcal{E}$ is analytic along the boundary of $\mathbb{H}^+$, it nonetheless extends analytically to a neighbourhood of $\overline{\mathbb{H}}^+$. It seems to us an interesting problem to determine how far along the negative imaginary axis this analytic continuation can be developed. Indeed, in the case where $\lambda:=(c,a)$ is a weighted geodesic, this would correspond heuristically to the supremal conformal modulus of cylinders in the homotopy class of $c$ that can be removed from $S$.
\subsection{Harmonic maps and minimal lagrangian diffeomorphisms}
\label{HarmonicMapsAndMinimalLagrangianDiffeomorphisms}
Given a holomorphic structure, $\mathcal{H}$, and a symmetric $2$-form, $a$, over $S$, the {\sl energy density} of $a$ with respect to $\mathcal{H}$ is defined by
\begin{equation*}
\opE(a|\mathcal{H}) := \opTr_h(a)\opdVol_h~,
\end{equation*}
where $h$ is any Riemannian metric conformal to $\mathcal{H}$ and $\opdVol_h$ is its volume form. We readily verify that this is independent of the metric, $h$, chosen. The {\bf energy} of $a$ with respect to $\mathcal{H}$ is then given by
\begin{equation*}
\mathcal{E}(a|\mathcal{H}) := \int_S\opE(a|\mathcal{H}) = \int_S\opTr_h(a)\opdVol_h~.
\end{equation*}
We are interested in critical points of this functional. However, since $\mathcal{E}$ is linear in $a$, this is only meaningful if we first restrict attention to some subspace. We therefore say that the $2$-form, $a$, is {\sl harmonic} whenever it is a critical point of $\mathcal{E}$ within the subspace $\left\{\phi^*a\right\}$, where $\phi$ here ranges over all smooth diffeomorphisms of $S$.
\par
Of particular interest to us is the case where $a$ is replaced by the pull-back of some fixed riemannian metric, $g$, through some $C^1$-map, $f:S\rightarrow S$. We then say that the map, $f$, is {\sl harmonic} with respect to
$(\mathcal{H},g)$ whenever the form $f^*g$ is harmonic with respect to $\mathcal{H}$.\footnote{Usually, in the literature, we say that $f$ is harmonic whenever it is a critical point of the functional $f\mapsto\mathcal{E}(\mathcal{H}|f^*g)$. Although this is more general than the condition of harmonicity of $f^*g$, in the case of interest to us, where the target manifold is $2$-dimensional, the two definitions are equivalent.} The following is a combination of many results \cite{ES64,hart67,al68,sam78,schoenyau}.
\begin{theorem}
\label{HarmonicDiffeomorphisms}
\noindent Given a marked holomorphic structure, $\mathcal{H}$, and a marked hyperbolic metric, $g$, there exists a unique harmonic diffeomorphism, $f:(S,\mathcal{H})\rightarrow(S,g)$, which preserves the markings.
\end{theorem}
Symmetrizing this concept of harmonic diffeomorphisms, Schoen developed the following notion of minimal lagrangian diffeomorphisms. Given two metrics, $g_1$ and $g_2$, over $S$, a diffeomorphism $f:S\rightarrow S$ is said to be {\sl minimal lagrangian} with respect to $(g_1,g_2)$ whenever its graph is a minimal lagrangian submanifold of the product space, $S\times S$, furnished with the metric, $g_{12}:=\pi_1^*h_1+\pi_2^*h_2$, and the symplectic form, $\omega_{12}:=\pi_1^*\opdVol_{1} - \pi_2^*\opdVol_{2}$, where $\pi_1$ and $\pi_2$ are the respective projections onto the first and second components, and $\opdVol_1$ and $\opdVol_2$ are the respective volume forms of the metrics, $g_1$ and $g_2$.
\par
The relationship between minimal lagrangian diffeomorphisms and harmonic diffeomorphisms becomes clearer when we recall the classical result of minimal surface theory (c.f. \cite{osserman}) that says that if $X$ is an immersed surface in the product space $S\times S$, and if $\mathcal{H}$ is the conformal structure of the restriction of $g_{12}$ to $X$, then $X$ is minimal if and only if its coordinate functions are harmonic with respect to $\mathcal{H}$, that is, if and only if the restrictions to $X$ of $\pi_1$ and $\pi_2$ are harmonic. In the case where $X$ is a graph, it naturally identifies with $S$, so that the two projections $\pi_1$ and $\pi_2$ then yield harmonic maps from $(S,\mathcal{H})$ into $(S,g_1)$ and $(S,g_2)$ respectively. Although there may exist minimal graphs which are not lagrangian, the lagrangian condition ensures uniqueness, yielding the following ``mid point'' theorem for harmonic maps, which is illustrated schematically in Figure~\ref{fig:schoen}.
\begin{restatable}[Labourie \cite{Lab92b}, Schoen \cite{sch93}]{theorem}{MLD}
\label{MinimalLagrangianDiffeomorphisms}
\noindent Given two marked hyperbolic metrics, $g_1$ and $g_2$, there exists a unique minimal lagrangian diffeomorphism, $f:(S,g_1)\rightarrow (S,g_2)$, which preserves the markings.
\end{restatable}
\begin{figure}
\begin{center}
\includegraphics[scale=0.5]{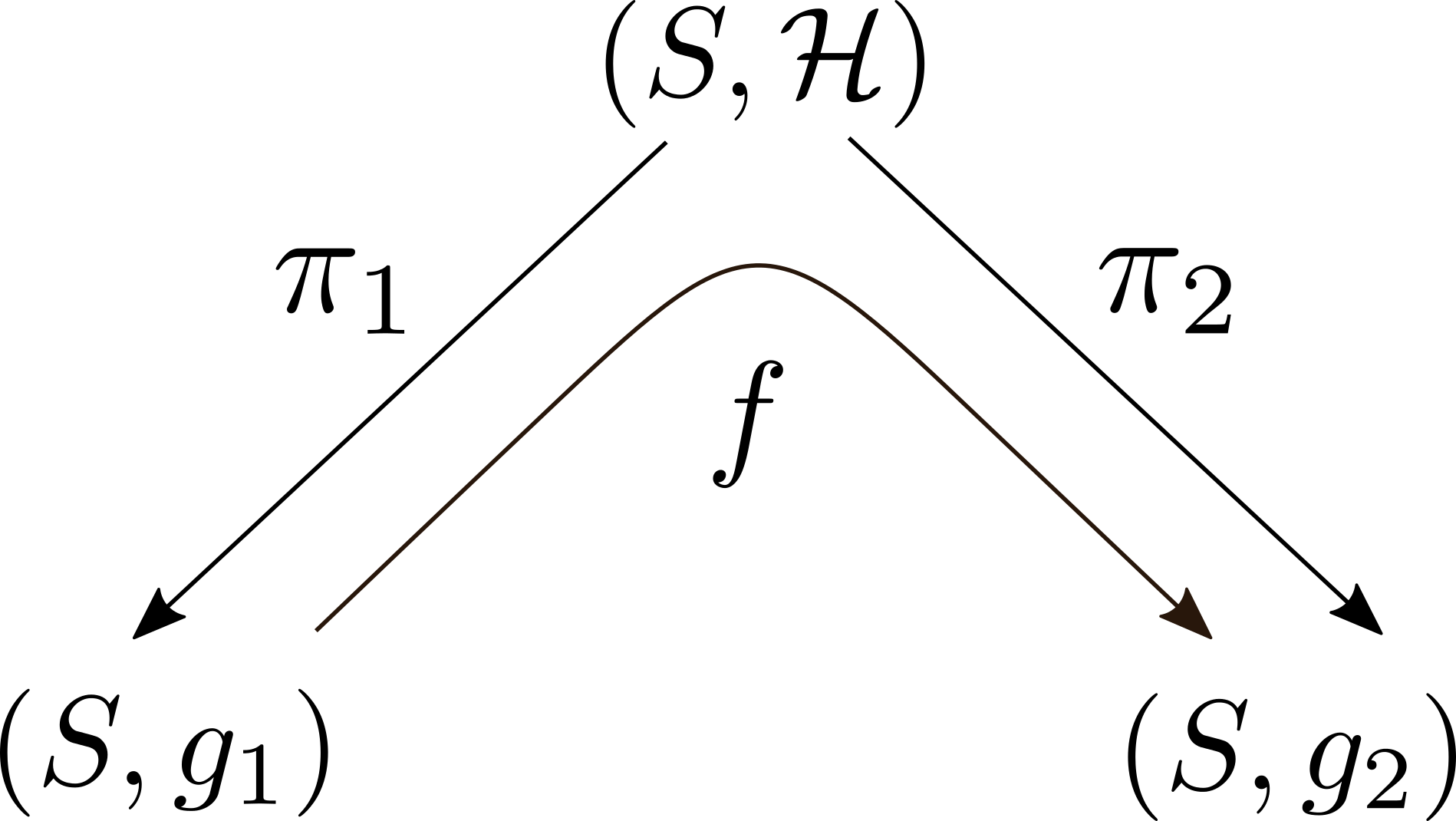}\caption{\label{fig:schoen} \textbf{Minimal lagrangian diffeomorphisms.} Here the upper vertex is given by a marked holomorphic structure over $S$, and the curved arrow joining the two lower vertices is the unique minimal lagrangian diffeomorphism between these points which preserves the markings.}
\end{center}
\end{figure}
\subsection{Hopf differentials}\label{HopfDifferentials}
Before proceeding, it is worth rephrasing Theorems \ref{HarmonicDiffeomorphisms} and \ref{MinimalLagrangianDiffeomorphisms} in a linear manner. To this end, we first introduce Hopf differentials as follows. Recall that any real, symmetric $2$-form, $a$, over $\mathbb{C}$ decomposes naturally as
\begin{equation*}
a =: \phi + \rho + \overline{\phi}~,
\end{equation*}
where $\rho$ is a real-valued $(1,1)$-form and $\phi$ is a $(2,0)$-form which we refer to as the {\sl Hopf form} of $a$ (with respect to the complex structure of $\mathbb{C}$). This concept extends to forms over bundles, so that for a given holomorphic structure, $\mathcal{H}$, and a given symmetric $2$-form, $a$, we define $\phi(a|\mathcal{H})$, the {\sl Hopf differential} of $a$ with respect to $\mathcal{H}$, to be its $(2,0)$-component. The Hopf differential is readily calculated as follows. Let $J\in\Gamma(\opEnd(TS))$ be the complex structure of $\mathcal{H}$, and for a given point $x\in S$, let $(e_1,e_2)$ be a basis of $T_xS$ such that $Je_1=e_2$. If $(e^1,e^2)$ is its dual basis, then
\begin{equation*}
\phi(a|\mathcal{H}) = \frac{1}{4}a(e_1-ie_2,e_1-ie_2)dzdz~,
\end{equation*}
where $dz:=e^1+ie^2$. Hopf differentials naturally arise in the study of harmonic forms due to the following result.
\begin{lemma}
\label{HopfDifferentialsAndHarmonicForms}
\noindent The symmetric $2$-form, $a$, is harmonic with respect to $\mathcal{H}$ if and only if the Hopf differential, $\phi(a|\mathcal{H})$, is holomorphic with respect to $\mathcal{H}$.
\end{lemma}
\begin{proof} Indeed, let $h$ be a riemannian metric conformal to $\mathcal{H}$. Consider an exponential chart about a point $x$ in $S$. Recall that the form $a$ is harmonic if and only if
\begin{equation*}
\omega_k := h^{ij}(2a_{ki;j} - a_{ij;k}) = 0~,
\end{equation*}
where the subscript ``;'' here denotes covariant differentiation with respect to the Levi-Civita covariant derivative of $h$. On the other hand, $\overline{\partial}\phi(a|\mathcal{H})(0)=\psi dzdzd\overline{z}$, where
\begin{equation*}
\psi = (a_{11;1} - a_{22;1} + 2a_{12;1}) + (a_{11;2} - a_{22;2} - 2a_{12;1})i = \omega_1 - i\omega_2~,
\end{equation*}
and the result follows.\end{proof}
Lemma~\ref{HopfDifferentialsAndHarmonicForms} allows us to restate the results of the previous section in the language of Hopf differentials. First, recall that the cotangent bundle to $\Thol$ naturally identifies with the space of all pairs $(\mathcal{H},\phi)$, where $\mathcal{H}$ is a marked holomorphic structure and $\phi$ is a quadratic holomorphic differential with respect to $\mathcal{H}$ (c.f. \cite{FK}). Theorem~\ref{HarmonicDiffeomorphisms} now yields a well defined map $\Phi:\Thol\times\Thyp\rightarrow \opT^*\Thol$. Indeed, for all $\mathcal{H}$ and for all $g$, $\Phi(\mathcal{H},g)=\phi(f^*g|\mathcal{H})$, where $f:(S,\mathcal{H})\rightarrow(S,g)$ is the unique harmonic diffeomorphism which preserves the markings.
\begin{theorem}[Wolf \cite{wol89}]\label{wolf}
\label{HopfDifferentialMapIsRealAnalyticDiffeomorphism}
\noindent $\Phi$ is a real analytic diffeomorphism.
\end{theorem}
Now let $\Psi:\opT^*\Thol\rightarrow\Thol\times\Thyp$ denote the inverse of $\Phi$. The existence part of Theorem~\ref{MinimalLagrangianDiffeomorphisms} is now restated in the form of a partial ``mid point'' theorem as follows.
\begin{theorem}
\label{SymmetricHopfDifferentials}
\noindent Given two marked hyperbolic metrics, $g_1$ and $g_2$, there exists a marked holomorphic structure, $\mathcal{H}$, and a holomorphic quadratic differential, $\phi$, such that
\begin{equation*}
\begin{cases}
\Psi(\mathcal{H},\phi) &= g_1\ \text{and}\\
\Psi(\mathcal{H},-\phi) &= g_2~.\end{cases}
\end{equation*}
\end{theorem}
\subsection{Labourie fields}\label{LabourieFields}
It is the lagrangian property in Theorem~\ref{MinimalLagrangianDiffeomorphisms} that ensures uniqueness. Since this has no straightforward interpretation in terms of Hopf differentials, we now introduce the more refined notion of Labourie fields. First, recall that if $g$ is any riemannian metric with Levi-Civita covariant derivative, $\nabla$, then a symmetric endomorphism field, $A:TS\rightarrow TS$, is said to be a {\sl Codazzi field} whenever
\begin{equation*}
(\nabla_X A)Y = (\nabla_Y A)X~,
\end{equation*}
for all vector fields $X$ and $Y$. The Codazzi field, $A$, is then said to be a {\sl Labourie field} whenever, in addition, it is positive definite, and
\begin{equation*}
\opDet(A) = 1~.
\end{equation*}
Labourie fields are characterised by the following useful result.
\begin{lemma}
\label{CharacterisationOfLabourieField}
\noindent Let $g$ be a riemannian metric, let $A:TS\rightarrow TS$ be a positive definite endomorphism field, and let $\mathcal{H}$ be the holomorphic structure of $g(A\cdot,\cdot)$. Any two of the following imply the third~:
\begin{enumerate}
\item $A$ is a Codazzi field,
\item $\opDet(A)$ is constant, and
\item $\phi(g|\mathcal{H})$ is a holomorphic quadratic differential.
\end{enumerate}
\end{lemma}
\begin{proof} Denote $h:=g(A\cdot,\cdot)$. Using the notation of Lemma~\ref{HopfDifferentialsAndHarmonicForms}, with $g$ in place of $a$, by the Koszul formula, we obtain
\begin{equation*}
\omega_k = - 2g_{kp}h^{pq}h^{ij}(h_{iq:j} - h_{ij:q}) - g_{kp}h^{pq}h^{ij}h_{ij:q}~,
\end{equation*}
where the subscript ``:'' here denotes covariant differentation with respect to the Levi-Civita covariant derivative of $g$. Thus,
\begin{equation*}
\psi = -2(\alpha_1 - i\alpha_2) - g_{kp}h^{pq}(\opDet(A)_{:1} - i\opDet(A)_{:2})~,
\end{equation*}
where
\begin{equation*}
\alpha_k = g_{kp}h^{pq}h^{ij}(h_{iq:j} - h_{ij:q})~.
\end{equation*}
Since $\alpha$ vanishes if and only if $A$ is a Codazzi field, the result follows.\end{proof}
\noindent In particular, Labourie fields possess the following symmetry property.
\begin{cor}\label{SymmetryOfLabourieFields}
\noindent If $A$ is a Labourie field of $g$, then $A^{-1}$ is a Labourie field of the riemannian metric $g(A\cdot,A\cdot)$.
\end{cor}
Now let $\opLab\Thyp$ be the space of all pairs $(g,A)$, where $g$ is a marked hyperbolic metric and $A$ is a Labourie field of $g$. This is a smooth, non-linear bundle over $\Thyp$ with typical fibre of real dimension $(6\mathfrak{g}-6)$ (c.f. \cite{Lab92}). The application, $\Phi$, constructed in the previous section now yields a map $\mathcal{A}:\Thol\times\Thyp\rightarrow\opLab\Thyp$ defined as follows. Given a marked holomorphic structure, $\mathcal{H}$, with complex structure, $J$, and a marked hyperbolic metric, $g$, with complex structure $J_0$, we define $\mathcal{A}(\mathcal{H},g):=-J_0f_*J$, where $f:(S,\mathcal{H})\rightarrow(S,g)$ is the unique harmonic diffeomorphism which preserves the marking, and $f_*$ denotes its push forward operation. It is relatively straightforward to verify that this is indeed a Labourie field. Theorem~\ref{HopfDifferentialMapIsRealAnalyticDiffeomorphism} is now restated as
\begin{theorem}\label{AIsARealAnalDiff}
\noindent $\mathcal{A}$ is a real analytic diffeomorphism.
\end{theorem}
\par
\begin{proof}[Sketch of proof] It suffices to show how to map between $\opLab\Thyp$ and $\opT^*\Thol$. Given a point, $(g,A)$ in $\opLab\Thyp$, let $\mathcal{H}$ to be the marked holomorphic structure conformal to $g(A\cdot,\cdot)$, and let $\phi$ to be the Hopf differential of $g$ with respect to $\mathcal{H}$. By Lemma \ref{CharacterisationOfLabourieField}, $\phi$ is a quadratic holomorphic differential, and $(\mathcal{H},\phi)$ is the point of $\opT^*\Thol$ corresponding to $(g,A)$. Conversely, given a point $(\mathcal{H},\phi)$ in $\opT^*\Thol$, let $g$ be the unique marked hyperbolic metric such that $\Phi(\mathcal{H},g)=(\mathcal{H},\phi)$, let $f:(S,\mathcal{H})\rightarrow(S,g)$ be the unique harmonic diffeomorphism which preserves the marking, and let $A\in\Gamma(\opEnd(TS))$ be such that $\opDet(A)=1$ and $g(A\cdot,\cdot)$ is conformal to $\mathcal{H}$. By Lemma \ref{CharacterisationOfLabourieField}, $A$ is a Labourie field of $g$, and $(g,A)$ is the point of $\opLab\Thyp$ corresponding to $(\mathcal{H},\phi)$.\end{proof}
Using Labourie fields, we now transform Theorem~\ref{MinimalLagrangianDiffeomorphisms} into another ``mid-point'' theorem, which is illustrated schematically in Figure~\ref{fig:labouriefield}. Observe, in particular, that, by Corollary~\ref{SymmetryOfLabourieFields}, the following statement is in fact symmetric.
\begin{figure}
\begin{center}
\includegraphics[scale=0.5]{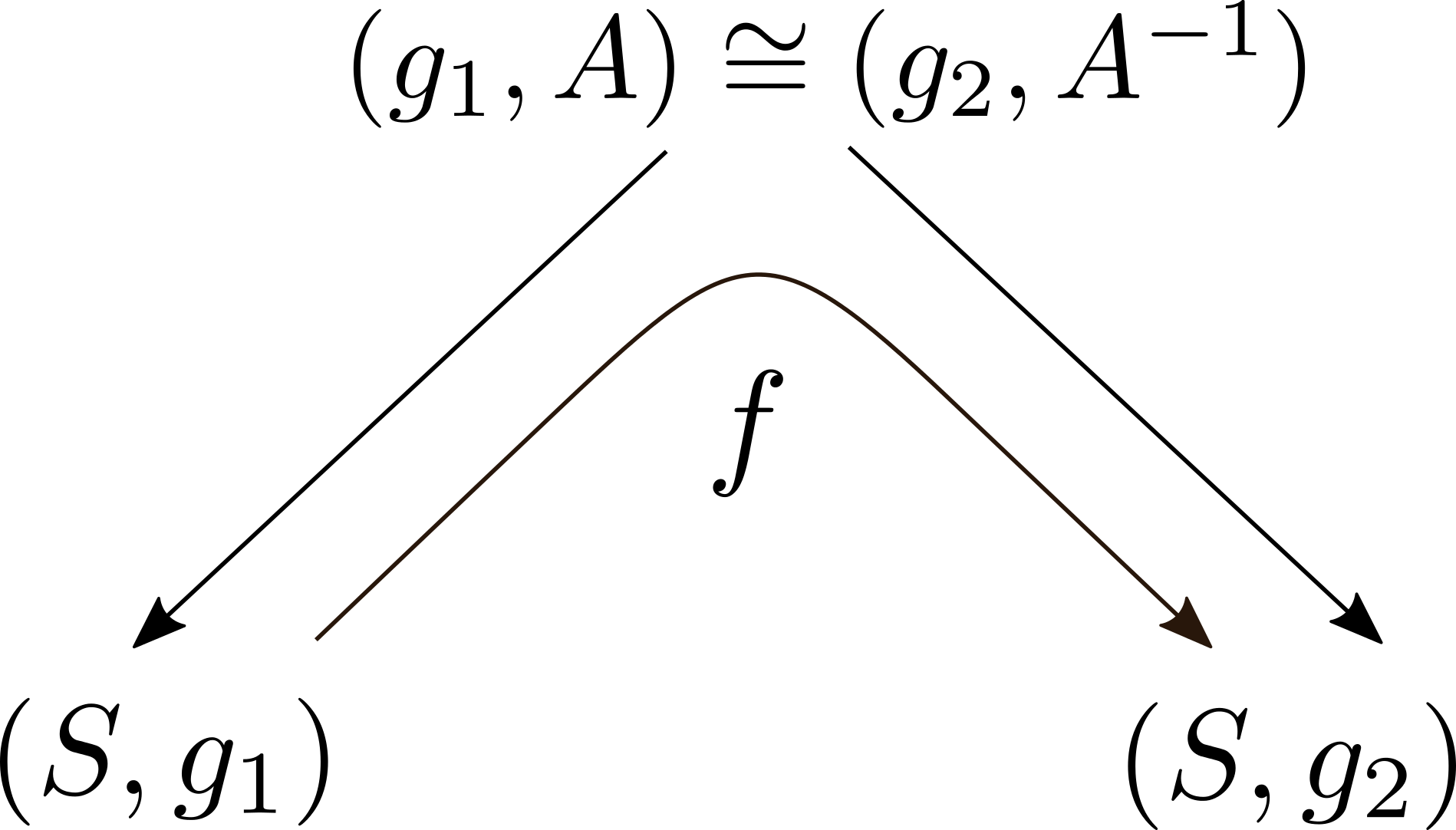}\caption{\label{fig:labouriefield} \textbf{The Labourie field theorem.} Here the upper vertex is given by the Labourie field, $A$, and the curved arrow joining the two lower  vertices is the unique diffeomorphism, $f$, which preserves the marking such that $f^*g_2=g_1(A \cdot, A\cdot )$.}
\end{center}
\end{figure}
\begin{theorem}[Labourie field theorem, Labourie \cite{Lab92b}, Schoen \cite{sch93}]\label{LabFieldthm}
\noindent Given two marked hyperbolic metrics, $g_1$ and $g_2$, there exists a unique Labourie field, $A$, and a unique diffeomorphism $f:S\rightarrow S$ which preserves the markings such that
\begin{equation*}
f^*g_2 = g_1(A\cdot,A\cdot)~.
\end{equation*}
\end{theorem}
\subsection{Landslides and rotations}\label{RotationsAndLandslides}
Landslides, which were introduced by Bonsante, Mondelo \& Schlenker in \cite{BMS1}, are parametrised by $\R\times\Thyp$ and act naturally on $\Thyp$. They are defined as follows. Consider a marked hyperbolic metric, $g$. Given another marked hyperbolic metric, $h$, and a real number, $t$, define the metric $g_{t,h}$ over $S$ by
\begin{equation*}
g_{t,h}(U,V) := g\left(e^{(t/2)J_0A}U,e^{(t/2)J_0A}V\right)~,
\end{equation*}
where $J_0$ is the complex structure of $g$, and $A$ is the unique Labourie field of $h$ with respect to $g$ that is furnished by Theorem~\ref{LabFieldthm}. In fact, denoting $J:=J_0A$, we readily determine that $J^2=-\opId$, so that $g_{t,h}$ can also be expressed in the following manner.
\begin{equation*}
g_{t,h}(U,V) = g\left(\opCos(t/2)J_0 U- \opSin(t/2)AU,\opCos(t/2)J_0V - \opSin(t/2)AV\right)~.
\end{equation*}
We now have,
\begin{lemma}
\noindent For all $g$ and for all $(t,h)$, the metric $g_{t,h}$ is hyperbolic.
\end{lemma}
\begin{proof} Indeed, denote $A_t:=e^{(t/2)J_0A}$. Let $\nabla$ denote the Levi-Civita covariant derivative of $g$. Since $A$ is a Codazzi field, $d^\nabla A_t=0$, and it follows that the Levi-Civita covariant derivative, $\nabla^t$, of $g_t$ is given by
\begin{equation*}
\nabla^tX = A_t^{-1}\nabla(A_t X)~.
\end{equation*}
Bearing in mind that $g$ is hyperbolic, the Riemann curvature tensor of $g_t$ is therefore given by
\begin{equation*}
R^t_{XY}Z = R_{XY}A_tZ = \langle X,A_t Z\rangle Y - \langle Y, A_t Z\rangle X~,
\end{equation*}
and the result now follows since $\opDet(A)=1$.\end{proof}
\noindent We now define $\mathcal{L}_{t,h}(g):=\mathcal{L}(t,g,h):=g_{t,h}$, and we call it the {\sl landslide} of $g$ along $(t,h)$.
\par
For any given $g$ and $h$, the orbit, $\mathcal{L}(\cdot,g,h)$, defines a closed curve in $\Thyp$ of period $2\pi$, and since $\mathcal{L}(0,g,h)=g$ and $\mathcal{L}(\pi,g,h)=h$, we consider it as a circle upon which $g$ and $h$ are diametrically opposite points. It turns out that this ``landslide flow'' shares many of the properties of the earthquake flow already discussed in Section~\ref{EarthquakesAndGraftings} (c.f. \cite{BMS1} and \cite{BMS2}). In particular, given a marked, hyperbolic metric, $h$, consider its {\sl mass function}, $M(h)\in[0,\infty[^{\Gamma_1}$, defined by
$$
M(h)(\langle\gamma\rangle):=
\inf_{\eta\in\langle\gamma\rangle}l_h(\eta)~,
$$
where $\langle\gamma\rangle$ here denotes a free homotopy class in $\langle\Pi_1\rangle$, $\eta$ varies over all closed curves in this free homotopy class, and $l_h(\eta)$ denotes its length with respect to the metric, $h$.
\begin{theorem}[Bonsante--Mondello--Schlenker \cite{BMS1}]
\noindent Let $g$ be a marked hyperbolic metric, let $(h_m)$ be a sequence of marked hyperbolic metrics, and let $(\theta_m)$ be a sequence of positive real numbers. If $(\theta_mM(h_m))_{m\in\mathbb{N}}$ converges pointwise to $M_g(\lambda)$, then
\begin{equation*}
\lim_{m\rightarrow\infty}\mathcal{L}(\theta_m,g,h_m) = \mathcal{E}^l_\lambda(g)~.
\end{equation*}
\end{theorem}
In analogy to the case of earthquakes, it is natural to ask whether there exists a landslide flow between two given hyperbolic metrics, and in \cite{BMS1}, Bonsante, Mondello \& Schlenker provide the affirmative answer in the form of the following ``mid point'' theorem for landslides.
\begin{restatable}[Landslide theorem, Bonsante--Mondello--Schlenker \cite{BMS1}]{theorem}{landslide}
\label{LandslideTheorem}
\noindent Given $t\in]-\pi,\pi[$, and two marked, hyperbolic metrics, $g_1$ and $g_2$, there exists a unique pair, $h^+$ and $h^-$, of marked hyperbolic metrics such that
\begin{equationarray*}{rclclcc}
\ g_1 &=& \mathcal{L}(-t,h^+,h^-)& =& \mathcal{L}_{-t,h^-}(h^+)~,& \text{and}\\
\ g_2 &=& \mathcal{L}(t,h^+,h^-)& =& \mathcal{L}_{t,h^-}(h^+)~.
\end{equationarray*}
\end{restatable}
\begin{figure}[t]
\begin{center}
\includegraphics[scale=0.5]{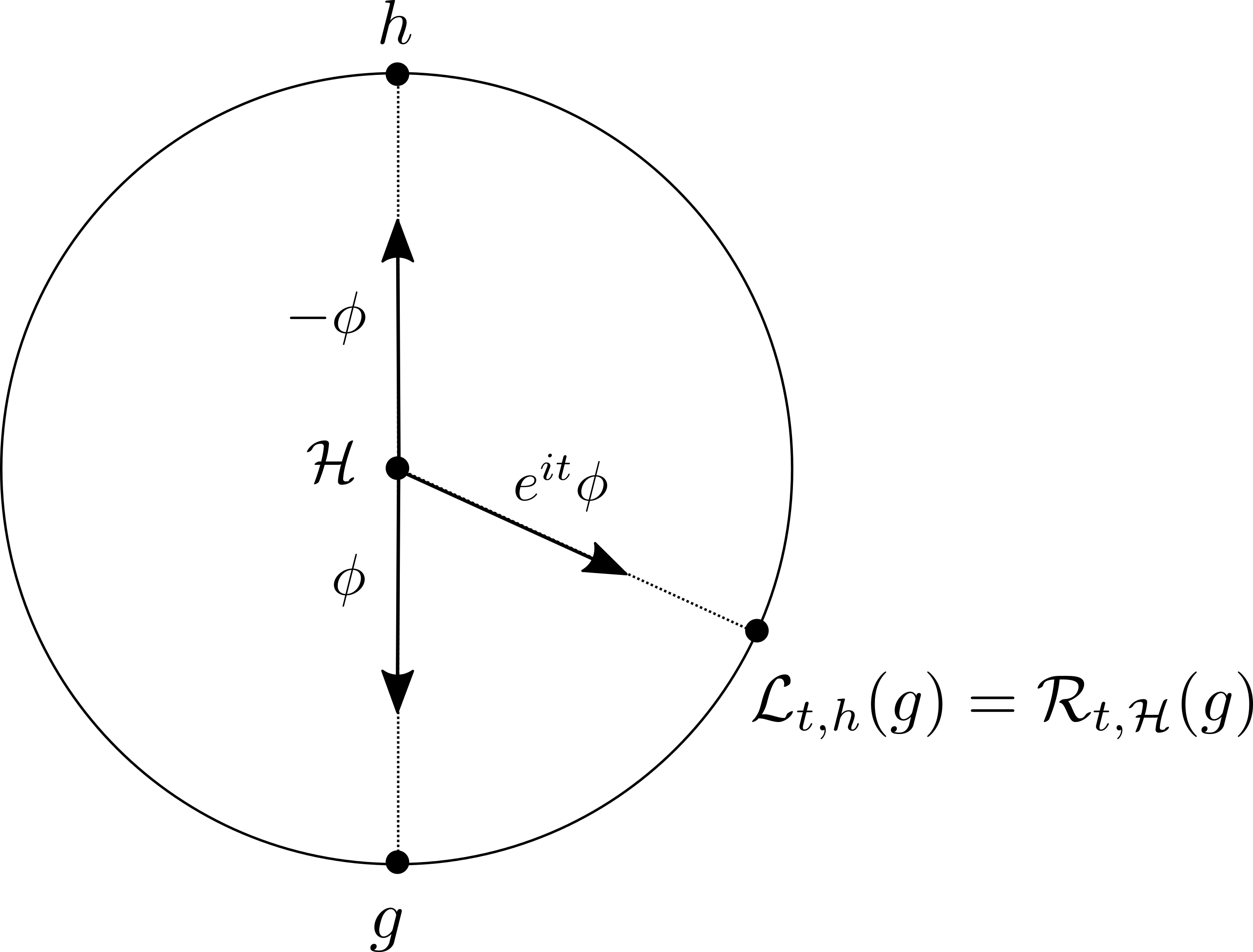}\caption{\label{fig:rotland} \textbf{Landslides and rotations.} Here the landslide of $g$ along $(t,h)$ is shown to coincide with the rotation of $g$ by an angle of $t$ about the marked holomorphic structure of the metric $g(A\cdot,\cdot)$, where $A$ is the unique Labourie field of $h$ with respect to $g$ furnished by Theorem~\ref{LabFieldthm}.}
\end{center}
\end{figure}
A more intuitively appealing approach to landslides is given by rotations, which are defined as follows. Given a marked hyperbolic metric, $g$, and a marked holomorphic structure, $\mathcal{H}$, let $\phi$ be the unique Hopf differential of $g$ with respect to $\mathcal{H}$ furnished by Theorem~\ref{HopfDifferentialMapIsRealAnalyticDiffeomorphism}. Now, given a real number, $t$, define $\mathcal{R}_{t,\mathcal{H}}(g):=\mathcal{R}(t,g,\mathcal{H})$ to be the unique marked, hyperbolic metric whose Hopf differential with respect to $\mathcal{H}$ is $e^{it}\phi$, see Figure~\ref{fig:rotland}. We call this the {\sl rotation} of $g$ by an angle $t$ about $\mathcal{H}$. Landslides and rotations are completely equivalent. Indeed,
\begin{theorem}
\noindent For all $(t,g,h)$,
\begin{equation*}
\mathcal{L}(t,g,h) = \mathcal{R}(t,g,\mathcal{H})~,
\end{equation*}
where $\mathcal{H}$ is the marked holomorphic structure of the metric $g(A\cdot,\cdot)$, and $A$ is the unique Labourie field of $h$ with respect to $g$ given by Theorem~\ref{LabFieldthm}.
\end{theorem}
\noindent This correspondence is illustrated in Figure~\ref{fig:rotland}. In particular, whilst the corresponding result for landslides is less straightforward to state, it is easy to see that the composition of two rotations about a given marked holomorphic structure, $\mathcal{H}$, is another rotation about the same marked holomorphic structure. Theorem~\ref{LandslideTheorem}, is thus restated as follows, and is illustrated schematically in Figure~\ref{fig:rotation}.
\begin{figure}[t]
\begin{center}
\includegraphics[scale=0.5]{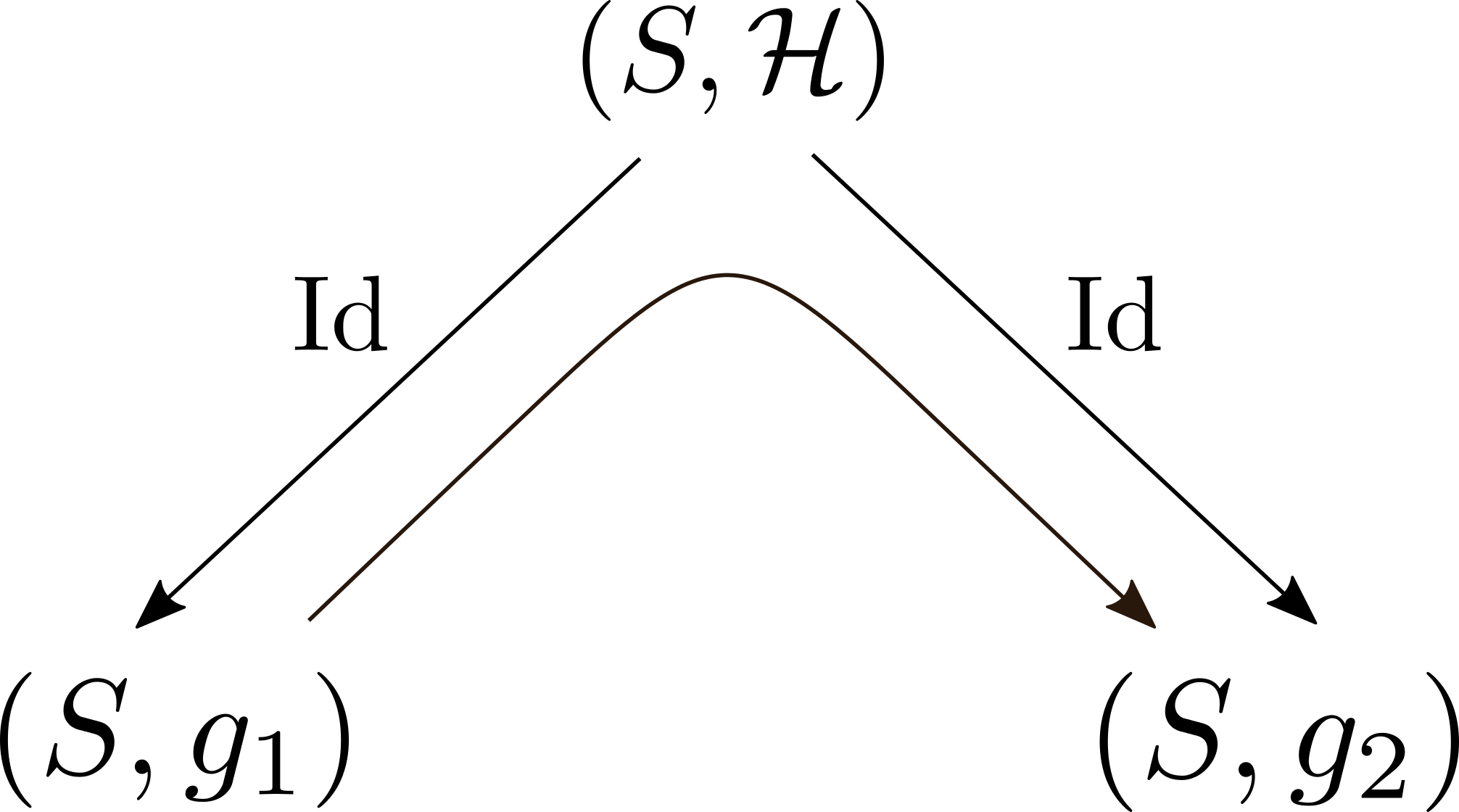}
\caption[caption]{ \textbf{The rotation theorem.} Here the upper vertex is given by a marked holomorphic structure over $S$, and the curved arrow joining the two lower vertices corresponds to a rotation of angle $2t$ about this point.}\label{fig:rotation}
\end{center}
\end{figure}
\begin{restatable}[Rotation theorem, Bonsante--Mondello--Schlenker \cite{BMS1}]{theorem}{rotation}
\label{rotations}
\noindent Given $t\in]-\pi,\pi[$, and two marked hyperbolic metrics, $g_1$ and $g_2$, there exists a unique marked hyperbolic metric, $h$, and a unique marked holomorphic structure, $\mathcal{H}$, such that
\begin{equationarray*}{rclclcc}
\ g_1 &=& \mathcal{R}_{-t,\mathcal{H}}(h) &=& \mathcal{R}(-t,h,\mathcal{H})~,& \text{and}\\
\ g_2 &=& \mathcal{R}_{t,\mathcal{H}}(h) &=& \mathcal{R}(t,h,\mathcal{H})~.
\end{equationarray*}
\end{restatable}
%
%
\section{Minkowski space}\label{MinkowskiSpace}
\subsection{Minkowski spacetimes}\label{MinkowskiSpacetimes}
{\sl Minkowski space}, which will be denoted by $\R^{2,1}$, is defined to be the space of all real triplets, $x:=(x_1,x_2,x_3)$, furnished with the metric\footnote{Throughout this text, in keeping with the notation of theoretical physics, spatial directions will have negative sign and temporal directions will have positive sign.}
\begin{equation*}
\langle x,y\rangle_{2,1} :=  - x_1y_1 - x_2y_2 + x_3y_3~.
\end{equation*}
An important object for the study of this space is the unit pseudosphere, defined to be the locus of all vectors of norm-squared equal to $1$. This hyperboloid consists of two isometric connected components, and the future oriented component is identified with $2$-dimensional hyperbolic space, that is
\begin{equation*}
\mathbb{H}^2 = \left\{x\ |\ \langle x,x\rangle_{2,1} = 1,\ x_3>0\right\}~.
\end{equation*}
The isometry group of Minkowski space is $\opO(2,1)\ltimes\R^{2,1}$, where $\opO(2,1)$ acts linearly, and $\R^{2,1}$ acts by translation. By considering the action of $\opO(2,1)$ on the unit pseudosphere, we see that this group consists of $4$ connected components which are determined by whether they preserve or reverse the orientation, and whether they preserve or exchange the two connected components of the unit pseudosphere. In particular, the identity component of $\opO(2,1)$ identifies with the group of orientation preserving isometries of $\mathbb{H}^2$, that is, $\opPSL(2,\R)$.
\par
Let $S$ be a closed surface of hyperbolic type. In all that follows, an {\sl equivariant immersion} of $S$ into $\R^{2,1}$ is a pair, $(e,\theta)$, where $e$ is an immersion of the universal cover, $\tilde{S}$, of $S$ into $\R^{2,1}$, which is {\sl locally strictly convex (LSC)} in the sense that its shape operator is everywhere positive definite, and $\theta:\Pi_1\rightarrow\opPSL(2,\R)\ltimes\R^{2,1}$ is a homomorphism such that, for all $\gamma\in\Pi_1$,
\begin{equation*}
e\circ\gamma = \theta(\gamma)\circ e~,
\end{equation*}
where $\Pi_1$ acts on $\tilde{S}$ by deck transformations. The homomorphism, $\theta$, will henceforth be referred as the {\sl holonomy} of the equivariant immersion. Two equivariant immersions, $(e,\theta)$ and $(e',\theta')$, will be considered equivalent whenever there exists a homeomorphism, $\phi$, of $S$ which is homotopic to the identity, and an element, $\alpha$, of $\opPSL(2,\R)\ltimes\R^{2,1}$ such that
\begin{eqnarray*}
e' &=& \alpha\circ e\circ\tilde{\phi}~,\ \text{and}\\
\theta' &=& \alpha\theta\alpha^{-1}~,\end{eqnarray*}
where $\tilde{\phi}:\tilde{S}\rightarrow\tilde{S}$ is a lifting of $\phi$. Throughout the sequel, by abuse of terminology, the equivalence class, $[e,\theta]$, of $(e,\theta)$ will also be referred to as an {\sl equivariant immersion}, and, furthermore, will be identified with the representative element, $(e,\theta)$, whenever convenient.
\par
Of repeated use throughout the sequel will be the following result (c.f. for example \cite{BGM})\footnote{Naturally, the fundamental theorem of surface theory is better known in the riemannian case (c.f. \cite{spi3}). The only difference between the two is the formula for $K$, which in the riemannian case is given by $K=c+\opDet(A)$.}.
\begin{theorem}[Fundamental theorem of surface theory --- lorentzian case]\label{FTS}
\noindent Let $S$ be a surface, let $g$ be a riemannian metric over $S$, and let $K$ be its sectional curvature. If $c\in\mathbb{R}$ is a real number, and if $A$ is a Codazzi field such that $K=c-\opDet(A)$, then there exists an isometric, spacelike immersion, $e:S\rightarrow M^c$, whose shape operator is given by $A$, where $M^c$ here denotes the model lorentzian spacetime of constant sectional curvature equal to $c$. Furthermore, $e$ is unique up to composition by isometries of $M^c$.
\end{theorem}
\par
The set of homomorphisms which arise as holonomies of equivariant immersions is characterised in a straightforward manner in terms of Teichm\"uller data. Indeed, given an equivariant immersion, $[e,\theta]$, its holonomy, $\theta$, defines a tangent vector to $\Trep$ as follows. Consider first the linear component, $\theta_0$, of $\theta$, which sends $\Pi_1$ into $\opPSL(2,\R)$. Since the equivariant immersion, $e$, is spacelike, it is actually an embedding, and is even a complete graph over a plane (c.f. \cite{mes+}). From this simple fact, we deduce
\begin{theorem}[Mess \cite{Mess}]\label{thm:mess mink}
\noindent The homomorphism, $\theta_0$, is injective, and its image, $\theta_0(\Pi_1)$, acts properly discontinuously on $\mathbb{H}^2$.
\end{theorem}
\noindent Indeed, let $N_e:\tilde{S}\rightarrow\mathbb{H}^2$ be the future oriented, unit, normal vector field over $e$. This map, which will henceforth be referred to as the {\sl Gauss map} of $e$, is $\theta_0$-equivariant in the sense that for all $\gamma\in\Pi_1$,
\begin{equation*}
N_e\circ\gamma = \theta_0(\gamma)\circ N_e~.
\end{equation*}
It follows by local strict convexity that $N_e$ is at every point a local diffeomorphism, and since $\Pi_1$ acts cocompactly on $\tilde{S}$, and since the latter is simply connected, a straightforward argument shows that $N_e$ is actually a global diffeomorphism. In particular, $\theta_0$ is injective and $\theta_0(\Pi_1)$ acts properly discontinuously on $\mathbb{H}^2$, as desired.
\par
It follows that the quotient, $\mathbb{H}^2/\theta_0(\Pi_1)$, is a compact hyperbolic surface, so that $\theta_0$ indeed defines a point of $\Trep$. Furthermore, if $(e',\theta')$ is another equivariant immersion in the same equivalence class, then the linear component, $\theta'_0$, of its holonomy is conjugate to $\theta_0$. It therefore defines the same point of $\Trep$, so that the point of Teichm\"uller space determined by $\theta$ is indeed a function of the equivalence class, $[e,\theta]$, only.
\par
The tangent vector above this base point is defined via the translation component, $\tau$, of $\theta$. Indeed, $\tau$ satisfies the following {\sl cocycle condition},
\begin{equation*}
\tau(\gamma\gamma') = \tau(\gamma) + \theta_0(\gamma)\tau(\gamma')~,
\end{equation*}
where $\gamma$ and $\gamma'$ are arbitrary elements of $\Pi_1$. This means that $\tau$ defines a class in the cohomology group, $H^1_{\theta_0}(\Pi_1,\R^{2,1})$. However, since $\R^{2,1}$ itself identifies with the Lie algebra, $\opsl(2,\R)$, this is the same as the cohomology group, $H^1_{\opAd\circ \theta_0}(\Pi_1,\opsl(2,\R))$, which is known to identify with the tangent space to $\Trep$ at the point $[\theta_0]$ (c.f. \cite{gol84}). Furthermore, suppose that $(e',\theta')$ is another equivariant immersion in the same equivalence class whose holonomy, $\theta'$, has the same linear component as $\theta$. Then $\theta'$ only differs from $\theta$ by conjugation by a translation, and so $\tau'$ only differs from $\tau$ by addition of a coboundary. It therefore defines the same class in $H^1_{\opAd\circ\theta_0}(\Pi_1,\opsl(2,\R))$, so that the tangent vector to Teichm\"uller space determined by $\theta$ is also a function of the equivalence class, $[e,\theta]$, only.
\par
In summary, Theorem~\ref{thm:mess mink} refines to
\begin{theorem}[Mess \cite{Mess}]\label{mess mink 22}
\noindent The homomorphism, $\theta_0$, is injective, and its image, $\theta_0(\Pi_1)$, acts properly discontinuously on $\mathbb{H}^2$. Furthermore, the translation component, $\tau$, is a $\theta_0$-cocycle, so that the holonomy, $\theta$, defines a point of $\opT\Trep$. In particular, this point only depends on the equivalence class, $[e,\theta]$, of $(e,\theta)$.
\end{theorem}
The converse problem of recovering an equivariant immersion from its holonomy is more involved. Indeed, it is straightforward to show that perturbations of equivariant immersions yield entire continua of inequivalent equivariant immersions with the same holonomy, so that the problem is clearly highly degenerate. However, this degeneracy can be removed by studying, instead of the immersion itself, the ambient space in which it lies. In fact, every equivariant immersion is contained in a well defined ghmc Minkowski spacetime. More precisely,
\begin{theorem}[Mess \cite{Mess}]\label{mess mink 23}
\noindent Given a representation, $\theta:\Pi_1\rightarrow\opPSL(2,\R)\ltimes\R^{2,1}$, whose linear component, $\theta_0$, is injective and acts properly discontinuously over $\mathbb{H}^2$, there exists a unique future complete, convex subset, $\Omega_\theta^+$, of Minkowski space which is maximal with respect to inclusion, and over the interior of which $\theta$ acts freely and properly discontinuously.
\end{theorem}
\begin{remark} Here a spacetime is said to be future (resp. past) complete whenever every future oriented (resp. past oriented) causal geodesic can be extended indefinitely.\end{remark}
\begin{remark} In fact, Mess constructs a unique, maximal, invariant convex subset, $\Omega_\theta$, in the projective space, $\opRP^3$. Its intersection with $\Bbb{R}^{2,1}$ then consists of two connected components, one of which, denoted by $\Omega_\theta^+$, is future complete, and the other of which, denoted by $\Omega_\theta^-$,
 is past complete. However, reflection through the origin maps $\Omega^-_\theta$ into
$\Omega^{+}_{\theta'}$,
where $\theta':
\Pi_1\rightarrow\opPSL(2,\Bbb{R})\ltimes\Bbb{R}^{2,1}$
is the homomorphism whose linear part, $\theta'_0$, is equal to $\theta_0$, and whose translation component, $\tau'$, is equal to $-\tau$. For this reason, it will be sufficient in all that follows to consider only the future complete component. \end{remark}
\par
The quotient, $\Omega^+_\theta/\theta(\Pi_1)$, is a future complete, ghmc Minkowski spacetime. Furthermore, since any compact, LSC Cauchy surface in a given ghmc Minkowski spacetime lifts to an equivariant immersion in $\R^{2,1}$, it turns out that every future complete, ghmc Minkowski spacetime actually arises in this manner. In other words, we have constructed two maps which send the space, $\opGHMC_0$, of future-complete, ghmc Minkowski spacetimes into spaces of Teichm\"uller data (c.f. Tables~\ref{table mink 1} and \ref{table mink 2}), and Theorems \ref{mess mink 22} and \ref{mess mink 23} now yield
\begin{table}[h!]
\begin{center}
\begin{tabular}{|c|c|c|}
\hline
\bf Map & \bf Description  &  \bf Codomain \\
\hline
$\Theta_0$ & \multicolumn{1}{|l|}{The linear component of the holonomy} & $\Trep$ \\
\hline
\end{tabular}
\end{center}
\caption{Maps taking values in spaces of real dimension $(6\mathfrak{g}-6)$.}\label{table mink 1}
\end{table}
\begin{table}[h!]
\begin{center}
\begin{tabular}{|c|c|c|}
\hline
\bf Map & \bf Description  &  \bf Codomain \\
\hline $\Theta$ & \multicolumn{1}{|l|}{The entire holonomy} & $\opT\Trep$ \\
\hline
\end{tabular}
\end{center}
\caption{Maps taking values in spaces of real dimension $(12\mathfrak{g}-12)$.}\label{table mink 2}
\end{table}
\begin{theorem}\label{thm:mess bij mink}
\noindent The map, $\Theta$, defines a bijection from $\opGHMC_0$ into $\opT\Trep$.
\end{theorem}
\noindent In particular, this yields a parametrisation of $\opGHMC_0$ by $\opT\Trep$, and it is this parametrisation that is used to furnish $\opGHMC_0$ with the structure of a real algebraic variety.
\subsection{Laminations and trees}\label{LaminationsAndTrees}
The inverse problem of reconstructing the equivariant immersion from its holonomy becomes clearer when our attention shifts from the homomorphism, $\theta$, to its invariant set, $\Omega_\theta^+$. Furthermore, by showing that certain geometric objects associated to $\Omega_\theta^+$ are unique, we obtain new functions from $\opGHMC_0$ into spaces of Teichm\"uller data, which suggest other potential parametrisations for this space. In this section, we consider the non-smooth geometric objects that are associated to $\Omega_\theta^+$. The first is a measured geodesic lamination over the hyperbolic surface, $\mathbb{H}^2/\theta_0(\Pi_1)$. It is constructed using the {\sl generalised Gauss map} of $\Omega_\theta^+$. This is a set-valued function defined as follows. Consider first an arbitrary closed, future complete convex subset, $X$, of $\R^{2,1}$. Given a boundary point, $p$, of $X$, any plane passing through $p$ is said to be a {\sl supporting plane} to $X$ at that point whenever $X$ lies entirely to one side of it. The set, $G(p)$, is then defined to be the set of all future oriented, timelike, unit vectors which are normal to some spacelike supporting plane to $X$ at $p$. In particular, $G(p)$ is a (possibly empty) subset of the future component of the unit pseudosphere, which, we recall, identifies with $\mathbb{H}^2$.
\par
In order to understand the geometry of the generalised Gauss map, it is useful to first consider configurations of null planes. To this end, let $\opC^+$ be the positive light cone,
\begin{equation*}
\opC^+ := \left\{x\ |\ \langle x,x\rangle_{2,1}=0,\ x_3>0\right\}~,
\end{equation*}
and let $\opPC^+$ be its projective quotient. Observe that any geodesic, $c$, in $\mathbb{H}^2$ is the intersection of $\mathbb{H}^2$ with a unique timelike plane, $P$. The intersection of $P$ with $\opC^+$ then defines two distinct rays which project down to two distinct points in $\opPC^+$. By identifying these points with the end points of the geodesic, $c$, we see how $\opPC^+$ naturally identifies with the ideal boundary of $\mathbb{H}^2$.
\par
Now consider a null plane, $P_1$, in $\R^{2,1}$. Its null direction, $N_1$, is the ray in $\opC^+$ given by the intersection of a suitable translate of itself with $C^+$. Furthermore, if $X$ is the half-space corresponding to the future of $P_1$, then for any $p\in P_1$, the image, $G(p)$, of the generalised Gauss map of $X$ at the point, $p$, is empty.
\par
Now let $P_2$ be another null plane which is not parallel to $P_1$. The intersection of these two planes is a complete, spacelike geodesic, $\Gamma$, normal to both $N_1$ and $N_2$. Furthermore, if $X$ now denotes the future complete, convex set determined by the intersections of the respective futures of these two planes, then the image, $G(p)$, of its generalised Gauss map at any point of $\Gamma$ is the complete geodesic in $\mathbb{H}^2$ whose end points at infinity are $N_1$ and $N_2$.
\par
Finally, let $P_3$ be a third null plane which is not parallel to $\Gamma$. The three planes, $P_1$, $P_2$ and $P_3$ then intersect in a single point, $p$, say, and if $X$ now denotes the future complete convex set determined by the intersection of their three respective futures, then the image, $G(p)$, of its generalised Gauss map at the point, $p$, is the ideal triangle in $\mathbb{H}^2$ with end points $N_1$, $N_2$ and $N_3$.
\par
Now, since the invariant set, $\Omega^+_\theta$, is maximal, its generalised Gauss map behaves much like that of a finite configuration of null planes. In particular, its boundary, $\partial\Omega_\theta^+$, is made up of three types of points. When $p$ is a face point, $G(p)$ is empty, when $p$ is an edge point, $G(p)$ is a complete geodesic in $\Bbb{H}^2$, and when $p$ is a vertex point, $G(p)$ is an ideal polygon in $\Bbb{H}^2$, possibly with infinitely many sides. The union of all complete geodesics which are images of edge points now defines a geodesic lamination, $L$, over $\mathbb{H}^2$. Furthermore, since it is invariant under the action of $\theta_0(\Pi_1)$, it projects to a lamination over the surface, $\mathbb{H}^2/\theta_0(\Pi_1)$.
\par
The construction of the transverse measure over this lamination is a bit more subtle, and relies on the observation that any point, $p$, of $\mathbb{H}^2$ not lying in $L$ has a unique preimage in $\partial\Omega^+_\theta$. With this in mind, the mass of any short transverse curve, $c$, with end points not in $L$ is first approximated by the Minkowski distance in $\R^{2,1}$ between the preimages of these two end points, and the mass of an arbitrary curve, $c$, compatible with $L$, is now determined in the usual manner by summing over short segments and taking a limit. Since this transverse measure is also invariant under the action of $\theta_0(\Pi_1)$, we thereby obtain a measured geodesic lamination over the surface, $\mathbb{H}^2/\theta_0(\Pi_1)$. An analogous construction also associates a measured geodesic lamination to the past-complete invariant set, $\Omega_\theta^-$. In this manner, we obtain two maps taking values in a space of Teichm\"uller data of real dimension $(6\mathfrak{g}-6)$ (c.f. Table~\ref{table 3 mink}).
\begin{table}[h!]
\begin{center}
\begin{tabular}{|c|c|c|}
\hline
\bf Map & \bf Description  &  \bf Codomain \\
\hline
$\operatorname{L}^\pm$ & \multicolumn{1}{|l|}{The measured geodesic lamination of $\partial\Omega^\pm_\theta$} & $\opML$\\
\hline
\end{tabular}
\end{center}
\caption{Maps taking values in spaces of real dimension $(6\mathfrak{g}-6)$.}\label{table 3 mink}
\end{table}
\par
In particular, the pair $(\Theta_0,\operatorname{L}^\pm)$ also parametrises the space of future-complete, ghmc Minkowski spacetimes. Indeed,
\begin{theorem}[Mess \cite{Mess}]\label{mess mink 5}
\noindent The map $(\Theta_0,\operatorname{L}^+)$ defines a bijection from $\opGHMC_0$ into $\Trep\times\opML$.
\end{theorem}
In order to visualise how the set $\Omega^+_\theta$ is recovered from the lamination, consider first the case of a single, complete geodesic, $c$, in $\mathbb{H}^2$, weighted by a positive real number, $a$. When $\mathbb{H}^2$ is identified with the unit pseudosphere in $\Bbb{R}^{2,1}$, the geodesic, $c$, is given by its intersection with a timelike plane, $P$, whose normal we denote by $N$. The future convex set, $\Omega^+_{c,a}$, is then constructed by slicing the interior of the future cone, $\opC^+$, along $P$, translating one of the components by a distance, $a$, in the direction of $N$, and filling in the ``V''-shaped region between the two. In fact, this set coincides, up to translation, with the intersection of all the future sides of all lightlike planes which lie in the past of the segment $[0,aN]$. In addition, it is worth observing that, in the same manner that hyperbolic space is the locus of all points in the interior of the future cone, $\opC^+$, which lie at unit distance from the origin, the locus of all points in $\Omega_{c,a}^+$ lying at unit distance from the segment, $[0,aN]$, coincides with a grafting of a cylinder of length, $a$, along the geodesic, $c$, as outlined in Section~\ref{EarthquakesAndGraftings}. In particular, as with graftings, this construction extends continuously to the space of all measured geodesic laminations, thereby yielding Theorem~\ref{mess mink 5}.
\par
Consider now the pair $(L^-,L^+)$. We show that this map defines a partial parametrisation of $\opGHMC_0$. First, let $\opFuc_0$ denote the space of {\sl Fuchsian} ghmc Minkowski spacetimes, that is, those spacetimes in $\opGHMC_0$ whose holonomy has vanishing translation component. Observe that, for all such spacetimes, the invariant set, $\Omega_\theta$, is simply the interior of the future cone, $C^+$, and so the corresponding measured geodesic lamination is trivial. In particular, the whole of $\opFuc_0$ is mapped by $(L^-,L^+)$ to the same point, $(0,0)$, of $\opML\times\opML$. We now describe the image of this map. We say that a pair $(\lambda^-,\lambda^+)$, of measured geodesic laminations {\sl fills} $S$ whenever every connected component of the complement of the union of their respective supports lifts to a bounded polygon in $\Bbb{H}^2$. Equivalently (c.f. \cite{ser12}), $(\lambda^-,\lambda^+)$ fills $S$ whenever there exists $\epsilon>0$ such that
\begin{equation*}
M^-(\langle\gamma\rangle) + M^+(\langle\gamma\rangle) > \epsilon~,
\end{equation*}
for every non-trivial free homotopy class, $\langle\gamma\rangle$, in $\langle\Pi_1\rangle$, where $M^-$ and $M^+$ denote respectively the mass functions of $\lambda^-$ and $\lambda^+$. We denote by
$$
\opML\times_{\operatorname{fill}}\opML
$$
the subset of $\opML\times\opML$ consisting of those pairs of laminations which fill $S$.
\begin{theorem}[Bonsante--Schlenker \cite{BS12}]
\noindent The map, $(\opL^+,\opL^-)$ defines a bijection from $\opGHMC_0\setminus\opFuc_0$ into $\opML\times_{\operatorname{fill}}\opML$.
%
\end{theorem}
Finally, consider the mass function, $M_\lambda$, of some measured geodesic lamination, $\lambda$. Recall that $M_\lambda$ maps free homotopy classes in $\langle\Pi_1\rangle$ into $[0,\infty[$. Now, consider a minimal, short action\footnote{Recall that an action of $\Pi_1$ on a real tree, $T$, is said to be minimal whenever it contains no invariant proper subtree, and is said to be short whenever the stabilizer of any isometric copy of $\Bbb{R}$ in $T$ is abelian.}, $\alpha$, of $\Pi_1$ on a real tree, $T$ (c.f. \cite{CM}), and let $D_\alpha:\langle\Pi_1\rangle\rightarrow[0,\infty[$ be its displacement function, that is, for every conjugacy class, $\langle\gamma\rangle$, in $\langle\Pi_1\rangle$, $D_\alpha(\langle\gamma\rangle)$ is the infimal displacement, $d(x,\alpha(\gamma)\cdot x)$, as $x$ ranges over all points, $x$, of $T$. Recall (c.f. \cite{hat88}), that, as with measured geodesic laminations and their mass functions, minimal, short actions of $\Pi_1$ on real trees are also uniquely defined up to isometry by their displacement functions. Now, although it is not known which functions arise as mass functions of measured geodesic laminations or as displacement functions of minimal, short actions on real trees, it is a remarkable fact (c.f. \cite{skora}), that the two coincide, that is, a function, $f:\langle\Pi_1\rangle\rightarrow[0,\infty[$, is the mass function of some measured geodesic lamination if and only if it is the displacement function of some minimal, short action on some real tree. In this manner, we obtain a natural duality between the space of measured geodesic laminations and the space of minimal, short actions on real trees which associates to a given measured geodesic lamination, $\lambda$, the unique minimal, short action, $\alpha$, whose displacement function is equal to the mass function of $\lambda$.
\par
It turns out that the minimal, short action on a real tree is, in fact, the easiest object to visualise in the above construction. Indeed, consider again the three planes, $P_1$, $P_2$ and $P_3$, introduced above. Let $X_{12}$ denote the convex set defined by the intersections of the respective futures of $P_1$ and $P_2$, and define the pseudometric, $d$, over its boundary, $\partial X_{12}$, by
\begin{equation*}
d(x,y) = \inf_{\gamma}l(\gamma)~,
\end{equation*}
where $\gamma$ ranges over all continuous curves in $\partial X_{12}$ starting at $x$ and ending at $y$, and $l(\gamma)$ denotes its length with respect to the Minkowski metric. It is a straightforward matter to show that the metric space, $\partial X_{12}/\sim$, obtained by identifying points separated by zero distance is naturally isometric to the geodesic, $\Gamma$, defined by the intersection of $P_1$ with $P_2$. Furthermore, performing the same construction on the boundary of the set, $X_{123}$, defined by the intersection of the respective futures of $P_1$, $P_2$ and $P_3$, yields a metric space, $\partial X_{123}/\sim$, which is isometric to the union of the edges of the polyhedron, $\partial X_{123}$. More generally, applying this construction to the boundary of $\Omega_\theta^+$ yields a real tree, over which $\theta(\Pi_1)$ acts in a minimal, short manner.
\par
We therefore obtain another map taking values in a space of Teichm\"uller data of real dimension $(6\mathfrak{g}-6)$, even though it is, in fact, completely equivalent to the function, $\operatorname{L}$, defined above (c.f. Table~\ref{table 4}).
\begin{table}[h!]
\begin{center}
\begin{tabular}{|c|c|c|}
\hline
\bf Map & \bf Description  &  \bf Codomain \\
\hline
$\opT$ & \multicolumn{1}{|l|}{The minimal, short action of $\theta$ on the real tree $\partial\Omega^+_\theta/\sim$} & $\opRT$\\
\hline
\end{tabular}
\end{center}
\caption{Maps taking values in spaces of real dimension $(6\mathfrak{g}-6)$. Here $\opRT$ denotes the space of minimal, short actions of $\Pi_1$ on real trees.}\label{table 4}
\end{table}
\subsection{Smooth parametrisations}\label{SmoothParametrisations}
Our starting point for constructing smooth parametrisations of $\opGHMC_0$ is the following result.
\begin{theorem}[Barbot--B\'eguin--Zeghib \cite{BBZ}]
\label{BBZmink}
\noindent Let $\theta:\pi_1(S)\rightarrow\opPSL(2,\R)\ltimes\R^{2,1}$ be a homomorphism whose linear component, $\theta_0$, is injective and acts properly discontinuously over $\mathbb{H}^2$. For all $\kappa>0$, there exists a unique smooth, spacelike, LSC surface, $\Sigma_\kappa$, which is embedded in $\Omega_\theta^+$, is invariant under the action of $\theta$, and has constant extrinsic curvature equal to $\kappa$. Furthermore, the family of all such surfaces foliates $\Omega_\theta^+$ as $\kappa$ varies over the interval $]0,\infty[$.
\end{theorem}
The foliations constructed here yield various families of maps taking values in spaces of Teichm\"uller data. Indeed, for $\kappa>0$, consider the space-like, LSC, embedded surface, $\Sigma_\kappa$, in $\Omega_\theta^+$ furnished by Theorem~\ref{BBZmink}, and let $I_\kappa$, $\II_\kappa$ and $\III_\kappa$ be its first, second and third fundamental forms respectively. The form, $\kappa I_\kappa$, defines a marked hyperbolic metric over $\Sigma_\kappa$, thereby yielding a point in $\Thyp$. The form, $\III_\kappa$, defines another marked hyperbolic metric over $\Sigma_\kappa$. However, this metric actually concides with the pull back through the Gauss map of the metric over $\mathbb{H}^2$, and since the Gauss map is equivariant with respect to the linear component, $\theta_0$, of the holonomy, the point that it defines in Teichm\"uller space actually coincides with the point already defined by the map, $\Theta_0$, given in Section~\ref{ads 1}, above. Next, by local strict convexity, the form, $\II_\kappa$, also defines a marked metric over $\Sigma_\kappa$, but since this metric has no clear curvature properties, we consider it rather as defining a marked holomorphic structure in $\Thol$. The same constructions also apply to the past complete invariant set, $\Omega_\theta^-$. We therefore have two pairs of maps, each taking values in spaces of Teichm\"uller data of real dimension $(6\mathfrak{g}-6)$ (c.f. Table~\ref{table 5}).
\begin{table}[h!]
\begin{center}
\begin{tabular}{|c|c|c|}
\hline
\bf Map & \bf Description  &  \bf Codomain \\
\hline
$\I_\kappa^\pm$ & \multicolumn{1}{|l|}{The constant curvature metric of $I_\kappa$ in $\Omega^{\pm}_\theta$} & $\Thyp$\\
$\operatorname{II}_\kappa^\pm$& \multicolumn{1}{|l|}{The holomorphic structure of $\II_\kappa$ in $\Omega^{\pm}_\theta$}& $\Thol$ \\
\hline
\end{tabular}
\end{center}
\caption{Maps taking values in spaces of real dimension $(6\mathfrak{g}-6)$.}\label{table 5}
\end{table}
\par
These maps are complemented to maps taking values in spaces of Teichm\"uller data of real dimension $(12\mathfrak{g}-12)$ as follows. First, the shape operator, $A_\kappa$, of $\Sigma_\kappa$ defines a Labourie field of $I_\kappa$, so that the pair $(I_\kappa,A_\kappa)$ yields a point of $\opLab\Thyp$. Likewise, the Hopf differential, $\phi_\kappa$, of $I_\kappa$ with respect to the conformal structure of $\II_\kappa$ is a quadratic holomorphic differential, so that the pair $(\II_\kappa,\phi_\kappa)$ yields a point of $\opT^*\Thol$. In summary, we have two maps taking values in spaces of Teichm\"uller data of real dimension $(12\mathfrak{g}-12)$ (c.f. Table~\ref{ads 1}).
\begin{table}[h!]
\begin{center}
\begin{tabular}{|c|c|c|}
\hline
\bf Map & \bf Description  &  \bf Codomain \\
\hline
$\opA_\kappa$ &  \multicolumn{1}{|l|}{The metric $I_\kappa$  together with the Labourie field $A_\kappa$} & $\opLab\Thyp$ \\
$\Phi_\kappa$&
\multicolumn{1}{|l|}{\pbox{20cm}{The holomorphic structure of $\II_\kappa$ in $\Omega^+_\theta$ \\  together with the Hopf differential of $I_\kappa$}} & $\opT^*\Thol$ \\
\hline
\end{tabular}
\end{center}
\caption{Maps taking values in spaces of real dimension $(12\mathfrak{g}-12)$.}\label{table 6}
\end{table}
\begin{theorem}\label{Lab mink}
\noindent For all $\kappa$, the map, $\opA_\kappa$, defines a real analytic diffeomorphism from $\opGHMC_0$ into $\opLab\Thyp$.
\end{theorem}
\begin{proof}[Sketch of proof] Consider a hyperbolic metric, $g$, and a Labourie field, $A$. By the fundamental theorem of surface theory (Theorem~\ref{FTS}), there exists a locally strictly convex equivariant immersion $e:(\tilde{S},\kappa^{-1}g)\rightarrow\R^{2,1}$ ,with shape operator equal to $\sqrt{\kappa}A$, which is unique up to isometries of $\R^{2,1}$. This yields a real analytic inverse of $(\opI_\kappa,\opA_\kappa)$, and the result follows.\end{proof}
\begin{theorem}\label{mink II}
\noindent For all $\kappa$, $\Phi_{\kappa}$ defines a real analytic diffeomorphism from $\opGHMC_0$ into $\opT^*\Thol$.
\end{theorem}
\begin{proof}[Sketch of proof] Indeed, observe that $\Phi_\kappa\circ \opA_\kappa^{-1}$ coincides with $\Phi\circ\mathcal{A}^{-1}$, where $\Phi$ and $\mathcal{A}$ are defined as in Sections \ref{HopfDifferentials} and \ref{LabourieFields} respectively. The result now follows by Theorems \ref{wolf}, \ref{AIsARealAnalDiff} and \ref{Lab mink}.\end{proof}
Certain pairs of data taking values inside spaces of Teichm\"uller data of real dimension $(6\mathfrak{g}-6)$ also yield parametrisations of $\opGHMC_0$. Indeed,
\begin{theorem}[Labourie]
\noindent For all $\kappa$, and for all $\epsilon\in\left\{-,+\right\}$, $(\I_\kappa^\epsilon,\Theta_0)$ defines a bijection from $\opGHMC_0$ into $\Thyp\times\Trep$.
\end{theorem}
\begin{proof} First observe that, by the discussion following Theorem \ref{BBZmink}, $(\opI_\kappa^\epsilon,\Theta_0)$ identifies with $(\opI_\kappa^\epsilon,\opIII_\kappa^\epsilon)$. However, the maps $(\opI_\kappa^\epsilon,\opIII_\kappa^\epsilon)\circ \opA_\kappa^{-1}$ is precisely the inverse of the map given by the Labourie field theorem (Theorem \ref{LabFieldthm}), and the result follows.\end{proof}
\begin{theorem}[Labourie]
\noindent For all $\kappa$, and for all $\epsilon\in\left\{-,+\right\}$, $(\I_\kappa^\epsilon,\operatorname{II}_\kappa^\epsilon)$ defines a real analytic diffeomorphism from $\opGHMC_0$ into $\Thyp\times\Thol$.
\end{theorem}
\begin{proof} The map $(\opI_\kappa^\epsilon,\opII_\kappa^\epsilon)\circ \opA_\kappa^{-1}$ coincides with $\mathcal{A}^{-1}$, where $\mathcal{A}$ is the map given in Section \ref{LabourieFields}. The result now follows by Theorems \ref{AIsARealAnalDiff} and \ref{Lab mink}.\end{proof}
Finally, we remark that the classical Weyl problem in Minkowski space, as yet unresolved at the time of writing, also admits a straightforward expression in the current framework. Indeed,
\begin{question}[Weyl problem]
\noindent For $\kappa,\kappa'>0$, does $(\I_\kappa^+,\I_{\kappa'}^{-})$ define a bijection from $\opGHMC_0$ into $\Thyp\times\Thyp$?
\end{question}
Other data can also be extracted using different existence results. However, these are not necessarily so regular, or do not necessarily interact so well with the invariant set, $\Omega_\theta^+$. For example,  foliations can be constructed using surfaces of constant mean curvature, although the foliation thereby obtained is not necessary contained in $\Omega_\theta^+$ (c.f., for example, \cite{ABBZ} and the references therein). Likewise, in \cite{fv}, foliations are constructed using surfaces of constant mean radius of curvature although, in this case, the surfaces constructed are not necessarily even smooth. Nonetheless, the unique surface of zero mean radius of curvature can be used to construct a bijection between $\opGHMC_0$ and the space of traceless Codazzi tensors on $S$, which in turn canonically identifies with $\operatorname{TT}_{\operatorname{rep}}[S]$ (c.f. \cite{lafontaine,BS}).
%
%
\section{Anti de Sitter space}\label{sec:ads}
\subsection{Definition of anti de Sitter space}\label{ads 1}
{\sl Anti de Sitter space}, which we denote by $\opAdS^3$, is defined to be the projective quotient of the unit pseudosphere in $\R^{2,2}$, that is
\begin{equation*}
\opAdS^3 := \left\{ x\ |\ x_1^2 + x_2^2 - x_3^2 - x_4^2 = 1\right\}/\left\{\pm\operatorname{Id}\right\}~.
\end{equation*}
This space carries a natural group structure which plays an important role in its study. This structure can be visualised by introducing the matrix
\begin{equation*}
J=
\begin{pmatrix} & -1 \\ 1 & \end{pmatrix},
\end{equation*}
and defining the bilinear form over $\opEnd(\R^2)$ by
\begin{equation*}
\langle A,B\rangle_{2,2} = -\frac{1}{2}\opTr(A^t J B J)~.
\end{equation*}
Indeed, since this form is non-degenerate with signature $(2,2)$, and since, for every matrix, $A$,
\begin{equation*}
\langle A,A\rangle_{2,2} = \opDet(A)~,
\end{equation*}
it follows that $\opAdS^3$ naturally identifies with the group $\opPSL(2,\R)$.
\par
The form $\langle\cdot,\cdot\rangle_{2,2}$ restricts to a $(2,1)$-form over the Lie algebra, $\opsl(2,\R)$, so that the latter identifies with Minkowski space, which was studied in detail in the preceding chapter. Furthermore, a straightforward calculation reveals that the unit pseudosphere in $\opsl(2,\R)$ coincides with the locus of all matrices, $M$, whose square is equal to $(-\operatorname{Id})$. In particular, the matrix $J$ itself is an element of this unit pseudosphere, and therefore distinguishes a preferred component which we identify with $\mathbb{H}^2$. By considering this as the future component, we see that $J$ defines a time orientation over $\opAdS^3$. We also use this matrix to define a spatial orientation. Indeed, the multiplications by $J$ on the left and on the right define two right angled rotations of the tangent space of $\mathbb{H}^2$ at the point $J$. Choosing multiplication on the left, we thus obtain a spatial orientation of $\mathbb{H}^2$ at this point, which extends uniquely to the whole of $\opAdS^3$.
\par
The isometry group of $\opAdS^3$ coincides with $\opPO(2,2)$, that is, the projective quotient of $\opO(2,2)$. This group has two connected components, determined by whether they preserve or reverse the orientation, and its identity component identifies with the Cartesian product, $\opPSL(2,\R)\times\opPSL(2,\R)$. Indeed, observe first that $\opPSL(2,\R)$ acts on itself transitively and isometrically by multiplication on the left and on the right. This yields a natural embedding of $\opPSL(2,\R)\times\opPSL(2,\R)$ into $\opPO(2,2)$ which sends the pair $(M,N)$ to the map $A\mapsto MAN^{-1}$. Consider now the stabiliser subgroup of $\Id$ in the identity component of $\opPO(2,2)$. Since this subgroup acts by orientation preserving isometries which send the $J$-component of the unit pseudosphere to itself, it actually coincides with the adjoint action of $\opPSL(2,\R)$ on $\opsl(2,\mathbb{R})$. The assertion now follows, since this adjoint action is none other than the action of the diagonal subgroup of $\opPSL(2,\R)\times\opPSL(2,\R)$ on this Lie algebra.
\par
Consider now a spacelike, LSC, equivariant immersion, $[e,\theta]$, in $\opAdS^3$. By the preceeding discussion, the holonomy, $\theta$, decomposes as $\theta=:(\theta_l,\theta_r)$, where each of $\theta_l$ and $\theta_r$ are homomorphisms into $\opPSL(2,\R)$. In order to understand the properties of these two homomorphisms, we introduce the future oriented, unit, normal vector field, $N_e$, over $e$. By composing this vector field with $e^{-1}$, we obtain the {\sl left} and {\sl right Gauss maps},
\begin{eqnarray*}
\ &&N_{l,e}(p) := e(p)^{-1}N_e(p)~,\ \text{and}\\
\ &&N_{r,e}(p) := N_e(p)e(p)^{-1}~.\end{eqnarray*}
These maps both take values in the subspace, $\mathbb{H}^2$, of $\opsl(2,\R)$. Furthermore, since the immersion, $e$, is spacelike and LSC, they are both local diffeomorphisms, and a straightforward topological argument shows that they define global diffeomorphisms from the universal cover, $\tilde{S}$, of $S$ into $\mathbb{H}^2$ which are each equivariant with respect to the respective actions by conjugation of $\theta_r$ and $\theta_l$ on $\opsl(2,\R)$. That is, for all $\gamma\in\Pi_1$, and for all $p\in\tilde{S}$,
\begin{eqnarray*}
\ &&N_{l,e}(\gamma\cdot p) = \theta_r(\gamma)N_{l,e}(p)\theta_r(\gamma)^{-1},\ \text{and}\\
\ &&N_{r,e}(\gamma\cdot p) = \theta_l(\gamma)N_{r,e}(p)\theta_l(\gamma)^{-1}~.\end{eqnarray*}
We thus obtain,
\begin{theorem}[{Mess \cite{Mess}}]\label{messAdS1}
\noindent The homomorphisms, $\theta_l$ and $\theta_r$, are injective and their images act properly discontinuously on $\mathbb{H}^2$, so that the holonomy, $\theta$, defines a point of $\Trep\times\Trep$.
\end{theorem}
\noindent In particular, since two homomorphisms, $\theta$ and $\theta'$, lie in the same conjugacy class if and only if the same holds for their respective left and right components, the point $([\theta_l],[\theta_r])$ of $\Trep\times\Trep$ determined by $\theta$ actually depends on the equivalence class, $[e,\theta]$, only.
\par
As in the Minkowski case, the converse problem of recovering the equivariant immersion from the holonomy is easiest understood once we see that every equivariant immersion is actually contained inside a well defined ghmc AdS spacetime. However, to properly state the result, we first require a good notion of convexity for subsets of $\opAdS^3$. To this end, consider $3$-dimensional, real projective space, $\opRP^3$. It follows from the definition that $\opAdS^3$ projects diffeomorphically onto an open subset of $\opRP^3$. Recall now that a {\sl homogeneous, convex cone}, $\Lambda$, in $\R^{2,2}$ is a (closed) convex subset which is invariant under the action of the dilatation group. A closed subset of $\opRP^3$ is then said to be {\sl convex} whenever it is the projective quotient of some homogeneous, convex cone, and a closed subset of $\opAdS^3$ is said to be {\sl convex} whenever it is convex in $\opRP^3$.
\par
There is also a notion of duality for convex subsets of $\opRP^3$ which is constructed in terms of the metric, $\langle\cdot,\cdot\rangle_{2,2}$, and which plays an important role in what follows. Indeed, given a homogeneous, convex cone, $\Lambda$, in $\R^{2,2}$, its {\sl dual cone} is defined by
\begin{equation*}
\Lambda^* := \left\{ x\in\R^{2,2}\ |\ \langle x,y\rangle_{2,2}\leq 0\ \forall y\in\Lambda\right\}~.
\end{equation*}
$\Lambda^*$ is another homogeneous, convex cone, whose own dual yields $\Lambda$ again. This concept of duality also applies to convex subsets of projective space, since these are equivalent to homogeneous convex cones,
\par
We now have,
\begin{theorem}[Mess \cite{Mess}]\label{messAdS2}
\noindent Given two homomorphisms, $\theta_l,\theta_r:\pi_1(S)\rightarrow\opPSL(2,\R)$, which are injective and which act properly discontinuously on $\mathbb{H}^2$, there exists a unique convex subset, $\Omega_\theta$, of $\opAdS^3$ which is maximal with respect to inclusion, and over the interior of which $\theta:=(\theta_l,\theta_r)$ acts freely and properly discontinuously.
\end{theorem}
In contrast to the Minkowski case, the set $\Omega_\theta$ is neither future nor past complete. However, by maximality, it contains its dual, which we refer to as its {\sl Nielsen kernel}, and which we denote by $\opK_\theta$. The complement of $\opK_\theta$ in $\Omega_\theta$ consists of $2$ connected components, one, denoted by $\Omega_\theta^+$, which lies in the future, and another, denoted by $\Omega_\theta^-$, which lies in the past (c.f. Figure~\ref{fig:ads}).
\begin{figure}
\begin{center}
\includegraphics[scale=0.9]{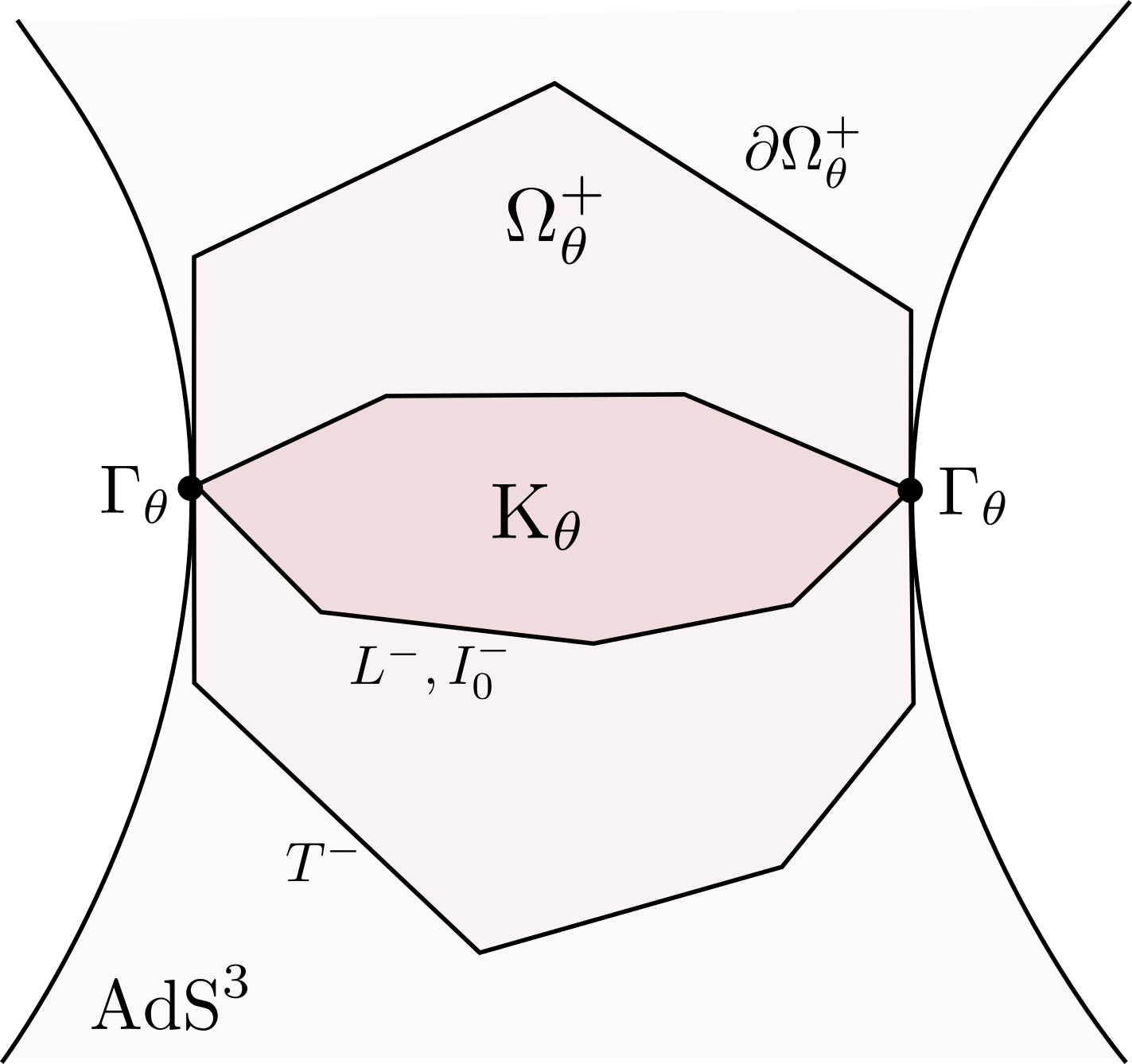}\caption{\label{fig:ads} \textbf{The structure of a ghmc AdS spacetime.}\ Anti de Sitter space can be realised in a projective chart as the interior of a one-sheeted hyperboloid in $\Bbb{R}^3$ plus a disk at infinity. $\Omega_\theta$ and $\opK_\theta$ both meet the boundary of anti de Sitter space along a Jordan curve, $\Gamma_\theta$, and $\opK_\theta$ is, in fact, the convex hull of this curve. The intrinsic metric of $\partial\opK_\theta^\pm$ is hyperbolic, and its singular set defines a measured geodesic lamination. The intrinsic metric of $\partial\Omega_\theta^\pm$ is a real tree, over which $\Pi_1$ acts in a minimal, short manner.}
\end{center}
\end{figure}
\par
The quotient, $\Omega_\theta/\theta(\Pi_1)$, is a ghmc AdS spacetime. Furthermore, since any compact, LSC Cauchy surface in a given ghmc AdS spacetime lifts to an equivariant immersion in $\opAdS^3$, we see that every ghmc AdS spacetime arises in this manner. In other words, we have two maps which send the moduli space, $\opGHMC_{-1}$, of ghmc AdS spacetimes into spaces of Teicm\"uller data of real dimension $(6\mathfrak{g}-6)$ (c.f. Table~\ref{ads table 1}). Theorems \ref{messAdS1} and \ref{messAdS2} therefore yield
\begin{table}[h!]
\begin{center}
\begin{tabular}{|c|c|c|}
\hline
\bf Map & \bf Description  &  \bf Codomain \\
\hline
$\Theta_l$& \multicolumn{1}{|l|}{The left component of the holonomy} & $\Trep$ \\
$\Theta_r$& \multicolumn{1}{|l|}{The right component of the holonomy} & $\Trep$ \\
\hline
\end{tabular}
\end{center}
\caption{Maps taking values in spaces of real dimension $(6\mathfrak{g}-6)$.}\label{ads table 1}
\end{table}
\begin{theorem}[Mess, \cite{Mess}]\label{messAdS3}
\noindent The map $(\Theta_l,\Theta_r)$ defines a bijection from $\opGHMC_{-1}$ into $\Trep\times\Trep$.
\end{theorem}
\noindent In particular, it is this parametrisation of $\opGHMC_{-1}$ which we use to furnish this space with the structure of a real algebraic variety.
\subsection{Mess' construction}\label{MessConstruction}
It is worth reviewing Mess' construction in some detail for the insight it provides into the structure of anti de Sitter space. First, let $\hat{\opC}$ be the light cone in $\R^{2,2}$, that is
\begin{equation*}
\hat{\opC} := \left\{ x\ |\ x_1^2 + x_2^2 = x_3^2 + x_4^2\right\}~,
\end{equation*}
and let $\opC$ be its projective quotient in $\opRP^3$. This subset is a torus whose complement consists of two connected components, one of which we have already identified with $\opAdS^3$, and the other of which identifies with a copy of $\opAdS^3$ whose semi-riemannian metric has reversed sign. In particular, the subset, $\opC$, can be viewed as the boundary at infinity of $\opAdS^3$.
\par
The lorentzian metric over $\R^{2,2}$ restricts to a degenerate metric over $\hat{\opC}$, having one null direction, namely, the radial direction, one positive direction, and one negative direction. However, since the radial direction is collapsed by projection onto $\opRP^3$, this degenerate metric projects down to a conformal class of non-degenerate metrics of signature $(1,1)$ over the torus, $\opC$. That is, $\opC$, carries a natural conformal Minkowski structure.
\par
The conformal Minkowski structure of $\opC$ yields two smooth, transverse distributions over this torus, given at each point by its two lightlike directions. Furthermore, these distributions integrate into two transverse foliations of $\opC$ by lightlike circles. Indeed, recalling first that $\R^{2,2}$ identifies with $\opEnd(\R^2)$, consider a non-zero, lightlike vector, $A$, in $\hat{\opC}$, together with the two planes, $\langle A,JA\rangle$, with generators $A$ and $JA$, and $\langle A, AJ\rangle$, with generators $A$ and $AJ$. These two planes consist of lightlike vectors only, and thus project onto circles in $\opC$, which meet transversally at a unique point, namely the projective image of $A$. The two foliations are now obtained by taking the families of all such circles. We call the foliation defined by planes of the form, $\langle A,JA\rangle$, the {\sl left foliation} of $\opC$, and we call the other foliation its {\sl right foliation}.\footnote{We observe in passing that the operation of matrix transposition exchanges these two foliations.} In particular, if $S_l$ and $S_r$ are fixed leaves of the left and right foliations respectively, then there is a canonical, smooth parametrisation of $C$ by the cartesian product, $S_l\times S_r$.
\par
An alternative and pleasing presentation of these foliations also arises upon consideration of the affine chart of $\opRP^3$ obtained by projecting radially onto the subspace, $E:=\left\{x_1=1\right\}$. Indeed, in this chart, $\opAdS^3$ projects onto the region, $\left\{x_2^2 - x_3^2 - x_4^2<1\right\}$, and $\hat{\opC}$ projects onto the one sheeted hyperboloid,
\begin{equation*}
\left\{x_2^2 - x_3^2 - x_4^2 =-1\right\}~.
\end{equation*}
Since lightlike planes in $\mathbb{R}^{2,2}$ project onto lightlike lines in $E$, it now follows that the double foliation of $\opC$ projects onto the classical double ruling of the one sheeted hyperboloid in $\mathbb{R}^{2,1}$ by lightlike lines.
\par
Consider again the parametrisation of $\opC$ by $S_l\times S_r$. The circles, $S_l$ and $S_r$ are the respective projective spaces of the planes $\langle A,AJ\rangle$ and $\langle B,JB\rangle$, for some non-trivial degenerate matrices, $A$ and $B$. Consider now the action of $\opPSL(2,\R)\times\opPSL(2,\R)$ on $\opC$, and, in particular, its effect on the left and right foliations. A straightforward calculation shows that multiplication on the left by an element of $\opPSL(2,\R)$ preserves each leaf of the left foliation, and thus leaves the right coordinate in $S_l\times S_r$ invariant. Likewise, multiplication on the right preserves each leaf of the right foliation, thereby leaving the left coordinate in $S_l\times S_r$ invariant. In this manner, we see how the parametrisation of $\opC$ by $S_l\times S_r$ separates the two components of $\opPSL(2,\R)\times\opPSL(2,\R)$.
\par
Furthermore, let $(e_1,e_2)$ be the canonical basis of $\R^2$, and let $\Phi:\R^2\rightarrow\langle A,JA\rangle$ be the linear map given by
\begin{eqnarray*}
\ &&\Phi(e_1) := A~,\ \text{and}\\
\ &&\Phi(e_2) := JA~.
\end{eqnarray*}
For a given element, $M$, of $\opPSL(2,\R)$, denoting by $M_L$ its action by multiplication on the left on the invariant plane, $\langle A,JA\rangle$, we readily obtain
\begin{equation*}
\Phi^{-1} M_L\Phi = M~.
\end{equation*}
In other words, $\Phi$ defines a projective linear diffeomorphism from $\opRP^1$ into $S_l$, which conjugates the action of $\opPSL(2,\R)$. The same observation also holds for $S_r$, so that the product, $S_l\times S_r$, actually identifies with $\opRP^1\times\opRP^1$ over each component of which $\opPSL(2,\R)$ acts in the usual manner.
\par
The completion of Mess' construction is now straightforward. Indeed, suppose that $\theta_l$ and $\theta_r$ satisfy the hypotheses of Theorem~\ref{messAdS2}. In particular, they both extend to homeomorphisms, $\hat{\theta}_l:\partial_\infty\Pi_1\rightarrow S_l$ and $\hat{\theta}_r:\partial_\infty\Pi_1\rightarrow S_r$. The image of $(\hat{\theta}_l,\hat{\theta}_r)$ is a Jordan curve, $\Gamma_\theta$, in $\opC$, which actually coincides with the closure of the set of all attractive and repulsive fixed points of $\theta(\gamma)$, as $\gamma$ ranges over all non-trivial elements of $\Pi_1$. Furthermore, since it is a graph over each of $S_l$ and $S_r$, it is spacelike. The curve, $\Gamma_\theta$, has a well defined convex hull, $\opK_\theta$, in $\opRP^3$, which is, in fact, contained in $\opAdS^3$, and its dual, $\Omega_\theta$, is the desired maximal, invariant convex set.
\subsection{Laminations and trees}\label{LaminationsAndEarthquakes}
We now turn our attention to the non-smooth geometric objects associated to $\Omega_\theta$. These will be contained in the union of $\Omega_\theta^+$ and $\Omega_\theta^-$. Furthermore, since the time orientation is readily reversed --- say, by replacing the matrix, $J$, with $(-J)$ --- it suffices for the moment to consider only the future component, $\Omega_\theta^+$. The boundary of this set itself consists of two components. The first, which we denote by $\partial\Omega_\theta^+$, is the intersection of $\Omega_\theta^+$ with $\partial\Omega_\theta$, and the second, which we denote by $\partial \opK_\theta^+$, is the intersection of $\Omega_\theta^+$ with $\partial \opK_\theta$.
\par
The measured geodesic lamination of $\Omega_\theta^+$ is determined by the boundary component, $\partial\opK_\theta^+$. In order to visualise the construction, consider first a timelike, unit vector, $B_1$, in $\R^{2,2}$. Observe that $B_1$ projects to an element of $\opAdS^3$, which we also denote by $B_1$. Furthermore, its orthogonal complement defines a spacelike, totally geodesic, embedded submanifold of $\opAdS^3$ isometric to $\mathbb{H}^2$. In particular, at every point, $A$, of this subset, the vector $B_1$ also defines a future oriented, unit, timelike tangent vector to $\opAdS^3$ which is orthogonal to this submanifold at this point. By abuse of terminology, we denote this subspace by $B_1^\perp$ and we refer to it is the {\sl spacelike plane} orthogonal to $B_1$.
\par
Consider now another element, $B_2$, of $\opAdS^3$, chosen such that $B_1^\perp$ and $B_2^\perp$ intersect along a shared, spacelike geodesic, $\Gamma$, which, in particular, divides each of $B_1^\perp$ and $B_2^\perp$ into two half-spaces with geodesic boundaries. Let $X_{12}$ be the intersection of the respective pasts of $B_1^\perp$ and $B_2^\perp$.\footnote{Technically, since $\opAdS^3$ is not causal, the past of a given spacelike plane, $P$, is not globally defined. However, since we are only concerned with what happens in a neighbourhood of $P$, this is resolved in one of two different ways: either by working in the universal cover of $\opAdS^3$, which is causal, or by working in a small causal neighbourhood of $P$.} The boundary, $\partial X_{12}$, of $X_{12}$ consists of two spacelike, totally geodesic components which meet along $\Gamma$, one of which is a half-space in $B_1^\perp$, and the other of which is a half-space in $B_2^\perp$. Recall that the intrinsic metric of $\partial X_{12}$ is defined by
$$
d(x,y) := \inf_{\gamma}l(\gamma)~,
$$
where $\gamma$ ranges over all continuous curves in $\partial X_{12}$ starting at $x$ and ending at $y$, and $l(\gamma)$ denotes its length with respect to the Minkowski metric, $\langle\cdot,\cdot\rangle_{2,2}$. It is now a straightforward matter to show that $\partial X_{12}$, furnished with this metric, is isometric to $\mathbb{H}^2$, and that the bending locus, $\Gamma$, is a complete geodesic in this space.
\par
Consider now a third element, $B_3$, of $\opAdS^3$, chosen in such a manner that the spacelike planes, $B_1^\perp$, $B_2^\perp$ and $B_3^\perp$ have trivial intersection. Let $X_{123}$ be the intersection of the respective pasts of these three planes. As before, $\partial X_{123}$, furnished with its intrinsic metric, is isometric to $\mathbb{H}^2$, and there will be two of the three curves, $\Gamma_{12}$, $\Gamma_{13}$ and $\Gamma_{23}$, which define complete, non-intersecting geodesics in this space.
\par
Now, since $\opK_\theta$ is a convex hull, it behaves much like the intersection of the respective pasts of a finite configuration of spacelike planes, no three of which share a common point in $\opAdS^3$. In particular, $\partial \opK_\theta^+$, furnished with its intrinsic metric, is isometric to $\mathbb{H}^2$, and since it is invariant under the action of $\theta(\Pi_1)$, it defines a compact, hyperbolic surface, $\partial \opK_\theta^+/\theta(\Pi_1)$.
\par
The measured geodesic lamination is now constructed using supporting planes. Indeed, given a convex set, $X$, and a boundary point, $A$, a spacelike plane, $B^\perp$, passing through $A$ is said to be a {\sl supporting plane} to $X$ at that point whenever $X$ lies entirely to one side of it. Once again, since $\opK_\theta$ is a convex hull, any supporting plane to its boundary, $\partial\opK_\theta^+$, meets this set, either along a complete geodesic, or along an ideal polygon. The lamination, $L$, of $\partial\opK_\theta^+$ is then defined to be the union of all complete geodesics determined by intersections of $\partial\opK_\theta^+$ with supporting hyperplanes. As in the case of a finite configuration of spacelike planes, no two of these geodesics intersect in $\partial\opK_\theta^+$, so that this set is indeed a lamination.
\par
Given a convex subset, $X$, of $\opAdS^3$, and a boundary point, $A$, an element, $B$, of $\opAdS^3$ is said to be a {\sl supporting normal} to $X$ at $A$ whenever the spacelike plane, $B^\perp$, is a supporting plane to $X$ at that point. It is a straightforward matter to show that if $A$ is a point of $\partial\opK_\theta^+$ not lying on the lamination, $L$, then $\partial\opK_\theta^+$ has a unique supporting normal at that point. Now, given a short curve, $c$, in $\partial\opK_\theta^+$, with end points not in $L$, its mass is approximated by the length of the shortest spacelike curve in $\opAdS^3$ joining the respective supporting normals of these two end points, and the mass of an arbitrary curve, $c$, compatible with $L$, is now determined in the usual manner by summing over short segments and taking a limit. In this way, we define a transverse measure over $L$ which yields a measured geodesic lamination over $\partial\opK_\theta^+$. In particular, since it is trivially invariant under the action of $\theta(\Pi_1)$, it defines a measured geodesic lamination over the hyperbolic surface, $\partial\opK_\theta^+/\theta(\Pi_1)$. We thus obtain two pairs of maps, each taking values in spaces of Teichm\"uller data of real dimension $(6\mathfrak{g}-6)$ (c.f. Table~\ref{table ads 2}).
\begin{table}[h!]
\begin{center}
\begin{tabular}{|c|c|c|}
\hline
\bf Map & \bf Description  &  \bf Codomain \\
\hline
$\opI_0^\pm$& \multicolumn{1}{|l|}{The intrinsic metric of $\partial\opK_\theta^\pm$} & $\Thyp$\\
$\opL^\pm$ & \multicolumn{1}{|l|}{The measured geodesic lamination of $\partial\opK_\theta^\pm$} & $\opML$ \\
\hline
\end{tabular}
\end{center}
\caption{Maps taking values in spaces of real dimension $(6\mathfrak{g}-6)$.}\label{table ads 2}
\end{table}\par
We now have the following scattering type result.
\begin{theorem}[Mess \cite{Mess}]\label{messAdS4}
\noindent For each $\epsilon\in\left\{+,-\right\}$, the map $(\opI_0^\epsilon,\opL^\epsilon)$ defines a bijection from $\opGHMC_{-1}$ into $\Thyp\times\opML$.
\end{theorem}
\noindent In order to visualise how the set, $\Omega_\theta$, is recovered from the metric and the lamination, we first recall a general result of convex sets. A subset, $\Sigma$, of $\opAdS^3$ is said to be a {\sl convex surface} whenever it is a relatively open subset of the boundary of some convex subset, $X$, which has non-trivial interior. We then say that $\Sigma$ is {\sl pleated} whenever it has the property that for every $p\in\Sigma$, there exists a {\sl relatively open} geodesic segment in $\opAdS^3$ which is contained in $\Sigma$ and which passes through $p$. Finally, we say that $\Sigma$ is {\sl future} or {\sl past oriented} depending on whether its outward pointing supporting normals are future or past oriented. In particular, $\partial\opK_\theta^+$ and $\partial \opK_\theta^-$ are both pleated convex surfaces, $\partial\opK_\theta^+$ is future oriented, and $\partial\opK_\theta^-$ is past oriented.
\begin{theorem}
\noindent If $\Sigma$ is a complete, future (resp. past) oriented, pleated convex surface in $\Omega_\theta$ which is invariant with respect to the action of $\theta(\Pi_1)$, then $\Sigma$ coincides with $\partial\opK_\theta^+$ (resp. $\partial\opK_\theta^-$).
\end{theorem}
\begin{proof} Indeed, $\Sigma$ is a convex surface in $\mathbb{R}\mathbb{P}^3$ whose boundary is $\Gamma_\theta$. Since $\Sigma$ is pleated, it is a boundary component of the convex hull of $\Gamma_\theta$ (c.f. Theorem~$4.18$ of \cite{smith-livre}), and the result follows.\end{proof}
The set, $\Omega_\theta$, is now recovered from the metric and the lamination by constructing a pleated convex surface. In order to see this, consider first the case of an isometric embedding, $e:\mathbb{H}^2\rightarrow\opAdS^3$, and a complete, weighted geodesic, $(\Gamma,a)$, in $\mathbb{H}^2$. The image of $e$ is a spacelike plane, $B_1^\perp$, say. We bend $B_1^\perp$ along $\Gamma$ as follows. Consider the locus of all points, $B$, in $\opAdS^3$ such that $B^\perp$ contains $\Gamma$, and observe that this is also a complete, spacelike geodesic, and it contains $B_1$. Now let $B_2$ be a point lying at a distance, $a$, along this geodesic from $B_1$, and, as before, let $X_{12}$ be the intersection of the respective pasts of $B_1^\perp$ and $B_2^\perp$. We now call $\partial X_{12}$ the {\sl bending} of $e$ along $\Gamma$ by the {\sl hyperbolic angle}, $a$. Observe, in particular, that in contrast to circular angles, which are bounded in absolute value by $\pi$, hyperbolic angles can be arbitrarily large. As with graftings, this construction extends continuously to all measured geodesic laminations, and, by identifying the resulting pleated convex surface with $\partial\opK_\theta^+$, we thereby obtain the desired inverse of $(\opI_0^+,\opL^+)$.
\par
Finally, the minimal, short action of $\theta(\Pi_1)$ on a real tree is determined by the boundary component, $\partial\Omega_\theta^+$, of $\Omega_\theta^+$. Indeed, let $d$ be the pseudo-metric over $\partial\Omega_\theta^+$ given by the infimal lengths of continuous curves in $\partial\Omega_\theta^+$ with respect to the Minkowski metric, $\langle\cdot,\cdot\rangle_{2,2}$. Identifying points separated by zero distance yields a quotient space, $\partial\Omega_\theta^+/\sim$. As in the Minkowski case, this quotient is a real tree, and the induced action of $\theta(\Pi_1)$ is minimal and short. In this manner, we obtain another pair of maps taking values in a space of Teichm\"uller data of real dimension $(6\mathfrak{g}-6)$ (c.f. Table~\ref{table ads 9}),
\begin{table}[h!]
\begin{center}
\begin{tabular}{|c|c|c|}
\hline
\bf Map & \bf Description  &  \bf Codomain \\
\hline
$\opT^\pm$ & \multicolumn{1}{|l|}{The minimal, short action of $\theta$ on the real tree $\partial\Omega^+_\theta/\sim$} & $\opRT$\\
\hline
\end{tabular}
\end{center}
\caption{Maps taking values in spaces of real dimension $(6\mathfrak{g}-6)$.}\label{table ads 9}
\end{table}
\noindent although, since the tree, $\partial\Omega_\theta^\pm/\sim$, is dual to the measured geodesic lamination, $L^\pm$, these maps do not actually yield any new information.
\subsection{Earthquakes}\label{Earthquakes}
Theorems \ref{messAdS3} and \ref{messAdS4} have a nice interpretation in terms of Teichm\"uller theory. To see this, we first introduce the left and right {\sl generalised Gauss maps}, $N_l$ and $N_r$, as follows. Given a convex set, $X$, and a boundary point, $A$, of this set, define
\begin{eqnarray*}
\ N_l(A) &:=& \left\{ A^{-1}B\ |\ B\ \text{a supporting normal to}\ \partial\opK_\theta^+\ \text{at}\ A\right\},\ \text{and}\\
\ N_r(A) &:=& \left\{ BA^{-1}\ |\ B\ \text{a supporting normal to}\ \partial\opK_\theta^+\ \text{at}\ A\right\}.
\end{eqnarray*}
In order to understand the geometry of these maps, consider first two spacelike planes, $B_1^\perp$ and $B_2^\perp$, in $\opAdS^3$ which intersect along a shared spacelike geodesic, $\Gamma$. As before, let $X_{12}$ be the intersection of their respective pasts, and consider the actions of $N_l$ and $N_r$ over the boundary, $\partial X_{12}$, of this set. Since $\partial X_{12}$ has a unique supporting normal at every point, $A$, not lying on $\Gamma$, both $N_l$ and $N_r$ define local isometries of the complement of $\Gamma$ in $\partial X_{12}$ into the future component, $\mathbb{H}^2$, of the unit pseudosphere in $\opsl(2,\R)$.
However, at every point, $A$, of $\Gamma$, $\partial X_{12}$ has an entire continuum of supporting normals which are all contained in the intersection of $\Bbb{H}^2$ with the plane in $\Bbb{R}^{2,2}$ spanned by $B_1$ and $B_2$. Furthermore, since $A\in\opAdS^3$ acts isometrically on $\R^{2,2}$,
\begin{equation*}
d_{\mathbb{H}^2}(A^{-1}B_1,A^{-1}B_2) = d_{\mathbb{H}^2}(B_1A^{-1},B_2A^{-1}) = d_{\opAdS^3}(B_1,B_2)~.
\end{equation*}
Thus, upon checking the orientations, we see that $N_l$ and $N_r$ define respectively left and right earthquakes of strength $d_{\opAdS^3}(B_1,B_2)$ along $\Gamma$ from $\partial X_{12}$ into $\mathbb{H}^2$.
\par
Consider now the left and right Gauss maps of $\partial\opK_\theta^+$. Since $N_l$ and $N_r$ are equivariant with respect to $\theta_r$ and $\theta_l$, respectively, it follows that the right earthquake along the measured geodesic lamination, $L$, sends the point of Teichm\"uller space determined by $X/\theta(\Pi_1)$ to the point determined by $\mathbb{H}^2/\theta_l(\Pi_1)$, whilst the left earthquake along this measured geodesic lamination sends this point to the point determined by $\mathbb{H}^2/\theta_r(\Pi_1)$. We thereby recover the well known earthquake theorem.
\earthquakebis*
\begin{remark}In fact, any two of Theorems \ref{messAdS3}, \ref{messAdS4} and \ref{EarthquakeTheorem2} imply the third.\end{remark}
\noindent This construction is illustrated by Figure~\ref{diag2ads}.
\begin{figure}[h!]
\begin{center}
\includegraphics[scale=0.5]{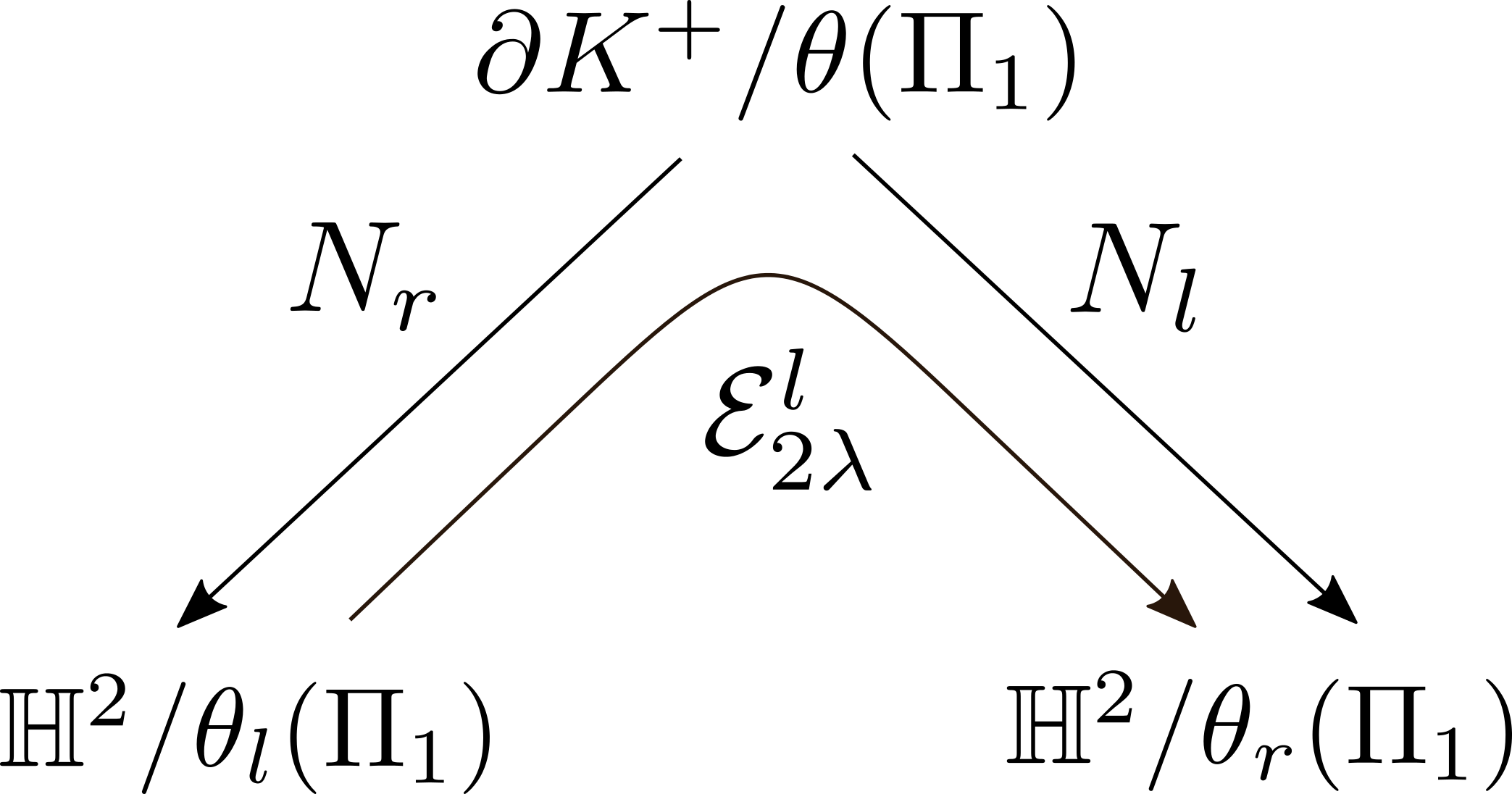}\caption{\textbf{The earthquake theorem I.} The upper vertex is here a future oriented, pleated convex surface, and the curved arrow joining the lower vertices is an earthquake given by twice the measured geodesic lamination of this pleated surface.}\label{diag2ads}
\end{center}
\end{figure}
\par
Theorem~\ref{EarthquakeTheorem2} is not symmetric in $h_1$ and $h_2$. Indeed, when the order of these two points is reversed, we obtain another point, $h'\in\Thyp$, and another measured geodesic lamination, $\lambda'$. However, by reversing the time orientation of $\opAdS^3$, it follows by uniqueness that the marked hyperbolic metric, $h'$, and the measured geodesic lamination, $\lambda'$, are precisely those determined by the past boundary component, $\partial\opK_\theta^-$, of $\opK_\theta$. This yields the situation illustrated in Figure~\ref{diag3ads}.
\begin{figure}[h!]
\begin{center}
\includegraphics[scale=0.5]{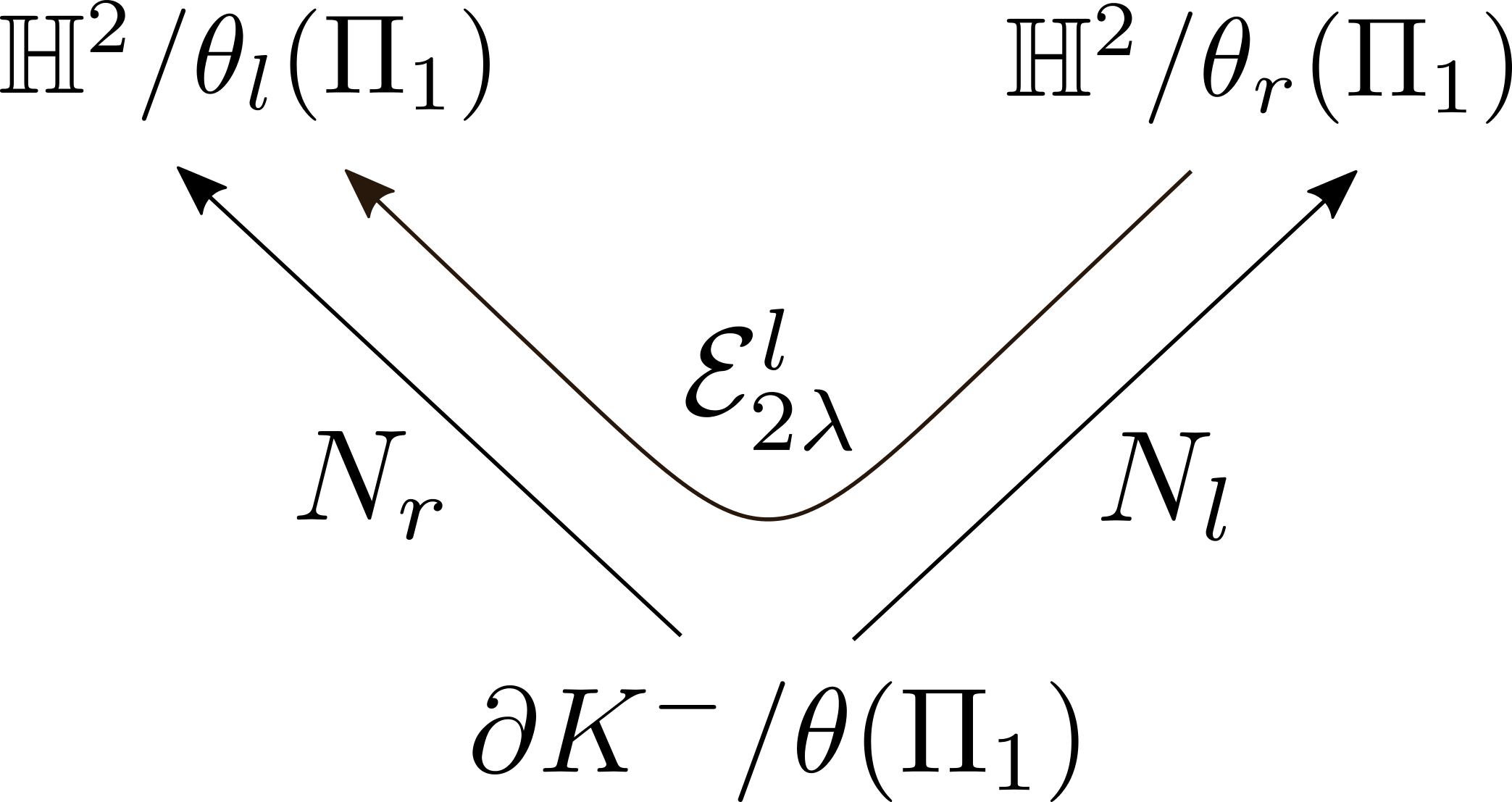}\caption{\textbf{The earthquake theorem II.} The lower vertex is here a past oriented, pleated convex surface, and the curved arrow joining the upper vertices is an earthquake given by twice the measured geodesic lamination of this pleated surface.}\label{diag3ads}
\end{center}
\end{figure}
\subsection{Fixed point theorems and other results}\label{FixedPointTheoremsAndOtherResults}
Consider the pair $(\opL^+,\opL^-)$. As in the Minkowski case, this map is not surjective, taking values instead in $\opML\times_{\operatorname{fill}}\opML$, which we recall is the subset of $\opML\times\opML$ consisting of those pairs that fill $S$ (c.f. Section~\ref{LaminationsAndTrees}). Denoting by $\opFuc_{-1}$ the space of Fuchsian AdS spacetimes, that is, those spacetimes in $\opGHMC_{-1}$ whose holonomies have equal left and right components, we now have
\begin{theorem}[Bonsante--Schlenker \cite{BS12}]\label{BSads}
\noindent The map $(\opL^+,\opL^-)$ defines a surjection from $\opGHMC_{-1}\setminus\opFuc_{-1}$ onto $\opML\times_{\operatorname{fill}}\opML$.
%
\end{theorem}
\begin{remark} In \cite{Mess}, Mess asks whether this map is a bijection.
\end{remark}
\noindent In particular, given a pair $(\lambda^+,\lambda^-)$ of measured geodesic laminations which fills $S$, there exists a homomorphism, $\theta:=(\theta_l,\theta_r)$, such that the diagram in Figure~\ref{diag ads 4} commutes. In other words, Theorem~\ref{BSads} is equivalent to the following fixed point theorem of Teichm\"uller theory.
\begin{figure}[h!]
\begin{center}
\includegraphics[scale=0.5]{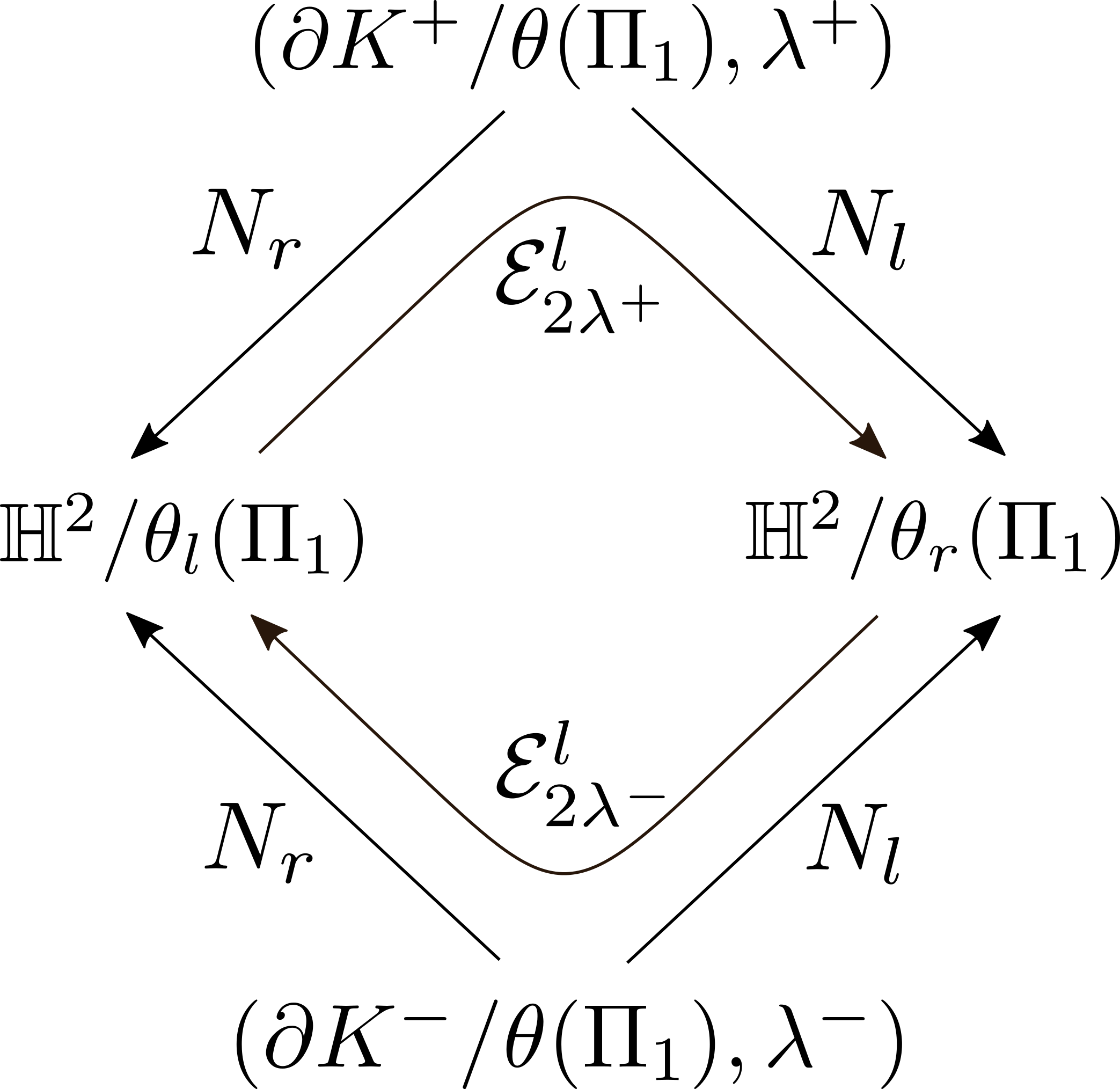}\caption{\textbf{The double earthquake theorem.} Here the upper vertex is a future oriented, pleated convex surface with measured geodesic lamination, $\lambda^+$, and the lower vertex is a past oriented, pleated convex surface with measured geodesic lamination, $\lambda^-$. The curved arrows joining the middle vertices are earthquakes given by $2\lambda^+$ and $2\lambda^-$ respectively.}\label{diag ads 4}
\end{center}
\end{figure}
\begin{theorem}[Bonsante--Schlenker \cite{BS12}]
\noindent Given any pair, $(\lambda,\mu)$ of measured geodesic laminations which fills $S$, the compositions $\mathcal{E}^l_\lambda\circ\mathcal{E}^l_\mu$ and $\mathcal{E}^r_\lambda\circ\mathcal{E}^r_\mu$ both have fixed points in $\Thyp$.
\end{theorem}
In a similar vein, we have,
\begin{theorem}[Diallo \cite{Dia13,Dia14}]\label{diallo}
\noindent The map $(\opI_0^+,\opI_0^-)$ defines a surjection from $\opGHMC_{-1}$ onto $\Thyp\times\Thyp$.
\end{theorem}
\begin{remark} It is not known whether this map is a bijection.
\end{remark}
\noindent In particular, given a pair $(g^+,g^-)$ of hyperbolic metrics, there exists a homomorphism, $\theta:=(\theta_l,\theta_r)$, such that the diagram in Figure~\ref{diag ads 5} commutes.
\begin{figure}[h!]
\begin{center}
\includegraphics[scale=0.5]{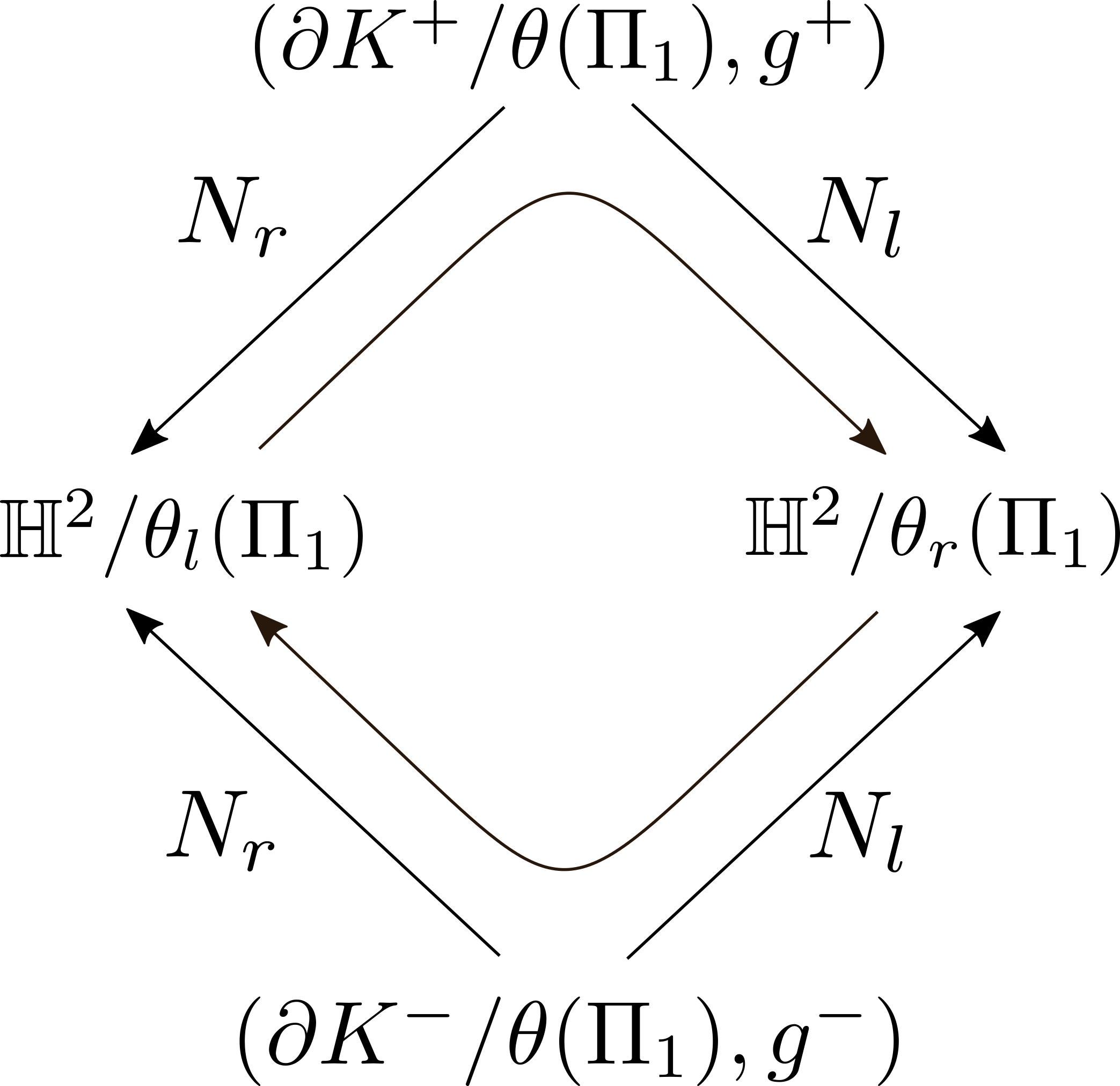}\caption{\textbf{The earthquake symmetry theorem.} Here the upper vertex is a future oriented, pleated convex surface with intrinsic metric, $g^+$, and the lower vertex is a past oriented, pleated convex surface with intrinsic metric, $g^-$. The curved arrows joining the middle vertices are earthquakes.}\label{diag ads 5}
\end{center}
\end{figure}
\par
Earthquake symmetries (c.f. \cite{BMS2}) present a nice interpretation of Theorem~\ref{diallo} within the framework of classical Teichm\"uller theory, and are defined as follows. First, by Theorem~\ref{EarthquakeTheorem2}, for all $g\in\Thyp$, the earthquake maps, $\mathcal{E}_g^l:\lambda\mapsto\mathcal{E}^l(g,\lambda)$ and $\mathcal{E}_g^r:\lambda\mapsto\mathcal{E}^r(g,\lambda)$, define bijections from $\opML$ into $\Thyp$. The {\sl left earthquake symmetry} of $\Thyp$ about the metric, $g$, is then defined by $\mathcal{ES}^l_g:=\mathcal{E}^l_g\circ(\mathcal{E}^r_g)^{-1}$. In other words, $h'=\mathcal{ES}^{l}_g(h)$ whenever there exists a measured geodesic lamination, $\lambda$, such that $\mathcal{E}^r_\lambda(g)=h$ and $\mathcal{E}^l_\lambda(g)=h'$. The {\sl right earthquake symmetry} is defined in a similar manner by $\mathcal{ES}^r_g:=\mathcal{E}^r_g\circ(\mathcal{E}^l_g)^{-1}$. As their names suggest, the two maps, $\mathcal{ES}^l_g$ and $\mathcal{ES}^r_g$, are both involutions of $\Thyp$ with $g$ as their unique fixed points. Theorem~\ref{diallo} is now equivalent to the following fixed point theorem.
\begin{theorem}[Diallo \cite{Dia13,Dia14}]
\noindent For any pair, $(g,h)$, of marked, hyperbolic metrics, the compositions, $\mathcal{ES}_g^l\circ\mathcal{ES}_h^l$ and $\mathcal{ES}_g^r\circ\mathcal{ES}_h^r$, both have fixed points in $\Thyp$.
\end{theorem}
Other scattering type theorems arise naturally from existing results. Indeed,
\begin{theorem}[\cite{mes+}]
\noindent For $\alpha\in\left\{l,r\right\}$ and $\epsilon\in\left\{+,-\right\}$, the map $(\Theta_\alpha,\opI_0^\epsilon)$ defines a bijection from $\opGHMC_{-1}$ into $\Trep\times\Thyp$.
\end{theorem}
\begin{proof} It suffices to consider the case $(\Theta_l,\opI_0^+)$, as the remaining cases are proven in a similar manner. Let $H:\Trep\rightarrow\Thyp$ be the canonical identification, and define $\Psi:\Trep\times\Thyp\rightarrow\Trep\times\Trep$ by
$$
\Psi([\theta],g) := ([\theta],(H^{-1}\mathcal{E}\mathcal{S}_g H)([\theta])).
$$
It follows by the earthquake theorem (Theorem \ref{EarthquakeTheorem}) that $\Psi$ is a bijection. However,
$$
(\Theta_l,\opI_0^+) = \Psi^{-1}\circ(\Theta_l,\Theta_r),
$$
and the result now follows by Theorem \ref{messAdS3}.\end{proof}
\noindent A similar argument yields,
\begin{theorem}[\cite{mes+}]
\noindent For $\alpha\in\left\{l,r\right\}$ and $\epsilon\in\left\{+,-\right\}$, the map $(\Theta_\alpha,L^\epsilon)$ defines a bijection from $\opGHMC_{-1}$ into $\Trep\times\opML$.
\end{theorem}
\subsection{Smooth parametrisations}\label{SmoothParametrisationsAdSCase}
Our starting point for constructing smooth parametrisations of $\opGHMC_{-1}$ is the following result.
\begin{theorem}[Barbot--B\'eguin--Zeghib \cite{BBZ}]
\label{bbz const ads}
\noindent Let $\theta_l,\theta_r:\pi_1(S)\rightarrow\opPSL(2,\R)$ be homomorphisms which are injective and which act properly discontinuously on $\mathbb{H}^2$. For all $\kappa>0$, there exists a unique, smooth, spacelike, LSC surface $\Sigma_\kappa^\pm$, which is embedded in $\Omega_\theta^\pm$, is invariant under the action of $\theta:=(\theta_l,\theta_r)$, and has constant extrinsic curvature equal to $\kappa$. Furthermore, the family of all such surfaces foliates $\Omega_\theta^\pm$ as $\kappa$ varies over the interval $]0,\infty[$.
\end{theorem}
Maps taking values in spaces of Teichm\"uller data are constructed using Theorem \ref{bbz const ads} as follows. For $\kappa>0$, consider the spacelike, LSC, embedded surface, $\Sigma_\kappa^\pm$, in $\Omega_\theta^\pm$ which is invariant with respect to $\theta$ and which has constant extrinsic curvature equal to $\kappa$. Let $I_\kappa^\pm$, $\II_\kappa^\pm$ and $\III_\kappa^\pm$ be its first, second and third fundamental forms respectively. The form, $(1+\kappa)I^\pm_\kappa$, defines a hyperbolic metric over $\Sigma_\kappa^\pm$. Furthermore, if we denote by $N$ the future oriented, unit, normal vector field over $\Sigma_\kappa^\pm$, then we can show that $N$ itself defines an equivariant immersion in $\Omega_\theta^\mp$ which, in fact, parametrises $\Sigma_{1/\kappa}^\mp$, so that $\kappa^{-1}(1+\kappa)\III^\pm_\kappa$ also defines a hyperbolic metric over $\Sigma_\kappa^+$. Next, by convexity, $\II_\kappa^\pm$ is also positive definite, and therefore also defines a metric over $\Sigma_\kappa$, but since it has no clear curvature properties, we consider it rather as defining a point in $\Thol$. In summary, we have three pairs of maps, each taking values in spaces of Teichm\"uller data of real dimension $(6\mathfrak{g}-6)$ (c.f. Table~\ref{table ds 1}).
\begin{table}[h!]
\begin{center}
\begin{tabular}{|c|c|c|}
\hline
\bf Map & \bf Description  &  \bf Codomain \\
\hline
$\I_\kappa^\pm$ & \multicolumn{1}{|l|}{\pbox{20cm}{The first fundamental form of $\Sigma_\kappa^\pm$}} & $\Thyp$\\
$\operatorname{II}_\kappa^\pm$& \multicolumn{1}{|l|}{\pbox{20cm}{The second fundamental form of $\Sigma_\kappa^\pm$}}&$\Thol$ \\
$\operatorname{III}_\kappa^\pm$& \multicolumn{1}{|l|}{\pbox{20cm}{The third fundamental form of $\Sigma_\kappa^\pm$ \\ (The first fundamental form of $\Sigma_{1/\kappa}^\mp$)}}&$\Thyp$\\
\hline
\end{tabular}
\end{center}
\caption{Maps taking values in spaces of real dimension $(6\mathfrak{g}-6)$.}\label{table ds 1}
\end{table}
\par
These maps are complemented to maps taking values in spaces of Teichm\"uller data of real dimension $(12\mathfrak{g}-12)$ as follows. First, the shape operator, $A_\kappa^\pm$, of $\Sigma_\kappa^\pm$ defines, up to a constant factor, a Labourie field of $I_\kappa^\pm$, whilst its inverse, $(A_\kappa^\pm)^{-1}$, defines a Labourie field of $\III_\kappa^\pm$, so that the pairs $(I_\kappa^\pm,A_\kappa^\pm)$ and $(\III_\kappa^\pm,(A_\kappa^\pm)^{-1})$ define points of $\opLab\Thyp$. Likewise, the Hopf differential, $\phi_\kappa^\pm$, of $I_\kappa^\pm$ with respect to the conformal structure of $\II_\kappa^\pm$ defines a holomorphic quadratic differential, so that the pair $(\II_\kappa^\pm,\phi_\kappa^\pm)$ yields a point of $\opT^*\Thol$. We thus have three pairs of maps taking values in spaces of Teichm\"uller data of real dimension $(12\mathfrak{g}-12)$ (c.f. Table~\ref{table ds2}).
\begin{table}[h!]
\begin{center}
\begin{tabular}{|c|c|c|}
\hline
\bf Map & \bf Description  &  \bf Codomain \\
\hline
$\opA_{\opI,\kappa}^\pm$ & \multicolumn{1}{|l|}{\pbox{20cm}{The first fundamental form of $\Sigma_\kappa^\pm$ \\ together with the Labourie field, $A_\kappa^\pm$}} & $\opLab\Thyp$\\
$\Phi_\kappa^\pm$& \multicolumn{1}{|l|}{\pbox{20cm}{The second fundamental form of $\Sigma_\kappa^\pm$ \\ together with the Hopf differential $\phi_\kappa^\pm$}} & $\opT^*\Thol$ \\
$\opA_{\opIII,\kappa}^\pm$&\multicolumn{1}{|l|}{\pbox{20cm}{The third fundamental form of $\Sigma_\kappa^\pm$ \\ together with the inverse of the Labourie field, $A_\kappa^\pm$}}& $\opLab\Thyp$  \\
\hline
\end{tabular}
\end{center}
\caption{Maps taking values in spaces of real dimension $(12\mathfrak{g}-12)$.}\label{table ds2}
\end{table}
\par
We readily show that these maps parametrise $\opGHMC_{-1}$. Indeed,
\begin{theorem}\label{FundamentalTheoremOfSurfacesADS}
\noindent The maps $\opA_{\opI,\kappa}^\pm$ and $\opA_{\opIII,\kappa}^\pm$ define real analytic diffeomorphisms from $\opGHMC_{-1}$ into $\opLab\Thyp$.
\end{theorem}
\begin{proof}[Sketch of proof] Consider a hyperbolic metric, $g$, and a Labourie field, $A$. By the fundamental theorem of surface theory (Theorem~\ref{FTS}), there exists an LSC equivariant immersion, $e:(\tilde{S},(1+\kappa)^{-1}g)\rightarrow\opAdS^{3}$, with shape operator equal to $\sqrt{\kappa}A$, which is unique up to isometries of $\opAdS^{3}$. This yields a real analytic inverse of $\opA_{\opI,\kappa}^\pm$, and since $\opA_{\opIII,\kappa}^\pm$ is equivalent to $\opA_{\opI,1/\kappa}^\mp$, it also yields a real analytic inverse of $\opA_{\opIII,\kappa}^\pm$, and the result follows.\end{proof}
\begin{theorem}
\noindent The map $\Phi_\kappa^\pm$ defines a real analytic diffeomorphism from $\opGHMC_{-1}$ into $\opT^*\Thol$.
\end{theorem}
\begin{proof} Indeed, $\opI_\kappa^\pm\circ(\opA_\kappa^\pm)^{-1}$ coincides with $\Phi\circ\mathcal{A}^{-1}$, where $\Phi$ and $\mathcal{A}$ are defined as in Sections \ref{HopfDifferentials} and \ref{LabourieFields} respectively. The result now follows by Theorems \ref{wolf}, \ref{AIsARealAnalDiff} and \ref{FundamentalTheoremOfSurfacesADS}.\end{proof}
\subsection{Landslides, rotations and more symmetries}
In order to interpret Theorem~\ref{bbz const ads} in terms of classical Teichm\"uller theory, consider again the spacelike, LSC, embedded surface, $\Sigma_\kappa^+$. Let $A$ be its shape operator, and let $g:=(1+\kappa)I_\kappa^+$ be its hyperbolic metric, where $I_\kappa^+$ is its first fundamental form. If $N_l$ and $N_r$ are its left and right Gauss maps respectively, and if $h$ denotes the hyperbolic metric of $\mathbb{H}^2$, then a straightforward calculation yields (c.f. \cite{KS07,bar}),
\begin{eqnarray}\label{InFactTheyAreLandslides}
\ && N_l^*h := g\left(\opCos(t)J_0 + \opSin(t)\frac{1}{\sqrt{\kappa}}A,\opCos(t)J_0 + \opSin(t)\frac{1}{\sqrt{\kappa}}A\right)~,\ \text{and}\\
\  && N_r^*h := g\left(\opCos(t)J_0 - \opSin(t)\frac{1}{\sqrt{\kappa}}A,\opCos(t)J_0 - \opSin(t)\frac{1}{\sqrt{\kappa}}A\right)~,
\end{eqnarray}
where $J_0$ is the complex structure of $g$ compatible with the orientation, and $t$ satisfies,
$$
\opTan(t) = \sqrt{\kappa}~.
$$
It follows that $N_r$ and $N_l$ define landslides from $\Sigma_\kappa^+$ into $\mathbb{H}^2/\theta_l(\Pi_1)$ and $\mathbb{H}^2/\theta_r(\Pi_1)$ respectively, so that Theorem~\ref{bbz const ads} therefore yields the following ``landslide theorem''.
\landslide*
\begin{proof} Let $\kappa:=\opTan(t/2)^2$. Consider two marked hyperbolic metrics, $g_1,g_2\in\Thyp$. Denote $\theta_l:=H^{-1}(g_1)$ and $\theta_r:=H^{-1}(g_2)$, where $H:\Trep\rightarrow\Thyp$ is the canonical identification. Define
$$
(h^+,h^-) = (\opI_\kappa^+,\opIII_\kappa^+)\circ(\Theta_l,\Theta_r)^{-1}(\theta_l,\theta_r).
$$
It follows from \ref{InFactTheyAreLandslides} that $h^+$ and $h^-$ yield the desired landslide, thus proving existence. To prove uniqueness, let $A$ be the Labourie field of $h^+$ with respect to $h^-$. By the fundamental theorem of surface theory (Theorem~\ref{FTS}), there exists an equivariant immersion of $\tilde{S}$ into $\opAdS^3$ whose metric is $(1+\kappa)^{-1}h^+$ and whose shape operator is $\sqrt{\kappa}A$. Since this immersion is invariant with respect to $(\theta_l,\theta_r)$, the result now follows by uniqueness in Theorem~\ref{bbz const ads}.\end{proof}
\noindent Since landslides and rotations are equivalent, Theorem~\ref{LandslideTheorem} can also stated as follows.
\rotation*
\noindent This result is illustrated in Figure~\ref{fig:ads lisse1}.
\begin{figure}
\begin{center}
\includegraphics[scale=0.5]{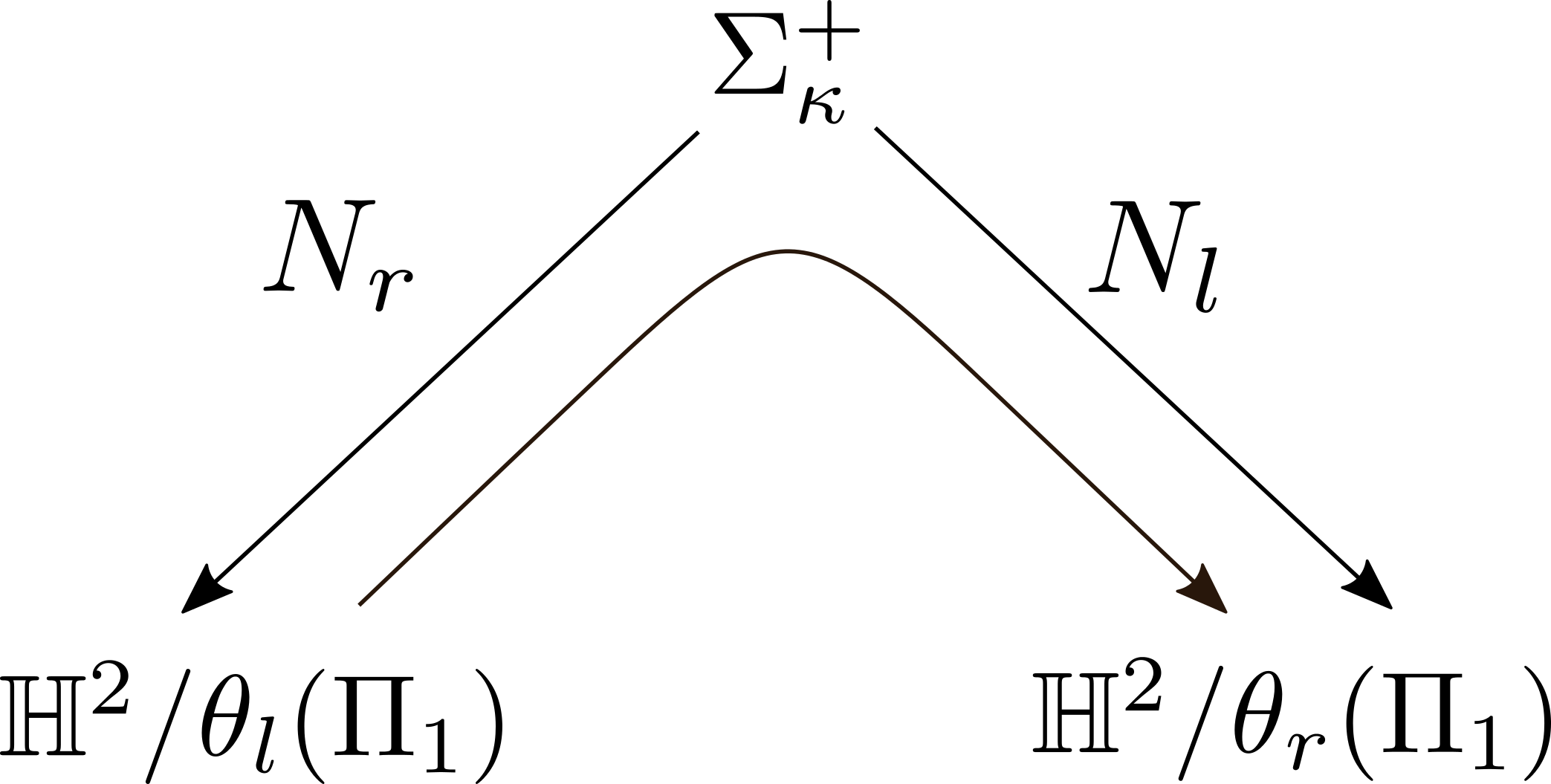}\caption{\textbf{The rotation theorem.} Here, the upper vertex is an equivariant immersion of constant extrinsic curvature equal to $\kappa:=\opTan(t/2)^2$ whose second fundamental form has the conformal class of $\mathcal{H}$, and the curved arrow joining the lower vertices is a rotation of angle $2t$ about $\mathcal{H}$.}\label{fig:ads lisse1}
\end{center}
\end{figure}
\par
Other results of scattering type yield new fixed point theorems in Teichm\"uller theory. Indeed, consider the pair $(\opI_\kappa^+,\opI_{\kappa'}^-)$. We have,
\begin{theorem}[Bonsante--Mondello--Schlenker \cite{BMS2}]\label{weyl ads}
\noindent For all $\kappa,\kappa'>0$, the map $(\opI_\kappa^+,\opI_{\kappa'}^-)$ defines a surjection from $\opGHMC_{-1}$ onto $\Thyp\times\Thyp$. Furthermore, when $\kappa\kappa'=1$, this map is bijection.
\end{theorem}
\begin{remark}In the general case, it is not known whether this map is a bijection.
\end{remark}
\begin{figure}
\begin{center}
\includegraphics[scale=0.5]{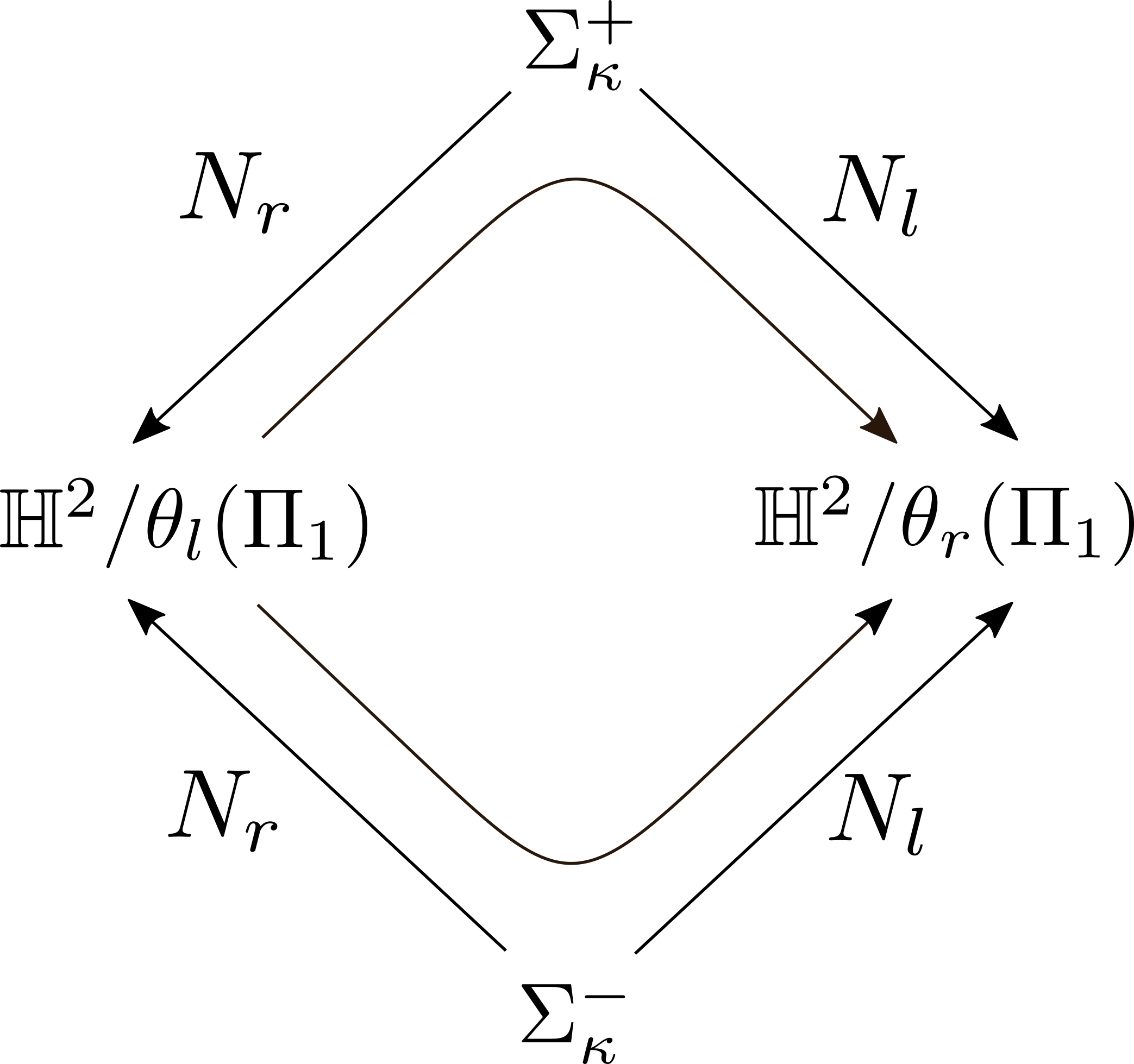}\caption{\textbf{The rotation symmetry theorem.} Here the upper vertex is a future oriented equivariant immersion of constant extrinsic curvature equal to $\kappa$ whose intrinsic metric is (up to a factor) equal to $g^+$, and the lower vertex is a past oriented equivariant immersion of constant extrinsic curvature equal to $\kappa'$ whose intrinsic metric is (also up to a factor) equal to $g^-$.}\label{fig:ads lisse2}
\end{center}
\end{figure}
\noindent Thus, given a pair $(g^+,g^-)$ of marked hyperbolic metrics, there exists a homomorphism, $\theta:=(\theta_l,\theta_r)$ such that the diagram in Figure~\ref{fig:ads lisse2} commutes.
\par
In analogy to the case of Theorem~\ref{BSads} and \ref{diallo}, rotation symmetries (c.f. \cite{BMS2}) present a nice interpretation of Theorem~\ref{weyl ads} within the framework of Teichm\"uller theory. They are defined as follows. First, by Theorem~\ref{rotations}, given $t\in]-\pi,\pi[$, and a marked hyperbolic metric, $g\in\Thyp$, the maps, $\mathcal{R}_{g,t}:\mathcal{H}\mapsto\mathcal{R}(t,g,\mathcal{H})$ and $\mathcal{R}_{g,-t}:\mathcal{H}\mapsto\mathcal{R}(-t,g,\mathcal{H})$, define bijections from $\Thol$ into $\Thyp$. The {\sl positive rotation symmetry} of magnitude $2t$ of $\Thyp$ about the marked hyperbolic metric, $g$, is then defined by $\mathcal{RS}^+_{g,t}:=\mathcal{R}_{g,t}\circ(\mathcal{R}_{g,-t})^{-1}$.
In other words, $h'=\mathcal{RS}^+_{g,t}(h)$ whenever there exists a marked holomorphic structure, $\mathcal{H}$, such that $\mathcal{R}(-t,g,\mathcal{H})=h$ and $\mathcal{R}(t,g,\mathcal{H})=h'$. The {\sl negative rotation symmetry} of magnitude $2t$ is defined in a similar manner by $\mathcal{RS}^-_{g,t}:=\mathcal{R}_{g,-t}\circ(\mathcal{R}_{g,t})^{-1}$. As their names suggest, the two maps, $\mathcal{RS}^+_{g,t}$ and $\mathcal{RS}^-_{g,t}$ are both involutions of $\Thyp$ with $g$ as their unique fixed point. We now see that Theorem~\ref{weyl ads} is equivalent to the following fixed point theorem of classical Teichm\"uller theory.
\begin{theorem}[Bonsante--Mondello--Schlenker \cite{BMS2}]
\noindent For any $t,t'\in]0,\pi[$, and for any pair $(g,g')$ of marked hyperbolic metrics, the compositions, $\mathcal{RS}_{g,t}^+\circ\mathcal{RS}_{g',-t'}^+$ and $\mathcal{RS}_{g,t}^-\circ\mathcal{RS}_{g',-t'}^-$, both have fixed points in $\Thyp$. Furthermore, when $t+t'=\pi$, these fixed points are unique.
\end{theorem}
Other scattering type results readily follow from existing results.
\begin{theorem}
\noindent For all $\alpha\in\left\{l,r\right\}$, for all $\epsilon\in\left\{+,-\right\}$, and for all $\kappa<-1$, the map $(\Theta_\alpha,\opI_{\kappa}^\epsilon)$ defines a bijection from $\opGHMC_{-1}$ into $\Trep\times\Thyp$.
\end{theorem}
\begin{proof} It suffices to consider the case $(\Theta_l,\opI_{\kappa}^+)$, as the remaining three cases are proven in a similar manner. Let $H:\Trep\rightarrow\Thyp$ be the canonical identification, and define $\Psi:\Trep\times\Thyp\rightarrow\Trep\times\Trep$ by
$$
\Psi([\theta],g) := ([\theta],(H^{-1}\mathcal{R}\mathcal{S}_{g,t}H)([\theta])),
$$
where $t:=2\opArcTan(\sqrt{\kappa})$. By the rotation theorem (Theorem \ref{rotations}), $\Psi$ is a bijection. However,
$$
(\Theta_l,\opI_\kappa^+) = \Psi^{-1}\circ(\Theta_l,\Theta_r),
$$
and the result now follows by Theorem \ref{messAdS3}.\end{proof}
\noindent In a similar manner, we obtain.
\begin{theorem}
\noindent For all $\alpha\in\left\{l,r\right\}$, for all $\epsilon\in\left\{+,-\right\}$, and for all $\kappa<-1$, the map $(\Theta_\alpha,\opII_{\kappa}^\epsilon)$ defines a bijection from $\opGHMC_{-1}$ into $\Trep\times\Thyp$.
\end{theorem}
Finally, we observe that invariant surfaces may also be constructed subject to other curvature conditions. For example, in \cite{BBZ2}, foliations by surfaces of constant mean curvature are constructed. Of particular interest is
\begin{theorem}[Barbot--B\'eguin--Zeghib \cite{BBZ2}]\label{bbz max}
\noindent There exists a unique space-like maximal surface $\Sigma_{\mathrm{max}}(\theta)$ in $\Omega$ which is invariant under the action of $\theta(\pi_1(S))$.
\end{theorem}
\begin{figure}
\begin{center}
\includegraphics[scale=0.5]{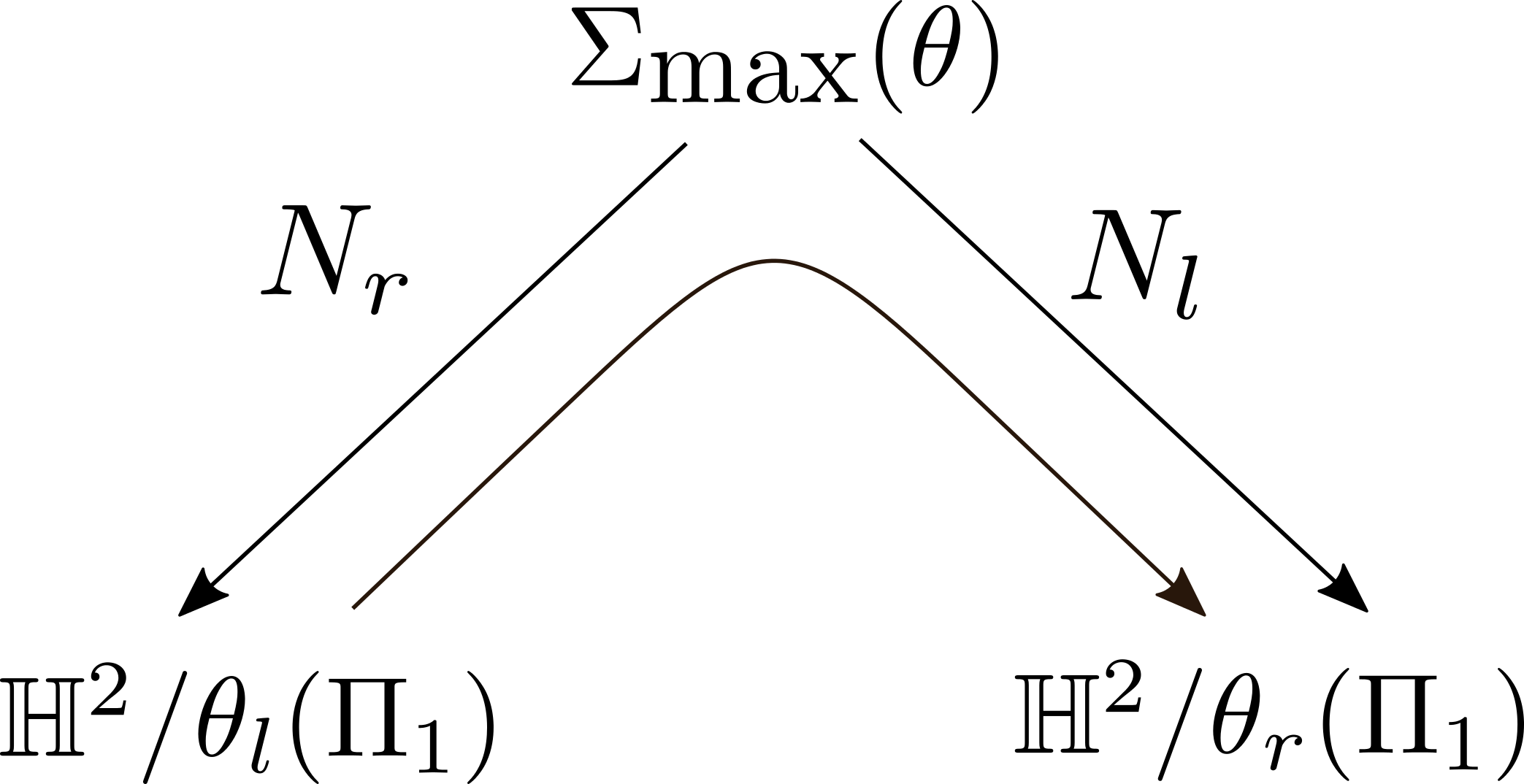}
\caption{\textbf{Minimal lagrangian diffeomorphisms.} The upper vertex is a maximal surface, and  the curved arrow joining the lower two vertices is a minimal lagrangian diffeomorphism}\label{fig:ads max}
\end{center}
\end{figure}
\noindent A straightforward calculation reveals that in this case, the composition $N_l \circ N_r^{-1}$ is a minimal Lagrangian diffeomorphism. This is illustrated in Figure~\ref{fig:ads max}. In particular (c.f. \cite{KS07}), we recover
\MLD*
\section{De Sitter space}\label{DeSitterSpaceRiemannSphereHyperbolicSpace}
\subsection{De Sitter space}\label{DeSitterSpace}
{\sl De Sitter space}, which will be denoted by $\opdS^3$, is defined to be the projective quotient of the one sheeted hyperboloid in $\R^{3,1}$, that is,
$$
\opdS^3 := \left\{ x\ |\ x_1^2 - x_2^2 - x_3^2 - x_4^2 = -1\right\}/\left\{\pm\opId\right\}~.
$$
Although de Sitter space is orientable, it is not time orientable. Its isometry group identifies with $\opPO(3,1)$, that is, the projective quotient of $\opO(3,1)$. This group has $2$ connected components, determined by whether they preserve or reverse the orientation, and its identity component identifies with $\opPSO(3,1)$.
\par
The study of equivariant immersions in de Sitter space is less straightforward than in the two preceding cases, mainly because, here, equivariant immersions are not always embedded. For this reason, we will introduce below the class of {\sl quasi-Fuchsian} immersions. Since all equivariant immersions in this class are embedded, the corresponding moduli spaces present structures similar to those studied in Sections \ref{MinkowskiSpace} and \ref{sec:ads}. For the didactic purposes of this paper, we will not consider equivariant immersions outside this class, but the reader should understand that a far richer theory nontheless exists for non-quasi-Fuchsian equivariant immersions (c.f. \cite{sca99,Mess,mes+}).
\par
In order to define quasi-Fuchsian immersions, we first study how $\opPSO(3,1)$ acts on the Riemann sphere. Recall the projective model introduced in Section~\ref{ads 1}. It follows from the definition that $\opdS^3$ projects diffeomorphically onto an open subset of $\opRP^3$. Consider now the light cone,
$$
\hat{\opC} := \left\{x\ |\ x_1^2 = x_2^2 + x_3^2 + x_4^2\right\}~,
$$
and let $\opC$ be its projective quotient. This subset is a smoothly embedded sphere whose complement consists of two connected components. One of these components identifies with $\opdS^3$, as we have already seen. The other identifies with hyperbolic space, $\mathbb{H}^3$. Indeed, recall that hyperbolic space can be defined as the projective quotient of the unit pseudosphere in $\R^{3,1}$, that is,
$$
\mathbb{H}^3 := \left\{x\ |\ x_1^2 - x_2^2 - x_3^2 - x_4^2 = 1\right\}/\left\{\pm\opId\right\}~,
$$
which naturally projects diffeomorphically onto an open subset of $\opRP^3$. In this manner, we see that $\opC$ naturally identifies as the boundary both of $\opdS^3$ and of $\mathbb{H}^3$.
\par
The Minkowski metric over $\R^{3,1}$ restricts to a degenerate metric over $\hat{\opC}$ having one null direction, namely the radial direction, and two spacelike directions. Since the radial direction is collapsed by projection onto $\opRP^3$, this degenerate metric projects to a conformal class of non-degenerate metrics over $\opC$, giving this submanifold the holomorphic structure of the Riemann sphere. Furthermore, the group, $\opPSO(3,1)$, acts faithfully on $\opC$ by conformal maps. It therefore embeds naturally as a subgroup of the M\"obius group, $\opPSL(2,\mathbb{C})$, and since these Lie groups are both connected, have the same dimension, and have the same fundamental groups, this embedding is, in fact, an isomorphism.
\par
Consider now a homomorphism, $\theta:\Pi_1\rightarrow\opPSO(3,1)$. This homomorphism is said to be {\sl quasi-Fuchsian} whenever it preserves each of the two connected components of the complement of some Jordan curve, $\Gamma_\theta$, in $\opC$, and acts properly discontinuously on each of these components. Suppose that $\Gamma_\theta$ is oriented, and let $\opC_\theta^l$ and $\opC_\theta^r$ be the components of its complement in $\opC$ which lie respectively to its left and to its right. The two quotients $\opC_\theta^l/\theta(\Pi_1)$ and $\opC_\theta^r/\theta(\Pi_1)$ then each define marked compact Riemann surfaces of genus $\mathfrak{g}$, that is, points of $\Thol$. Furthermore, if $\theta':\Pi_1\rightarrow\opPSO(3,1)$ is another homomorphism such that $\theta'=\alpha\theta\alpha^{-1}$, for some $\alpha\in\opPSO(3,1)$, then $\theta'$ preserves each of $\alpha(\opC_\theta^l)$ and $\alpha(\opC_\theta^r)$, and acts properly discontinuously on each of these sets, so that it is also a quasi-Fuchsian homomorphism. Since the quotient surfaces, $\alpha(\opC_\theta^l)/\theta'(\Pi_1)$ and $\alpha(\opC_\theta^r)/\theta'(\Pi_1)$, are conformally equivalent to $\opC_\theta^l/\theta(\Pi_1)$ and $\opC_\theta^r/\theta(\Pi_1)$ respectively, we see that the point of $\Thol\times\Thol$ defined by $\theta$ only depends on its conjugacy class.
\par
In this manner, we obtain two maps from the space of quasi-Fuchsian homomorphisms into a space of Teichm\"uller data of real dimension $(6\mathfrak{g}-6)$ (c.f. Table~\ref{table hyp 1}). Bers' theorem tells us that these maps parametrise the space of quasi-Fuchsian homomorphisms up to conjugation.
\begin{table}[h!]
\begin{center}
\begin{tabular}{|c|c|c|}
\hline
\bf Map & \bf Description  &  \bf Codomain \\
\hline
$\opH^{l/r}$& \multicolumn{1}{|l|}{The holomorphic structure to the left/right of $\Gamma_\theta$}& $\Thol$\\
\hline
\end{tabular}
\end{center}
\caption{Maps taking values in spaces of real dimension $(6\mathfrak{g}-6)$.}\label{table hyp 1}
\end{table}
\begin{theorem}[Bers \cite{Ber60}]
\noindent $(\opH^l,\opH^r)$ defines a bijection between the space of conjugacy classes of quasi-Fuchsian homomorphisms of $\Pi_1$ into $\opPSO(3,1)$ and $\Thol\times\Thol$.
\end{theorem}
An equivariant immersion, $[e,\theta]$, in $\opdS^3$ is now said to be {\sl quasi-Fuchsian} whenever it is embedded and its holonomy, $\theta$, is quasi-Fuchsian. The theory of quasi-Fuchsian equivariant immersions is now developed in the same way as before. Significantly, not all equivariant immersions with quasi-Fuchsian holonomy are themselves actually quasi-Fuchsian. Indeed, given a quasi-Fuchsian equivariant immersion, $[e,\theta]$, a large family of non-embedded equivariant immersions with the same holonomy can be constructed by so called grafting operations, although these are essentially the only ones (c.f. \cite{Gol87}).
\par
As in the preceding two sections, the converse problem of recovering the quasi-Fuchsian equivariant immersion from the holonomy is studied in terms of ghmc de Sitter spacetimes. We first require a good notion of convex subsets for de Sitter space. In the present context, it will be useful to say that a subset, $X$, of $\opdS^3$ is {\sl convex} whenever it coincides with the intersection of $\opdS^3$ with some convex subset, $\hat{X}$, of $\opRP^3$. Likewise, duality of convex subsets of $\opRP^3$ is defined with reference to the Minkowski metric, $\langle\cdot,\cdot\rangle_{3,1}$, so that, given a homogeneous convex cone, $\Lambda$, in $\R^{3,1}$, its dual cone is given by,
$$
\Lambda^* := \left\{ y\ |\ \langle y,x\rangle_{3,1} \leq 0\ \forall x\in\Lambda\right\}~,
$$
and this notion of duality projects to a notion of duality for convex subsets of $\opRP^3$.
\par
We now have
\begin{theorem}\label{max convex QF}
\noindent Given a quasi-Fuchsian homomorphism, $\theta:\Pi_1\rightarrow\opPSO(3,1)$, there exists a unique convex subset, $\Omega_\theta$, of $\opdS^3$, which is maximal with respect to inclusion, over the interior of which $\theta$ acts freely and properly discontinuously.
\end{theorem}
In fact, $\Omega_\theta$ is the intersection of $\opdS^3$ with a convex subset of $\opRP^3$ which contains the closure of $\mathbb{H}^3$. As usual, it is defined via its dual, $\opK_\theta$, which we henceforth refer to as its {\sl Nielsen kernel}, and which is simply the convex hull in $\mathbb{H}^3$ of the invariant Jordan curve, $\Gamma_\theta$. $\Omega_\theta$ itself has two disjoint connected components, each of which intersects $\opC$ along the closure of one of the connected components of the complement of $\Gamma_\theta$. Let $\Omega_\theta^l$ and $\Omega_\theta^r$ be the components lying to the left and right of $\Gamma_\theta$ respectively (c.f. Figure~\ref{fig:desitter}). Each of these components now carries a well defined time orientation for which it is future complete. Observe that reversing the orientation of $\Gamma_\theta$ exchanges $\Omega_\theta^l$ and $\Omega_\theta^r$, so that it is sufficient for much that follows to consider only $\Omega_\theta^l$.
\par
The quotient, $\Omega_\theta^l/\theta(\Pi_1)$, is a  future-complete ghmc dS spacetime, and we say that a ghmc dS spacetime is {\sl quasi-Fuchsian} whenever it can be constructed in this manner. We denote the space of quasi-Fuchsian ghmc dS spacetimes by $\opGHMC_1^\opqf$. By Theorem~\ref{max convex QF}, this space is parametrised by a certain subset of the space of homomorphisms of $\Pi_1$ into $\opPSO(3,1)$, and since the image is an open set (c.f. the Ehresmann--Thurston theorem \cite{gol06}), we use this parametrisation to furnish $\opGHMC_1^{\opqf}$ with the structure of a complex manifold.
\begin{figure}
\begin{center}
\includegraphics[scale=0.9]{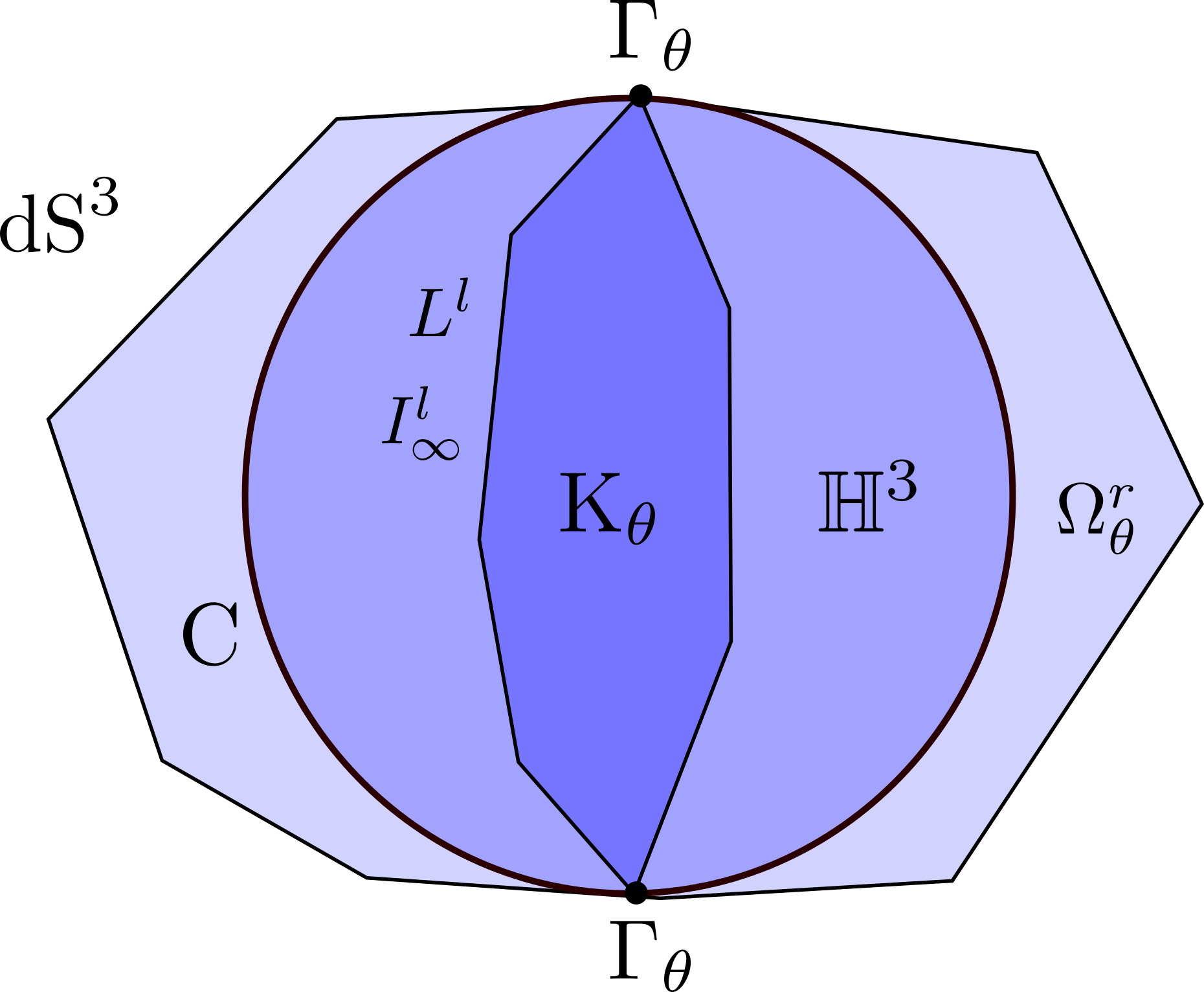}\caption{\textbf{The structure of a ghmc de Sitter spacetime.}\ De Sitter space can be realised in a projective chart as the exterior of the unit ball in $\Bbb{R}^3$ plus a copy of $\opRP^2$ at infinity. The interior of this unit ball identifies with hyperbolic space, and its boundary identifies with the Riemann sphere. $\Omega_\theta$ contains the unit ball, whilst $\opK_\theta$ is contained within it. They both meet the boundary of the unit ball along a Jordan curve, $\Gamma_\theta$. $\opK_\theta$ is, in fact, the convex hull of this curve. The intrinsic metric of $\partial\opK_\theta^\pm$ is hyperbolic, and its singular set is a measured geodesic lamination. The intrinsic metric of $\partial\Omega_\theta$ is a real tree.}\label{fig:desitter}
\end{center}
\end{figure}
\subsection{Laminations and trees}\label{LaminationsAndTrees2}
As usual, measured geodesic laminations are constructed using the boundary of the Nielsen kernel, $\opK_\theta$. In order to visualise this construction, consider first a spacelike, unit vector, $x_1$, in $\R^{3,1}$. Observe that $x_1$ projects to an element of $\opdS^3$. Furthermore, the intersection of its orthogonal complement with $\mathbb{H}^3$ is a totally geodesic embedded submanifold, isometric to $\mathbb{H}^2$. In addition, at every point of this subspace, $x_1$ defines a tangent line to $\mathbb{H}^3$ which is normal to this subspace at this point. By abuse of terminology, we denote this subspace by $x_1^\perp$, and we refer to it as the {\sl plane} orthogonal to $x_1$.
\par
Consider now another spacelike vector, $x_2$, in $\R^{3,1}$, chosen so that $x_1^\perp$ and $x_2^\perp$ intersect along a shared geodesic, $\Gamma_{12}$. Observe that this geodesic divides each of $x_1^\perp$ and $x_2^\perp$ into two half-planes with geodesic boundary. Let $X_{12}$ be one of the four connected components of the complement of the union of $x_1^\perp$ and $x_2^\perp$ in $\Bbb{H}^3$. The boundary of this set is the union of two half-planes, one in $x_1^\perp$, and one in $x_2^\perp$, which meet along $\Gamma_{12}$. Recall now that the intrinsic metric of $\partial X_{12}$ is defined by
$$
d(x,y) := \inf_{\gamma} l(\gamma)~,
$$
where $\gamma$ ranges over all continuous curves in $\partial X_{12}$ starting at $x$ and ending at $y$, and $l(\gamma)$ is its length with respect to the Minkowski metric $\langle\cdot,\cdot\rangle_{3,1}$. It is now a straightforward matter to show that $\partial X_{12}$, furnished with this metric, is isometric to $\mathbb{H}^2$, and that the bending locus, $\Gamma_{12}$, is a complete geodesic.
\par
Consider now a third element, $x_3$, chosen in such a manner that the planes, $x_1^\perp$, $x_2^\perp$ and $x_3^\perp$ have trivial intersection in $\mathbb{H}^3$. The complement of the union of these three planes has at least one connected component, which we will denote by $X_{123}$, which meets each one of the planes, $x_1^\perp$, $x_2^\perp$ and $x_3^\perp$. As before, its boundary, $\partial X_{123}$, furnished with the intrinsic metric, is isometric to $\mathbb{H}^2$, and two of the three curves, $\Gamma_{12}$, $\Gamma_{13}$ and $\Gamma_{23}$, will define complete, non-intersecting geodesics in this space.
\par
Now, since $\opK_\theta$ is a convex hull, it behaves much like the intersection of a finite configuration of spacelike planes, no three of which share a common point in $\mathbb{H}^3$. In particular, the boundary component, $\partial \opK_\theta^l$, furnished with the intrinsic metric, is isometric to $\mathbb{H}^2$, and since it is invariant under the action of $\theta(\Pi_1)$, it defines a compact, hyperbolic surface, $\partial \opK_\theta^l/\theta(\Pi_1)$.
\par
As before, the measured geodesic lamination is now constructed using supporting planes. Indeed, given a convex set, $X$, in $\mathbb{H}^3$ and a boundary point, $x$, a plane, $y^\perp$, passing through $x$ is said to be a {\sl supporting plane} to $X$ at that point whenever $X$ lies entirely to one side of it. Since $\opK_\theta$ is a convex hull, any supporting plane to $\partial\opK_\theta^l$ meets this set either along a complete geodesic or along a non-trivial ideal polygon (with possibly infinitely many sides). The lamination, $L^l$, of $\partial\opK_\theta^l$, is now defined to be the union of all complete geodesics determined by intersections of this set with supporting hyperplanes. In analogy to the case of a finite configuration of geodesics, no two geodesics in $L^l$ intersect, so that this set is indeed a lamination.
\par
Given a convex subset, $X$, of $\H^3$, and a boundary point, $x$, an element, $y$, of $\opdS^3$ is said to be a {\sl supporting normal} to $X$ at $x$ whenever the plane, $y^\perp$, is a supporting plane to $X$ at this point. It is a straightforward matter to show that if $z$ is a point of $\partial\opK_\theta^l$ not lying on the lamination $L^l$, then $\partial\opK_\theta^l$ has a unique supporting normal at that point. Now, given a short curve, $c$, in $\partial\opK_\theta^l$, with end points not in $L^l$, its mass is approximated by the length of the shortest spacelike curve in $\opdS^3$ joining the respective supporting normals of its two end points, and the mass of an arbitrary curve, $c$, compatible with $L$ is now determined in the usual manner by summing over short segments and taking a limit. This defines a measured geodesic lamination, $\lambda$, over $\partial K^l_\theta$. Since $\lambda$ is invariant under the action of $\theta(\Pi_1)$, it projects to a measured geodesic lamination over the hyperbolic surface, $\partial\opK_\theta^l/\theta(\Pi_1)$. Likewise, as in the Minkowski and anti de Sitter cases, the minimal, short action of $\theta(\Pi_1)$ on a real tree is determined by the intrinsic metric of the boundary component, $\partial\Omega_\theta^l$ of $\Omega_\theta$. In summary, bearing in mind that the same constructions also apply to the right hand sides, $\partial\opK_\theta^r$ and $\partial\Omega_\theta^r$, of $\opK_\theta$ and $\Omega_\theta$, respectively, we obtain three pairs of maps, each taking values in spaces of Teichm\"uller data of real dimension $(6\mathfrak{g}-6)$ (c.f. Table~\ref{ds table 2}).
\begin{table}[h!]
\begin{center}
\begin{tabular}{|c|c|c|}
\hline
\bf Map & \bf Description  &  \bf Codomain \\
\hline
$\opI_\infty^{l/r}$&\multicolumn{1}{|l|}{ The intrinsic metric of $\partial\opK_\theta^{l/r}$} & $\Thyp$\\
$\opL^{l/r}$& \multicolumn{1}{|l|}{The measured geodesic lamination of $\partial\opK_\theta^{l/r}$} & $\opML$\\
$\opT^{l/r}$& \multicolumn{1}{|l|}{The minimal, short action of $\theta$ on the real tree $\partial\Omega_\theta^{l/r}$} & $\opRT$\\
\hline
\end{tabular}
\end{center}
\caption{Maps taking values in spaces of real dimension $(6\mathfrak{g}-6)$.}\label{ds table 2}
\end{table}
\par
The relationship between $\opI_\infty^{l/r}$, $\opL^{l/r}$ and $\opH^{l/r}$ can be illustrated via the {\sl generalised Gauss map}, $N$, defined as follows. Consider a convex set, $X$, in $\mathbb{H}^3$. Let $x$ be a boundary point of $X$, and let $y\in\opdS^3$ be a supporting normal to $X$ at this point. Let $\nu_X(x,y)$ be the end-point in $\opC$ of the geodesic leaving $X$ at $x$ in the direction of $y$. Observe that $\nu_X(x,y)$ can also be defined in the following manner, more compatible with our projective viewpoint. Let $P:=\langle x,y\rangle$ be the linear plane in $\mathbb{R}^{3,1}$ generated by $x$ and $y$. The intersection of $P$ with the light cone, $\hat{\operatorname{C}}$, defines two distinct lines, which project to distinct points, $z^-$ and $z^+$, in $\operatorname{C}$. Without loss of generality, $z^+$ lies on the opposite side of the plane, $y^\perp$, as $X$, and we set $\nu_X(x,y):=z^+$. The generalised Gauss map of $X$ is now defined by
\begin{equation*}
N(x) := \left\{ \nu_X(x,y)\ |\ y \mbox{ a supporting normal to } X \mbox{  at } x \right\}~.
\end{equation*}
\par
In order to understand the geometry of $N$, consider first an element, $y_1$, in $\opdS^3$. Observe that $y_1^\perp$ intersects $\opC$ along a circle. Let $D_1$ be one of the connected components of the complement of this circle. Let $X_1$ be the connected component of the complement of $y_1^\perp$ in $\Bbb{H}^3$ which lies on the other side of $y_1^\perp$ from $D_1$, and let $N$ be the generalised Gauss map of $X_1$. Since $\partial X_1$ has a unique supporting normal at every point, $x$, the map, $N$, defines a diffeomorphism from $\partial X_1$ into $D_1$, which is, in fact, holomorphic.
\par
Now consider another element, $y_2$, of $\opdS^3$, such that $y_2^\perp$ meets $y_1^\perp$ along a shared geodesic, $\Gamma$. With $D_2$ and $X_2$ defined as before, denote $X_{12}:=X_1\cap X_2$, and let $N$ be its generalised Gauss map. Since $\partial X_{12}$ has a unique supporting normal at every point, $x$, not lying on $\Gamma$, $N$ defines a holomorphic diffeomorphism from the complement of $\Gamma$ in $\partial X_{12}$ into the symmetric difference, $D_1\Delta D_2$, of $D_1$ and $D_2$. However, at every point, $x$, of $\Gamma$, $\partial X_{12}$ has an entire continuum of supporting normals. The set, $N(x)$, is a circular arc in $\opC$. Furthermore these circular arcs trace out the intersection, $D_1\cap D_2$, as $x$ moves along $\Gamma$. In this way, we see that $N$ identifies with the grafting map sending the hyperbolic structure of $\partial X_{12}$ into the conformal structure of $D_1\cup D_2$.
\par
More generally, let $\opH^{l/r}/\theta(\Pi_1)$. Since $\opK_\theta$ behaves like the intersection of a finite configuration of half spaces in $\Bbb{H}^3$, we now see that the generalised Gauss map, $N$, of the boundary component, $\partial\opK_\theta^{l/r}$, defines the grafting map along the measured geodesic lamination, $L_\theta^{l/r}$, which sends the marked hyperbolic metric, $I_\infty^{l/r}$, into the marked holomorphic structure, $H_\theta^{l/r}$.
\par
We now obtain scattering type results for $\opdS^3$. First, we have
\begin{theorem}
\noindent The map $(\opI_{\infty}^l,\opI_{\infty}^r)$ defines a surjection from $\opGHMC_1^\opqf$ onto $\Thyp\times\Thyp$.
\end{theorem}
\begin{remark} This theorem is considered as part of the folklore by experts in the field. In \cite{EM1}, Epstein \& Marden indicate that it follows from a theorem of Sullivan. The only complete proof that we are aware of is provided by Slutskiy in \cite{slu}, following the work, \cite{Lab92}, of Labourie.\end{remark}
\noindent Consider now the map $(\operatorname{L}^l,\operatorname{L}^r)$. For $r>0$, let $\opML_r\subseteq\opML$ be the set of those measured geodesic laminations all of whose compact leaves have weight strictly less than $r$. Surprisingly, this is not an open condition over $\opML$, and it is not known whether $\opML_r$ is itself an open set (c.f. \cite{BO04}). However, it is straightforward to show that $(\operatorname{L}^l,\operatorname{L}^r)$ takes values in $\opML_\pi\times_{\text{fill}}\opML_\pi$. Now let $\opFuc_1$ denote the space of Fuchsian ghmc dS spacetimes. This is the set of all spacetimes in $\opGHMC_1$ whose holonomy preserves a circle in $C$, or, equivalently, whose holonomy is conjugate to a homomorphism taking values in $\opPSL(2,\Bbb{R})$.
\begin{theorem}[Bonahon--Otal \cite{BO04}]
\noindent The map $(\opL^l,\opL^r)$ defines a surjection from $\opGHMC_1^\opqf\setminus\opFuc_1$ onto $\opML_\pi\times_{\operatorname{fill}}\opML_\pi$.
%
\end{theorem}
\begin{remark} It has been conjectured by Thurston that this map is injective. In particular, by invariance of the domain, this would imply openness of the set $\opML_\pi\times_{\operatorname{fill}}\opML_\pi$.\end{remark}
Finally, we recall the following two scattering type results of Lecuire.
\begin{theorem}[Lecuire \cite{Lec06}]
\noindent The maps $(\opI_\infty^l,\opL^r)$ and $(\opI_\infty^r,\opL^l)$ defines surjections from $\opGHMC_1^\opqf$ onto $\Thyp\times\opML_\pi$.
\end{theorem}
\begin{theorem}[Lecuire \cite{Lec06}]
\noindent The maps $(\opH^l,\opL^r)$ and $(\opH^r,\opL^l)$ define surjections from $\opGHMC_1^\opqf$ onto $\Thol\times\opML_\pi$.
\end{theorem}
\subsection{Smooth parametrisations}\label{SmoothParametrisations2}
Our starting point for the construction of smooth parametrisations of $\opGHMC_{-1}$ is the following result.
\begin{theorem}[Labourie \cite{Lab91}]
\noindent Let $\theta:\Pi_1\rightarrow\opPSO(3,1)$ be quasi-Fuchsian. For all $\kappa\in]1,\infty[$, and for each $\alpha\in\left\{l,r\right\}$, there exists a unique, smooth, spacelike, LSC surface, $\Sigma^\alpha_\kappa$ which is embedded in $\Omega_\theta^\alpha$, which is invariant under the action of $\theta$, and which has constant extrinsic curvature equal to $\kappa$. Furthermore, the family of all such surfaces foliates $\partial\Omega_\theta$ as $\kappa$ varies over the interval $]1,\infty[$.
\end{theorem}
Maps taking values in spaces of Teichm\"uller data are constructed as follows. For $\kappa\in]1,\infty[$ and $\alpha\in\left\{l,r\right\}$, consider the spacelike, LSC, embedded surface, $\Sigma_\kappa^\alpha$, in $\Omega_\theta^\alpha$ which is invariant with respect to $\theta$ and which has constant extrinsic curvature equal to $\kappa$. Let $I_\kappa^\alpha$, $\II_\kappa^\alpha$ and $\III_\kappa^\alpha$ be its first, second and third fundamental forms respectively. The forms, $(\kappa-1)^{-1}I_\kappa^\alpha$ and $\kappa^{-1}(\kappa-1)\III_\kappa^\alpha$, each define marked hyperbolic metrics over $\Sigma_\kappa^\alpha$, thus defining points in $\Thyp$. Next, by convexity, $\II_\kappa^\alpha$ is also positive definite, and therefore also defines a metric over $\Sigma_\kappa^\alpha$, but since this metric has no clear curvature properties, we consider it rather as defining a point in $\Thol$. In summary, we have three pairs of maps, each taking values in spaces of Teichm\"uller data of real dimension $(6\mathfrak{g}-6)$ (c.f. Table~\ref{table ds 3}).
\begin{table}[h!]
\begin{center}
\begin{tabular}{|c|c|c|}
\hline
\bf Map & \bf Description  &  \bf Codomain \\
\hline
$\opI_\kappa^{l/r}$& \multicolumn{1}{|l|}{The first fundamental form of $\Sigma_\kappa^{l/r}$} & $\Thyp$\\
$\opII_\kappa^{l/r}$& \multicolumn{1}{|l|}{The second fundamental form of $\Sigma_\kappa^{l/r}$} & $\Thol$\\
$\opIII_\kappa^{l/r}$& \multicolumn{1}{|l|}{The third fundamental form of $\Sigma_\kappa^{l/r}$} & $\Thyp$\\
\hline
\end{tabular}
\end{center}
\caption{Maps taking values in spaces of real dimension $(6\mathfrak{g}-6)$.}\label{table ds 3}
\end{table}
\par
These maps are complemented to maps taking values in spaces of Teichm\"uller data of real dimension $(12\mathfrak{g}-12)$ as follows. First, the shape operator, $A_\kappa^\alpha$, of $\Sigma_\kappa^\alpha$ defines, up to a constant factor, a Labourie field of $I_\kappa^\alpha$, whilst its inverse $(A_\kappa^\alpha)^{-1}$, defines a Labourie field of $\III_\kappa^\alpha$, so that the pairs, $(I_\kappa^\alpha,\opA_\kappa^\alpha)$ and $(\III_\kappa^\alpha,(\opA_\kappa^\alpha)^{-1})$, define points of $\opLab\Thyp$. Likewise, the Hopf differential, $\phi^\alpha_\kappa$, of $I_\kappa^\alpha$ with respect to the conformal structure of $\II_\kappa^\alpha$ defines a holomorphic quadratic differential, so that the pair $(\II_\kappa^\alpha,\phi_\kappa^\alpha)$ yields a point of $\opT^*\Thol$. We thus have three pairs of maps taking values in spaces of Teichm\"uller data of real dimension $(12\mathfrak{g}-12)$ (c.f. Table~\ref{hyp table 4}).
\begin{table}[h!]
\begin{center}
\begin{tabular}{|c|c|c|}
\hline
\bf Map & \bf Description  &  \bf Codomain \\
\hline
$\opA_{\opI,\kappa}^{l/r}$ & \multicolumn{1}{|l|}{\pbox{20cm}{The first fundamental form of $\Sigma_\kappa^{l/r}$ \\ together with the Labourie field, $A_\kappa^{l/r}$}} & $\opLab\Thyp$\\
$\Phi_\kappa^{l/r}$& \multicolumn{1}{|l|}{\pbox{20cm}{The second fundamental form of $\Sigma_\kappa^{l/r}$ \\ together with the Hopf differential $\phi_\kappa^{l/r}$} }& $\opT^*\Thol$ \\
$\opA_{\opIII,\kappa}^{l/r}$&\multicolumn{1}{|l|}{\pbox{20cm}{The third fundamental form of $\Sigma_\kappa^{l/r}$\\ together with the Labourie field, $A_\kappa^{l/r}$}}& $\opLab\Thyp$  \\
\hline
\end{tabular}
\end{center}
\caption{Maps taking values in spaces of real dimension $(12\mathfrak{g}-12)$.}\label{hyp table 4}
\end{table}
\par
It follows from existing results that these maps parametrise $\opGHMC_1^\opqf$. Indeed,
\begin{theorem}\label{AIsRealAnalDiffDSCase}
\noindent For each $\alpha\in\left\{l,r\right\}$, the maps $\opA_{\opI,\kappa}^\alpha$ and $\opA_{\opIII,\kappa}^\alpha$ define real analytic diffeomorphisms from $\opGHMC_1^\opqf$ onto open subsets of $\opLab\Thyp$ containing the zero section.
\end{theorem}
\begin{remark} We are not aware of any explicit descriptions of the images of these maps.\end{remark}
\begin{proof}[Sketch of proof] Consider a hyperbolic metric, $g$, a Labourie field, $A$, and a real number, $\kappa>1$. By the fundamental theorem of surface theory (Theorem~\ref{FTS}), there exists an LSC equivariant immersion $e:(\tilde{S},(\kappa-1)^{-1}g)\rightarrow\opdS^3$, with shape operator equal to $\sqrt{\kappa}A$, which is unique up to isometries of $\opdS^3$. This yields a real analytic inverse of $\opA_{\opI,\kappa}^\alpha$. The real analytic inverse of $\opA_{\opIII,\kappa}^\alpha$ is constructed in a similar manner using equivariant immersions into $\Bbb{H}^3$, and the result follows.\end{proof}
\begin{theorem}
\noindent For each $\alpha\in\left\{l,r\right\}$, the map $\Phi_\kappa^\alpha$ defines a real analytic diffeomorphism from $\opGHMC_1^\opqf$ onto an open subset of $\opT^*\Thol$.
\end{theorem}
\begin{remark} We are not aware of any explicit description of the image of this map.\end{remark}
\begin{proof} Indeed, observe that $\Phi_\kappa^\alpha\circ(\opA_{\opI,\kappa}^\alpha)^{-1}$ coincides with $\Phi\circ\mathcal{A}^{-1}$, where $\Phi$ and $\mathcal{A}$ are defines as in Sections \ref{HopfDifferentials} and \ref{LabourieFields} respectively. The result now follows by Theorems~\ref{wolf}, \ref{AIsARealAnalDiff} and \ref{AIsRealAnalDiffDSCase}.\end{proof}
Finally, considering pairs of maps taking values in spaces of Teichm\"uller data of real dimension $(6\mathfrak{g}-6)$, we have the following two scattering type results.
\begin{theorem}[Schlenker \cite{sch06}]
\noindent For all $\kappa>0$, the map, $(\opI_\kappa^l,\opI_\kappa^r)$, defines a bijection from $\opGHMC_1^\opqf$ into $\Thyp\times\Thyp$.
\end{theorem}
\begin{theorem}[Labourie \cite{Lab92}]
\noindent For all $\kappa>0$, the map, $(\opIII_\kappa^l,\opIII_\kappa^r)$, defines a bijection from $\opGHMC_1^\opqf$ into $\Thyp\times\Thyp$.
\end{theorem}

%


\bibliographystyle{alpha}
\bibliography{survey-thm}
\end{document}